\documentclass[11pt]{article}

\usepackage[utf8]{inputenc}
\usepackage{a4wide}
\usepackage[backend=bibtex, sortcites=true, sorting=nty]{biblatex}
\usepackage{amsmath}
\usepackage{amsfonts,amssymb,amsthm}
\usepackage{mathtools}
\usepackage{enumerate}
\usepackage{nicefrac}
\usepackage{graphicx}
\usepackage{bbm}
\usepackage{esint}
\usepackage{caption}
\usepackage{subcaption}
\usepackage{stmaryrd}
\usepackage[labelformat = brace, position=top]{subcaption}
\usepackage{hyperref}
\usepackage[normalem]{ulem}
\usepackage{authblk}

\usepackage{tikz}
\usetikzlibrary{calc,intersections,through,backgrounds,patterns,arrows.meta,hobby}

\newcounter{statement} 
\numberwithin{statement}{section}

\numberwithin{equation}{section}

\newtheorem{Def}[statement]{Definition}

\newtheorem{thm}[statement]{Theorem}
\newtheorem{lemma}[statement]{Lemma}
\newtheorem{cor}[statement]{Corollary}
\newtheorem{prop}[statement]{Proposition}
\newtheorem{conj}[statement]{Conjecture}

\theoremstyle{remark}

\newcommand{\R}{\mathbb R} 				
\newcommand{\N}{\mathbb N}

\newcommand{\intd}{\, \mathrm{d}} 				
\newcommand{\warr}{\rightharpoonup}				
\newcommand{\Sph}{{\mathbb S}}				
\newcommand{\Hd}{\mathcal H}				

\newcommand{\stcomp}[1]{{#1}^{\mathsf{c}}}		

\newcommand{\eps}{\varepsilon}

\newcommand{\interior}[1]{{\kern0pt#1}^{\mathrm{o}}	
}

\def\XXint#1#2#3{{\setbox0=\hbox{$#1{#2#3}{\int}$ }
\vcenter{\hbox{$#2#3$ }}\kern-.58\wd0}}

\renewbibmacro{in:}{%
  \ifentrytype{article}{}{\printtext{\bibstring{in}\intitlepunct}}}
  
\bibliography{screened_gamow.bib}

\title{\bf Euler's elastica functional as a large mass limit of a
  two-dimensional non-local isoperimetric problem}

\author[1]{Cyrill B. Muratov}

\author[1]{Matteo Novaga}

\author[2]{Theresa M. Simon}

\affil[1]{Department of Mathematics,
    University of Pisa, Largo B. Pontecorvo 5, 56127 Pisa, Italy.}

  \affil[2]{Institut f{\"u}r Analysis und Numerik, Universit{\"a}t
    M{\"u}nster, 48149 M{\"u}nster, Germany.}

\begin{document}
\maketitle
\begin{abstract}
  We consider a large mass limit of the non-local isoperimetric
  problem with a repulsive Yukawa potential in two space dimensions.
  In this limit, the non-local term concentrates on the boundary,
  resulting in the existence of a critical regime in which the
  perimeter and the non-local terms cancel each other out to leading
  order.  We show that under appropriate scaling assumptions the
  next-order $\Gamma$-limit of the energy with respect to the $L^1$
  convergence of the rescaled sets is given by a weighted sum of the
  perimeter and Euler's elastica functional, where the latter is
  understood via the lower-semicontinuous relaxation and is evaluated
  on the system of boundary curves. As a consequence, we prove that in
  the considered regime the energy minimizers always exist and
  converge to either disks or annuli, depending on the relative
  strength of the elastica term.
\end{abstract}


\tableofcontents

\section{Introduction}

In many physical systems the onset of spatial pattern formation is
driven by a competition of short-range attractive and long-range
repulsive foces \cite{seul95,m:pre02,giuliani09a}. In binary systems,
this is often captured by a prototypical model in which the
short-range attractive interactions between the two phases are modeled
by an interfacial energy term, while the long-range repulsion is due
to a two-body interaction through a positive kernel:
\begin{align}\label{eq:nonlocal_iso_problem}
  E(\Omega) =  P(\Omega) + \frac12 \int_{\Omega}\int_{\Omega}
  K(x-y) \intd^n y \intd^n x.
\end{align}
Here $\Omega \subset \R^n$, with $n\geq2$, is the spatial domain
occupied by the minority phase whose ``mass'' $|\Omega| = m > 0$ is
fixed, $P(\Omega)$ is the perimeter \cite{ambrosio} of $\Omega$,
and $K: \mathbb R^n \to [0,\infty)$ is a suitable kernel. A typical
example is given by the Coulombic kernel $K(x) = {1 \over 4 \pi |x|}$
in three space dimensions, giving rise to the celebrated Gamow's model
of the atomic nucleus \cite{Gamow1930mass,cmt:nams17}.  However,
many other kernels may be considered, notably the regularized dipolar
kernel $K(x) \sim {1 \over |x|^3}$ in two dimensions that arises in
the context of magnetic domains, ferrofluids and Langmuir monolayers
\cite{seul95,andelman09, andelman85,mcconnell88,mcconnel91}. Yet
another variant is obtained by considering a Yukawa potential
$K = K_\alpha$ in the plane:
\begin{align}
  \label{Kalpha}
  K_\alpha(x) = {e^{-\alpha |x|} \over 2 \pi |x|} \qquad x \in \mathbb R^2,
\end{align}
where $\alpha > 0$ is a screening parameter, which naturally arises in
the studies of Langmuir monolayers in the presence of weak ionic
solutions (see Appendix \ref{sec:model-derivation}).

The behavior of the minimizers of the non-local isoperimetric problem
governed by \eqref{eq:nonlocal_iso_problem} depends rather crucially
on the rate of decay of the kernel $K(x)$ as $|x| \to \infty$. In
particular, there is a notable difference for large masses: In the
case of the three-dimensional Coulombic kernel and two- or
three-dimensional domains $\Omega$, minimizers fail to exist beyond a
certain critical mass
\cite{knupfer2013isoperimetric,muratov2014isoperimetric,
  lu2014nonexistence,frank2016nonexistence}, while for a {\em
  screened} three-dimensional Coulombic kernel represented by the
Yukawa potential the minimizers do exist for all $\alpha > \alpha_c$
for some explicit $\alpha_c = \alpha_c(n) > 0$, provided that
$m \geq m_c$ for some $m_c = m_c(\alpha, n) > 0$, see Pegon
\cite{pegon2021large}. Moreover, in two dimensions the minimizers for
sufficiently large values of $\alpha$ and all large enough masses are
known to be disks \cite{merlet2022large}, something that in the
absence of screening ($\alpha = 0$) is known to occur at small masses
$m \ll 1$ instead
\cite{knupfer2013isoperimetric,muratov2014isoperimetric,
  julin2014isoperimetric,figalli2015isoperimetry}. Thus, one can
imagine that for a given $m \gg 1$ a transition occurs at some
threshold value of $\alpha > 0$ that may lead to the onset of
minimizers which are no longer necessarily disks as the value of
$\alpha$ is lowered.

Our work attempts to look into the transition that bridges the gap
between the two regimes described above in two space dimensions. We
focus on the parameters for which the non-local term at large masses
cancels the interfacial energy term of the energy
\eqref{eq:nonlocal_iso_problem} to the leading order. It turns out
that to next order in the asymptotic expansion of the energy as
$m \to \infty$, this yields Euler's elastica functional plus a
  term proportional to the perimeter:
\begin{align}
  \label{Eelastica}
  E_0(\Omega) =  \int_{\partial \Omega}\left(
  \sigma + \frac{\pi}{2} \kappa^2\right) \intd \Hd^1.
\end{align}
Here, $\kappa$ is the curvature of $\partial \Omega$ and
$\sigma > 0$. More precisely, we will show that a relaxed version
of the energy in \eqref{Eelastica} can be obtained as the
$\Gamma$-limit of a suitably rescaled energy in
\eqref{eq:nonlocal_iso_problem} with the kernel from \eqref{Kalpha},
as the value of $\alpha$ approaches the critical value
$\alpha_c = {1 \over \sqrt{2 \pi}}$ with the right rate (see the
following section for the precise statement). As a consequence, we can
conclude that the minimizers of the energy in
\eqref{eq:nonlocal_iso_problem} in the considered limit change from
disks to annuli as the parameter of the asymptotic expansion is
varied. Note that annular domains are frequently observed in the
experiments on lipid monolayers \cite{mcconnell90}.

Similar regimes may be studied when instead of screened Coulombic
repulsion one considers regularized and renormalized dipolar repulsion
for $n=2$.  In this setting, Muratov and Simon proved that in regimes
where perimeter asymptotically still carries a cost, minimizers are
disks even for finite regularization lengths
\cite{muratov2018nonlocal}.  They also identified the next-order limit
in the case of vanishing cost of the perimeter, which by a result of
Cesaroni and Novaga \cite{cesaroni2022second} coincides with the
second-order expansion of the fractional perimeters close to the local
one.  Muratov and Simon also proved existence of non-spherical
minimizers for a modified, yet still isotropic kernel
\cite{muratov2018nonlocal}.  Closely related results were obtained for
a class of general kernels in the regime of large mass by Pegon
\cite{pegon2021large}, Merlet and Pegon \cite{merlet2022large}, and
Goldman, Merlet and Pegon \cite{goldman2022uniform}, as well as by
Kn{\"u}pfer and Shi \cite{knuepfer2022second} in the case of a torus.

We note that Euler's elastica energy is a classical problem in the
calculus of variations, which was first analyzed by Euler in 1744 for
$\sigma=0$, after Daniel Bernoulli proposed the energy to him in a
letter \cite{euler1774methodus}.  While the original motivation was to
study thin elastic rods, it has since also appeared in image
segmentation problems, see for example Mumford
\cite{mumford1994elastica}.  Its higher-dimensional analog, the
Willmore energy, which asks to minimize the $L^2$-norm of the mean
curvature of a hypersurface and, more generally, the Helfrich energy,
appear in a variety of fields from differential geometry to the
modeling of cell membranes in biology, see for example Willmore
\cite{willmore1992survey} and Helfrich \cite{helfrich1973elastic}.  We
will require the elastica energy in its relaxed form (with respect to
the $L^1$ topology of the enclosed sets).  It has been characterized
by Bellettini and Mugnai
\cite{belletini2004characterization,bellettini2007varifolds}, see also
Bellettini, Dal Maso, and Paolini \cite{bellettini1993semicontinuity}.
Its minimizers have been identified by Goldman, Novaga, and R{\"o}ger
\cite{goldman2020quantitative}, even after augmentation by a non-local
term as in \eqref{eq:nonlocal_iso_problem}.  To the best of our
knowledge, they were also the first to include curvature-depending
terms in the context of non-local isoperimetric problems.

Finally, we remark on results regarding the passage from first-order
variational problems to second-order problems.  The most prominent
body of literature certainly pertains to the rigorous derivation of
bending energies from non-linear elasticity with its many
contributions being thoroughly outside the scope of this introduction.
We thus only mention the seminal paper by Friesecke, James, and
M{\"u}ller \cite{friesecke2002theorem}, which serves as the foundation
for virtually all contributions following it.  Indeed, it is also
where our argument takes part of its inspiration. On the other
  hand, the question of this type was posed by De Giorgi in the
  context of phase field models of phase transitions
  \cite{degiorgi91}. While the original conjecture from
  \cite{degiorgi91} was shown not to lead to an energy of the form of
  \eqref{Eelastica} \cite{bellettini23} (compare with \cite{roger06}),
  a natural alternative would be provided by the diffuse interface
  version of the energy in \eqref{eq:nonlocal_iso_problem} in two
  space dimensions:
  \begin{align}
    \label{eq:diffuse}
    E(u) = \int_{\R^2} \left( \frac12 |\nabla u|^2 + \frac{9}{32} (1 -
    u^2)^2 \right) \intd^2 x + \frac12 \int_{\R^2} \int_{\R^2}
    K_\alpha(x - y) u(x) u(y) \intd^2 x  \intd^2 y,
  \end{align}
  with the kernel $K_\alpha$ from \eqref{Kalpha} and the mass
  constraint
  \begin{align}
    \int_{\R^2} u \intd^2 x = m.
  \end{align}
  Here the choice of the double-well potential ensures that the
  surface energy associated with the optimal transition layer
  connecting $u = 0$ and $u = 1$ is equal to unity, hence, yielding
  the perimeter functional as the $\Gamma$-limit of the first term in
  \eqref{eq:diffuse} in the limit $m \to \infty$ after rescaling
  lengths by $m^{1/2}$ \cite{modica87}. We thus would expect that the
  limit behavior of the energy in \eqref{eq:diffuse} would be the same
  as that of \eqref{eq:nonlocal_iso_problem}, yielding an example of a
  second-order variational problem arising from phase field models of
  phase transitions.

  This paper is organized as follows. In Sec.\ \ref{sec:main}, we give
  the precise formulation of the problem under consideration and its
  limit, and present the precise statement of the obtained results,
  followed by an outline of the proof. In Sec.\ \ref{sec:prelim}, we
  prove existence of minimizers in the considered regime and derive
  the representation of the non-local energy term used throughout the
  rest of the paper. Then, in Sec.\ \ref{sec:compact} we establish
  compactness of boundary curves in the considered limit and in Sec.\
  \ref{sec:Gamma} we prove $\Gamma$-convergence.  We also provide the
  details of model derivation in the appendix.

\paragraph{Acknowledgments.} The first two authors are members of
INdAM-GNAMPA, and acknowledge partial support by the MUR
Excellence Department Project awarded to the Department of
Mathematics, University of Pisa, CUP I57G22000700001, by the PRIN 2022
Project P2022E9CF89 and by the PRIN 2022 PNRR Project P2022WJW9H.  The
last author is funded by the Deutsche Forschungsgemeinschaft (DFG,
German Research Foundation) under Germany's Excellence Strategy EXC
2044 –390685587, Mathematics Münster: Dynamics–Geometry–Structure.

\section{Main results}
\label{sec:main}

For the screening parameter $\alpha>0$ and mass $m>0$,
we study the non-local isoperimetric problem whose kernel is
given by the Yukawa potential \eqref{Kalpha}. Up to a mass-dependent
additive constant, the energy \eqref{eq:nonlocal_iso_problem} is then
given by
\begin{align}\label{energy}
  E_{\alpha}(\Omega):= P(\Omega) -
  \frac{1}{4\pi}\int_\Omega\int_{\stcomp{\Omega}}
  \frac{e^{-\alpha|x-y|}}{|x-y|} \intd^2 y \intd^2 x ,
\end{align}
on the admissible class
\begin{align}\label{admissible_class}
  \mathcal{A}_m := \{\Omega \subset \R^2: \ |\Omega| = m, \
  P(\Omega) <\infty\}. 
\end{align}
A direct calculation shows that for $\lambda >0$ we have
	\begin{align*}
          E_{\alpha}(\lambda\Omega) = \lambda \left( P(\Omega) -
          \frac{\lambda^2}{4\pi}\int_\Omega\int_{\stcomp{\Omega}}
          \frac{e^{-\lambda \alpha |x-y|}}{|x-y|} \intd^2 y \intd^2 x
          \right). 
	\end{align*}
We may therefore instead analyze the energy
\begin{align}\label{energy_large_mass}
  F_{\lambda , \alpha}(\Omega):= P(\Omega) - \frac{\lambda
  ^2}{4\pi}\int_\Omega\int_{\stcomp{\Omega}} \frac{e^{-\lambda
  \alpha |x-y|}}{|x-y|} \intd^2 y \intd^2 x 
\end{align}
on the admissible class $\mathcal{A}_\pi$, where $m = \lambda^2\pi$
gives the relation between $\lambda$ and $m$. In particular, studying
the $\lambda \to \infty$ limit of the energy in
\eqref{energy_large_mass} over the admissible class $\mathcal A_\pi$
is equivalent to studying the limit of $m \to \infty$ of the energy in
\eqref{energy} over the admissible class $\mathcal A_m$.

The existence of a subcritical regime of screening parameters for this
energy is already established by the general results of P{\'e}gon and
collaborators
\cite{pegon2021large,merlet2022large,goldman2022uniform}. Indeed,
their result gives the following:
\begin{thm}[Merlet, P{\'e}gon \cite{merlet2022large}]
  For $\alpha > \frac{1}{\sqrt{2\pi}}$, the $L^1$ $\Gamma$-limit of
  $F_{\lambda_n, \alpha}$ as $\lambda_n \to \infty$ is given
  by
  \begin{align}\label{eq:gamma_limit_subcritical}
    G_\alpha(\Omega) = \left(1-\frac{1}{2\pi\alpha^2} \right) P(\Omega).
  \end{align}
  Furthermore, there exists $\lambda_\alpha>0$ such that for all
  $\lambda>\lambda_\alpha$ all minimizers of $F_{\lambda,\alpha}$ are,
  up to translation, given by the disk $B_1(0)$.
  \label{thm:pegon}
\end{thm}
\noindent By an $L^1$ $\Gamma$-limit, we mean the limit with respect
to convergence of the characteristic functions
$\chi_{\Omega_n} \to \chi_\Omega$ in $L^1(\R^2)$ of measurable sets
$\Omega_n \subset \R^2$ to that of the limiting set
$\Omega \subset \R^2$ as $n \to \infty$. Similar results have been
obtained by Muratov and Simon for a non-local isoperimetric
problem with dipolar repulsion \cite{muratov2018nonlocal}.

In this paper, we will instead investigate the large mass behaviour
near the \emph{critical} screening length
$\alpha = \frac{1}{\sqrt{2\pi}}$ via a $\Gamma$-convergence analysis.
First, we note that in this regime minimizers always exist and are
sufficiently regular.

\begin{prop}\label{lem:minimizers}
  Let $\lambda > 0$ and $\alpha > \frac{1}{\sqrt{2\pi}}$. Then a
  minimizer of $F_{\lambda,\alpha}$ over $\mathcal{A}_\pi$ exists.
  Furthermore, all minimizers are bounded, connected, open sets with
  boundary of class $C^{2,\alpha}$ for any $\alpha \in (0,1)$ and have
  finitely many holes.
\end{prop}

As the next step, we observe that for fixed and sufficiently regular
sets the energy has an expansion in terms of the perimeter and the
squared $L^2$-norm of the curvature of the boundary, i.e., the
elastica energy. Throughout the rest of the paper, we call a set
  {\em regular}, if it is a bounded open set with the boundary of
  class $C^\infty$. As the energy $F_{\lambda,\alpha}(\Omega)$ of any
  admissible set $\Omega \in \mathcal A_\pi$ may be approximated by
  that of a regular set, restricting our attention to regular sets
  will suffice for our purposes.

\begin{prop}\label{prop:expansion}
  Let $\Omega$ be a regular set.  Then as $\lambda \to \infty$ we
  have
  \begin{align*}
    F_{\lambda,\alpha}(\Omega) = \left( 1- \frac{1}{2\pi
    \alpha^2}\right) P(\Omega) + \frac{1}{8\pi 
    \alpha^4 \lambda^2}  \int_{\partial \Omega} \kappa^2 \intd
    \Hd^1 + o\left(\lambda^{-2} \right).
  \end{align*}
\end{prop}

We can thus indeed hope to obtain the combination of the perimeter
  and the elastica energy as a large-mass $\Gamma$-limit of the
functionals \eqref{energy_large_mass} in the critical regime
$\alpha = \frac{1}{\sqrt{2\pi}}$.  However, note that the integral of
the curvature squared is ill-behaved on its own, since
\begin{align}\label{eq:annuli_estimate}
  \int_{\partial \left(B_{\sqrt{1+r^2}}(0)\setminus B_{r}(0)\right) }
  \kappa^2 \intd \Hd^1 \sim \frac{1}{r} \to 0 
\end{align}
as $r\to \infty$.  Therefore, we will need to retain control over the
perimeter in order to obtain a reasonable $\Gamma$-limit.  To this
end, we will consider sequences of screening parameters which approach
the critical parameter from above as $\lambda \to \infty$ with an
appropriate rate.

\begin{thm}\label{thm:main}
  Let $\lambda_n \to \infty$ and $\alpha_n >\frac{1}{\sqrt{2\pi}}$ be
  sequences such that
  $\sigma_n := \lambda_n^2 \left( 1- \frac{1}{2\pi \alpha_n^2}
  \right)$ satisfies
	\begin{align}
		\lim_{n\to \infty} \sigma_n = \sigma >0.
	\end{align}
	Then, the $L^1$ $\Gamma$-limit of
        $\lambda_n^2 F_{\lambda_n, \alpha_n}$ as
        $n \to \infty$
        is given by
        \begin{align}
       	 	F_{\infty,\sigma} := \operatorname{rel} \tilde F_{\infty,\sigma},
        \end{align}
        where for a regular set $\Omega \in \mathcal{A}_\pi$ we define
         \begin{align}
           \label{elastica_smooth}
          \tilde F_{\infty,\sigma}(\Omega) := \sigma
          P(\Omega)+ \frac{\pi}{2} \int_{\partial \Omega}
          \kappa^2 \intd \Hd^1,
         \end{align}
         and the relaxation is with respect to the $L^1$-convergence
         of the characteristic functions.
\end{thm}

Since finite energy sequences might break up into multiple pieces
which drift infinitely far apart, we have not included a compactness
statement here.  As minimizers must be connected due to the non-local
kernel being repulsive, we do get convergence of minimizers up to
translations.  The characterization of the minimizers used here is due
to Goldman, Novaga, and R{\"o}ger \cite{goldman2020quantitative}, to
which we also refer for more precise descriptions of the minimizers.

\begin{cor}\label{cor:convergence_minimizers}
  Under the assumptions of Theorem \ref{thm:main}, minimizers
  $\Omega_n \subset \mathcal{A_\pi}$ of $F_{\lambda_n, \alpha_n}$
  exist and after suitable translations and along a subsequence
  converge in the $L^1$-topology to a minimizer $\Omega_\infty$ of
  $F_{\infty, \sigma}$.  In particular, there exists $\bar \sigma >0$
  such that for $\sigma >\bar \sigma$ we have
  $\Omega_\infty = B_1(0)$, while for $\sigma <\bar \sigma$ there
  exists $r_\sigma>0$ such that
  $\Omega_\infty = B_{\sqrt{1+r_\sigma^2}}(0)\setminus B_{r_\sigma}(x)
  $ with $x \in \R^2$ such that
  $|x| \leq \sqrt{1+r_\sigma^2} - r_\sigma$.  For
  $\sigma= \bar\sigma$, both cases may occur, with $r_{\bar\sigma}>0$.
\end{cor}
\noindent The values of $\bar \sigma$ and $r_{\bar \sigma}$ may be
  found explicitly as solutions of an algebraic system of
  equations. Numerically, we have $\bar \sigma \approx 0.112736$ and
  $r_{\bar \sigma} \approx 3.66882$.

The inner ball of $\Omega_{\infty}$ in the case
$\sigma \leq \bar \sigma$ in Corollary
  \ref{cor:convergence_minimizers} need not be concentric with the
outer ball due to locality of $F_{\infty, \sigma}$.  We conjecture
that this will not actually occur in the limits of $\Omega_n$ as
$n\to \infty$.  Indeed, the following is expected to hold:

\begin{conj}
  Under the assumptions of Theorem \ref{thm:main}, there exists
  $\bar n>0$ such that for $n>\bar n$ minimizers $\Omega_n$ of
  $F_{\lambda_n,\sigma_n}$ are given, up to translation, by $B_1(0)$
  if $\sigma > \bar \sigma$ and by
  $B_{\sqrt{1+r_{\sigma,n}^2}}(0)\setminus B_{r_{\sigma,n}}(0) $ for
  some $r_{\sigma,n} >0$ converging to $r_{\sigma}$ as $n \to \infty$
  if $\sigma < \bar \sigma$.
\end{conj}
\noindent This conjecture is supported by the fact that the
  minimum of $F_{\lambda,\alpha}$ among all sets of the form
  $\Omega_x := B_{\sqrt{1 + r^2}} (0) \backslash B_r(x) \in \mathcal
  A_\pi$ with $r > 0$ and
  $x \in \overline {B_{\sqrt{1 + r^2} - r}(0)}$ is
  uniquely attained for $x = 0$ for any $\lambda > 0$ and
  $\alpha > 0$, see Proposition \ref{prop:centered} below.
  Proving this conjecture, however, would require to go to higher
  orders in the expansion of the energy and, in particular, to keep
  track of the exponentially small terms arising from the non-local
  interactions between the inner and the outer boundaries of the
  minimizers, as well as understanding the asymptotic rigidity of
  concentric annuli with respect to the energy
  $F_{\lambda,\alpha}$. Such an analysis goes well beyond the scope of
  the present paper.

  \begin{prop}\label{prop:centered}
    Let $K : (0,\infty) \to (0,\infty)$ be monotone decreasing such
    that $r \mapsto rK(r)$ is integrable.  For $r>0$ and
    $x \in \overline {B_{\sqrt{1 + r^2} - r}(0)}$, let
    $\Omega_x := B_{\sqrt{1 + r^2}} (0) \backslash B_r(x)$.  Then
    $\Omega_0$ minimizes
	\begin{align}
	  \begin{split}
            f(x) & := \int_{\Omega_x} \int_{\Omega_x}  K( |y-z|) \intd^2 y \intd^2 z\\
            & = - \int_{\Omega_x} \int_{\stcomp\Omega_x} K( |y-z|)
            \intd^2 y \intd^2 z + 2\pi^2 \int_0^\infty rK(r) \intd r.
	  \end{split}
	\end{align}
	among all points $x \in \overline {B_{\sqrt{1 + r^2} -
            r}(0)}$.  Additionally, if $K$ is strictly monotone
        decreasing, then $\Omega_0$ is the unique minimizer.
  \end{prop}

  We finally comment on the issue of the relaxation contained in
  $F_{\infty,\sigma}$, which is surprisingly subtle and technically
  challenging. This is due to the $L^1$-topology controlling the set
  while the elastica energy controls the boundary. Passing from one to
  the other is particularly tricky when parts of the boundaries of
  sequences of sets collapse, in the limit resulting in singular
  points of the boundary. The singular set may have accumulation
  points, see Figure \ref{fig:collapse}, and may even have positive
  $\mathcal{H}^1$ measure, see \cite[Example
  1]{bellettini1993semicontinuity}.

	\begin{figure}[t]
		\centering
		\includegraphics{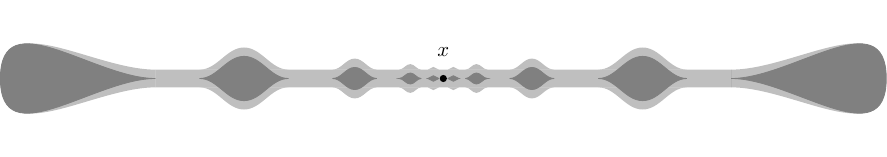}
		\caption{\label{fig:collapse} Sketch of a sequence of
                  regular sets (light gray) with finite elastica
                  energy and its limit (dark gray). The singular part
                  of the boundary of the limit has an accumulation
                  point at $x$.}
	\end{figure}

        For curves in two dimensions, the relaxation has been
        identified by Bellettini and Mugnai
        \cite{belletini2004characterization} building on ideas by
        Bellettini, Dal Maso, and Paolini
        \cite{bellettini1993semicontinuity}, which we present below
        after introducing the necessary notation.  For the also
        physically relevant case of surfaces in three dimensions, the
        relaxation of the Willmore (or Helfrich) energy is not yet
        known. Therefore, a three-dimensional analysis of the energy
        \eqref{energy_large_mass} currently seems to be out of reach.

        Before we describe how to pass from sets to boundaries in a
        suitable manner, we start with collecting standard notions for
        single curves.

\begin{Def}
  We will consider regular closed curves
  $\gamma : \Sph^1 \to \R^2$ to be parametrized by $t\in [0,1]$.  The
  corresponding Sobolev space is
\begin{align}
  H^2(\mathbb S^1; \mathbb R^2) :=  \left\{\gamma \in H_\mathrm{loc}^2(\R;
  \mathbb R^2) : \gamma(t+1)= \gamma(t) \quad \forall t \in \R
  \right\},
\end{align}
and throughout the paper we will only refer to continuous
representatives.  A curve
  $\gamma \in H^2(\mathbb S^1; \mathbb R^2)$ is called regular if
  $\gamma'(t) \not= 0$ for any $t \in [0,1]$. The length of such a
  curve is
\begin{align}
  L(\gamma) := \int_0^1 | \gamma' | \intd t.
\end{align}
We abbreviate the image of the curve as $\Gamma:= \gamma([0,1])$.  For
$x\in \R^2\setminus \Gamma$, we define the winding number of $\gamma$
around $x$ as
	\begin{align}\label{eqref:winding_number}
          \mathcal{I}(\gamma,x) := \frac{1}{2\pi} \int_0^1 
          \frac{ (\gamma(t)-x)^\perp \cdot \gamma'(t) }{|\gamma(t) -x|^2}
          \intd t,
	\end{align}
        where $y^\perp := (-y_2, y_1)$ for every
          $y = (y_1, y_2) \in \R^2$.  We will say that a regular
        curve is parametrized by constant speed if for all
        $t\in [0,1]$ we have
\begin{align}
	|\gamma'(t)| = L(\gamma).
\end{align}

\end{Def}
\noindent One can readily check that $\mathcal{I}(\gamma,0) =1$
for $\gamma(t) = (\cos(2\pi t), \sin(2\pi t))$, which is a
  constant speed parametrization for $t\in [0,1]$.

  It is instructive to consider the passage from a bounded set
  $\Omega \subset \R^2$ with smooth boundary to its boundary curves in
  detail.  Of course, the boundary $\partial \Omega$ of such a set can
  always be decomposed into the union of the images of a finite
  collection of smooth, disjoint Jordan curves
  $\{\gamma_i\}_{i=1}^N \subset C^\infty(\Sph^1; \R^2)$ for some
  $N \in \N$, i.e., smooth, closed curves $\gamma_i$ parametrized by
  $t \in [0,1]$ without self-intersections. Such a decomposition
    in the case of regular sets is classical and can be found, e.g.,
    in the appendix of Milnor's book on differential topology
    \cite{milnor}. See also Ambrosio {\em et al.} for the
    corresponding and much deeper result on sets of finite perimeter
    in the plane \cite{ambrosio01}.  Throughout the paper, we always
  order such curves by decreasing length.  With this notation, we have
  $\partial \Omega = \bigcup_{i=1}^N \Gamma_i$, where
  $\Gamma_i = \gamma_i([0,1])$. Furthermore, we always orient the
  curves such that at every point of the boundary, the \emph{outward}
  normal $\nu_i$ is given by
  \begin{align}\label{eq:chosen_normal}
  	\nu_i := - \tau_i^\perp,
  \end{align}
  where
  \begin{align}
  	\tau_i := \frac{\gamma_i'}{|\gamma_i'|}
  \end{align}
  is the unit tangent along the curve.
  This way, the curvature $\kappa$ of $\partial \Omega$ (positive if
  $\Omega$ is convex) and the curvature $\kappa_i : \Sph^1 \to \R$ of
  the boundary curve $\gamma_i$ coincide in the sense that
  $\kappa(\gamma_i(t)) = \kappa_i(t)$ for all $t \in [0,1]$.  For
  constant speed curves we have the identities
\begin{align}
  \nu_i' & =  \kappa_i L(\gamma_i)
           \tau_i, \label{eq:derivative_normal}\\ 
  \gamma_i'' & = - \kappa_i L^2(\gamma_i)
               \nu_i. \label{eq:second_derivative} 
\end{align}
For $\sigma>0$, the limit energy of $\Omega$ in
\eqref{elastica_smooth} can then be written in terms of the family of
constant speed boundary curves $\{\gamma_i\}_{i=1}^N$ as
\begin{align}
  \tilde F_{\infty,\sigma}(\Omega) = \sum_{i=1}^N \left(
  \sigma L(\gamma_i) + \frac{\pi}{2} L(\gamma_i) \int_0^1
  \kappa_i^2(t) \intd t \right).  
\end{align}

It remains to recover $\Omega$ by means of its boundary curves.  By
the Jordan decomposition theorem, for every $i =1, \ldots, N$, we can
always decompose $\R^2$ into an interior of the curve $\gamma_i$ and
an exterior as
  \begin{align} \label{eq:int} \operatorname{int}(\gamma_i) & :=
    \left\{ x\in \R^2\setminus \Gamma_i :
      |\mathcal{I}(\gamma_i,x)| =
                                                              1\right\},\\
    \operatorname{ext}(\gamma_i) & := \left\{ x\in \R^2\setminus
                                   \Gamma_i :
                                   \mathcal{I}(\gamma_i,x) = 0
                                   \right\}.
	\end{align}
        Here, the absolute value in \eqref{eq:int} accounts for both
        counter-clockwise and clockwise oriented curves.  Indeed, if
        $\gamma_i$ is oriented counter-clockwise, we have
        $\mathcal{I}(\gamma_i,x) = 1 $ for all
        $x\in \operatorname{int}(\gamma_i)$, while for a curve
        $\gamma_i$ oriented clockwise we have
        $\mathcal{I}(\gamma_i,x) = -1 $ for all
        $x\in \operatorname{int}(\gamma_i)$.  Via elementary
        combinatorics, one can then recover the original bounded set
        $\Omega$ as
\begin{align}\label{eq:omega_via_winding}
  \Omega = \left\{x \in \R^2\setminus \bigcup_{i=1}^N \Gamma_i:
  \sum_{i=1}^N \mathcal{I}(\gamma_i,x)\equiv 1 \, (\mathrm{mod} \, 2) 
  \right\}.  
\end{align}
In fact, with the orientations of the boundary curves chosen for
identity \eqref{eq:chosen_normal} to hold, we even have
$\sum_{i=1}^N \mathcal{I}(\gamma_i,x) = \chi_\Omega(x)$ for
$x\in \R^2\setminus \bigcup_{i=1}^N \Gamma_i$.  However, to streamline
the arguments in this paper, not fixing the orientation and
taking the sum modulo 2 instead is more convenient.

We now extend these notions to collections of regular, closed, but not
necessarily simple curves in the Sobolev space $H^2(\Sph^1; \R^2)$ and
thus having square integrable curvature. We note from the start that
the formulas in \eqref{eq:chosen_normal}--\eqref{eq:second_derivative}
clearly remain valid a.e.\ for such curves parametrized with constant
speed.

\begin{Def}
  \label{def:H2curves}
  Let $\sigma>0$.  Let $I \subset \N$ be finite, let
  $\{\gamma_i\}_{i\in I} \subset H^2(\Sph^1;\R^2)$ be a collection of
  regular closed curves, and for each $i \in I$ let
  \begin{align}
    \label{eq:kappai}
    \kappa_i := { (\gamma_i')^\perp \cdot \gamma_i'' \over |\gamma'_i|^3}
    \in L^2(\Sph^1)
  \end{align}
be the curvature of $\gamma_i$.  We abbreviate
$\gamma := \{\gamma_i\}_{i\in I}$ and
$\Gamma:= \bigcup_{i \in I} \Gamma_i$, where
$\Gamma_i := \gamma_i([0,1])$.  We then define
\begin{align}
    \label{eq:Fhat}
  \hat F_{\infty,\sigma} (\gamma) := \sum_{i\in I} \left(
  \sigma L(\gamma_i) + \frac{\pi}{2} \int_0^1 \kappa_i^2(t)
  |\gamma_i'(t)| \intd t \right)
	\end{align}
	and, for $x\in \R^2\setminus \Gamma$,
	\begin{align}
          \mathcal{I}(\gamma,x):= \sum_{i\in I} \mathcal{I}(\gamma_i,x).
	\end{align}
        Finally, we define
	\begin{align}\label{eq:winding_interior}
          A_{\gamma}^o := \{ x \in \R^2\setminus \Gamma :
          \mathcal{I}(\gamma,x) \equiv 1 \, (\mathrm{mod}\, 2) \}. 
	\end{align}
\end{Def}

Notice that while $\Omega = A_{\gamma}^o$ when every curve $\gamma_i$
is simple, this need not hold in the relaxation process: If two
\emph{interior} boundaries collapse, as in the example in Figure
\ref{fig:collapse_interior}, then $A_{\gamma}^o$ excludes a
one-dimensional segment. This exceptional set of course has measure
zero, but needs to be taken care of in topological statements,
motivating the following definition.

\begin{figure}
 	\centering
	\includegraphics{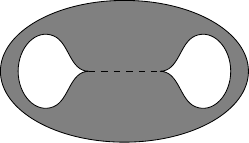}
        \caption{\label{fig:collapse_interior} An example of a set
          with a ``collapsed'' interior boundary shown as dashed. }
\end{figure}

\begin{Def}[Bellettini, Mugnai \cite{belletini2004characterization}]
  Given a set of finite perimeter $\Omega \subset \R^2$, we define
  the open set
\begin{align}
    \label{eq:Omstar}
  \Omega^* := \left\{ x\in \R^2 : \exists r>0 : | B_r(x) \setminus
  \Omega |=0 \right\}, 
\end{align}
while $\partial^*\Omega$ denotes the reduced boundary of $\Omega$.
The set of its system of $H^2$-boundary curves is then
  defined as
\begin{align}
  \begin{split}
    G(\Omega) & := \Big\{ \{\gamma_i\}_{i\in I} \subset H^2( \Sph^1
    ;\R^2): I \subset \N, \ |I| < \infty,
    \ \partial^* \Omega \subset \Gamma, \\
    & \qquad \qquad \qquad \Omega^*= \mathrm{int}\left(A_\gamma^o \cup
      \Gamma \right) ,\ |\gamma_i'| \equiv \text{const} \ \forall
    i \in I \Big\} ,
  \end{split}
\end{align}
with the convention that $G(\Omega) := \emptyset$ if such a system
  of curves does not exist.
\end{Def}
\noindent In the example of Figure \ref{fig:collapse_interior} the
  system of boundary curves will consist of an outer circle and a
  single interior curve which traverses the collapsed interior
  boundary interval twice. Here the set $A_\gamma^o$ is shown in gray,
  while the set $\Omega^*$ is obtained from $A_\gamma^o$ by adding
  back the white interval without the cusp points (resulting in a disk
  with two holes).

  The following representation of the relaxed elastica functional
  was established by Bellettini and Mugnai.

  \begin{thm}[Bellettini, Mugnai
    \cite{belletini2004characterization}]\label{thm:relaxation}
  For a bounded set $\Omega \subset \R^2$ of finite perimeter, we have
	\begin{align}
          F_{\infty,\sigma}(\Omega) = \inf_{\gamma \in G(\Omega)} \hat
          F_{\infty,\sigma}(\gamma). 
	\end{align}
\end{thm}

\subsection{Outline of the proof}

Theorem \ref{thm:main} being a $\Gamma$-convergence statement, its
proof is roughly split into a compactness part, a lower
bound, and an upper bound.  However, here we take ``compactness'' to
mean that limit sets essentially have an $H^2$-regular boundary rather
than showing that all finite energy sequences have an $L^1$-convergent
subsequence, which is wrong, as noted below Theorem
\ref{thm:main}.

The first step is to rewrite the energy in the
form, following the ideas of Muratov and Simon
\cite{muratov2018nonlocal}:
\begin{align}\label{eq:repr_simplified}
  \begin{split}
    F_{\lambda,\alpha} (\Omega) & = \left( 1-\frac{1}{2\pi\alpha^2}
    \right) P(\Omega)\\
    & \quad +\frac{1}{ 4 \pi \alpha} \int_{\partial^* \Omega}
    \int_{H^0_-(\nu(y)) \Delta \lambda (\Omega - y )} \left| \nu(y)
      \cdot \frac{z}{|z|} \right| \frac{e^{-\alpha |z|}}{ |z|} \intd^2
    z \intd \mathcal{H}^1(y),
  \end{split}
\end{align}
where $H^0_-(\nu(y))$ denotes the half-plane through
  $0$ sharing the outward normal $\nu(y)$ with $\Omega$ at
$y\in \partial^* \Omega$.  See Figure
\ref{fig:domains_of_integration} for an illustration and Lemma
\ref{lem:energy_on_boundary} for the precise statement.

        The strategy for proving compactness loosely follows ideas of
        the derivation of plate theory by Friesecke, James, and
        M{\"u}ller \cite{friesecke2002theorem} in that we provide an
        $L^2$-bound for difference quotients along the sequence.
        However, our situation is much simpler as the representation
        \eqref{eq:repr_simplified} directly provides a quantitative,
        non-local comparison of the set with its tangent half-planes
        without having to first establish further rigidity properties.
        Therefore, two tangent half-planes at two close boundary
        points cannot deviate too much without increasing the energy.
        We can also only have finitely many boundary curves as the
        elastica energy of short, closed curves blows up.

        For the upper and lower bounds, we introduce an anisotropic
        version of the blowup used in the identity
        \eqref{eq:repr_simplified}, so that we can expect the blowup
        to approach the subgraph of a parabola with curvature at the
        vertex determined by the curvature of $\Omega$.  In the
        upper bound, we will be able to work with a fixed and regular
        set to make this intuition rigorous and to compute the
        resulting energy contribution.  We will argue similarly for
        the lower bound, but even if we can restrict ourselves to only
        considering sequences of regular sets by a density argument,
        we will have to deal with quite a few measure-theoretic
        details to handle the geometric consequences of weak
        $H^2$-convergence.
        
        	\begin{figure}
		\centering
		\includegraphics{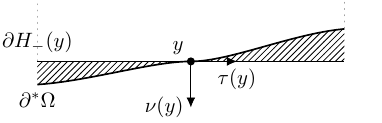}
		\caption{\label{fig:domains_of_integration} Sketch
                  indicating the domain of integration in $z$
                  (hatched) around $y \in \partial^* \Omega$ in the
                  representations \eqref{eq:repr_simplified} and
                  \eqref{eq:representation_lower_bound}.  The domain
                  $\Omega$ is located above the solid curve and the
                  half-plane $H_-(y)$ is located above the solid
                  line, respectively.}
	\end{figure}	

\section{Preliminaries and existence of minimizers}
\label{sec:prelim}

Before we turn to the individual steps, we present a rewriting via
integration by parts of the non-local term in $F_{\lambda,\alpha}$ in
terms of a mixed boundary/bulk integral.  A similar computation was
already crucial in identifying the critical $\Gamma$-limit in the case
of dipolar repulsion \cite{muratov2018nonlocal}.

To this end we solve the equation
$\Delta \Phi_\alpha (|z|) = \frac{e^{-\alpha |z|}}{|z|}$ in
  $\mathbb R^2 \backslash \{0\}$ with sufficient decay at infinity.
This gives
\begin{align}
	\Phi_\alpha(r) & =  \frac{1}{\alpha} \operatorname{E_1}(\alpha r),\\
	\Phi_\alpha'(r) & = -  \frac{e^{-\alpha r}}{\alpha r},\label{Phi_prime}
\end{align}
where $E_1(z) := \int_{z}^\infty \frac{e^{-t}}{t} \intd t$ for $z>0$
is the exponential integral.  For $\nu \in \Sph^1$ and
$y\in \partial^* \Omega$ for a set $\Omega$ of finite perimeter we
also define
\begin{align}
  H_-^0(\nu)& := \left\{ x  \in \R^2 : \nu \cdot x< 0 \right\},\\
  H_-(y)& := \left\{ x  \in \R^2 : \nu(y) \cdot (x- y) < 0
          \right\},\label{def:half-space} 
\end{align}
where $\nu (y)$ denotes the outward unit normal of $\Omega$ at $y$.
Let furthermore
\begin{align}
  R_\nu & := e_2 \otimes \nu - e_1 \otimes \nu^\perp, \\
	A_\lambda & := \lambda e_1\otimes e_1 + \lambda^2 e_2\otimes
                    e_2, 
\end{align}
where $\nu^\perp = (-\nu_2, \nu_1)$ is the 90-degree counter-clockwise
rotation of $\nu = (\nu_1, \nu_2)$, i.e., $R_\nu\in SO(2)$ is the
unique rotation such that $R_\nu \nu = e_2$, and $A_\lambda$ is a
matrix of anisotropic dilations along the first and the second
coordinate directions. Here and everywhere below $z = (z_1, z_2)$.

\begin{lemma}\label{lem:energy_on_boundary}
	Let $\Omega \in \mathcal{A}_\pi$.
	Then we have the representations
	\begin{align}\label{eq:representation_lower_bound}
	  \begin{split}
            F_{\lambda,\alpha}(\Omega)
            & = \left( 1- \frac{1}{2\pi \alpha^2}\right) P(\Omega)  \\
            &\quad + \frac{1}{4\pi \alpha} \int_{\partial^* \Omega}
            \int_{ H_-^0(\nu(y)) \Delta \lambda (\Omega - y) }
            \left| \nu(y) \cdot \frac{ z}{|z|} \right|
            \frac{e^{-\alpha|z|}}{|z|} \intd^2 z \intd \Hd^1(y)
	  \end{split}\\
	  \begin{split}
            & = \left( 1- \frac{1}{2\pi \alpha^2}\right) P(\Omega)  \\
            &\quad + \frac{1}{4\pi \alpha \lambda^2} \int_{\partial^*
              \Omega} \int_{ H_-^0(e_2) \Delta A_\lambda R_{\nu(y)}
              (\Omega - y) } \left| z_2 \right| \frac{e^{-\alpha
                \sqrt{z_1^2 + \frac{z_2^2}{\lambda^2} }}}{z_1^2 +
              \frac{z_2^2}{\lambda^2}} \intd^2 z \, \intd
            \Hd^1(y).
	  \end{split}\label{eq:representation_revised}
	\end{align}
        In particular, we have
          \begin{align}
            \label{eq:FlaP}
            F_{\lambda,\alpha}(\Omega)
            \geq \left( 1- \frac{1}{2\pi \alpha^2}\right) P(\Omega).
          \end{align}
\end{lemma}

These representations have the advantage that the non-local term
penalizes the deviation of $\Omega$ from its tangent half-plane at
each $y\in \partial^* \Omega$, see Figure
\ref{fig:domains_of_integration}.  Furthermore, the perimeter term
already exhibits the correct leading order behaviour.  In particular,
together with Proposition \ref{prop:expansion} we immediately get the
already-known $\Gamma$-convergence statement of Theorem
  \ref{thm:pegon} for sufficiently strong screening as a corollary.

The proof of Lemma \ref{lem:energy_on_boundary} relies on integrating
the kernel by parts, as already mentioned, and at each point of the
boundary moving the expected non-local contribution of the tangent
half-plane to the perimeter.

With this representation, the proof of existence of minimizers is a
surprisingly simple computation.

\begin{proof}[Proof of Proposition \ref{lem:minimizers}]
  As is common in the field, the full proof operates by the
  concentration-compactness dichotomy.  We refer the reader to, for
  example, the proof of \cite[Lemma 4.4]{muratov2018nonlocal} (see
  also \cite{knupfer2016low}) for the details.  In the
  following we only prove that for all $\beta \in (0,1)$ we have
	\begin{align}\label{eq:strict_subadd}
          \inf_{\mathcal{A_{\pi}}} F_{\lambda,\alpha} <
          \inf_{\mathcal{A_{\beta \pi}}} F_{\lambda,\alpha} +
          \inf_{\mathcal{A}_{(1-\beta) \pi}} F_{\lambda,\alpha}, 
	\end{align}
	which can then be used to rule out the splitting case in the
        concentration-compactness principle.
		
	For every $\Omega \in \mathcal A_{\pi}$, we compute, using the
        representation \eqref{eq:representation_lower_bound} and the
        condition $\alpha > \frac{1}{\sqrt 2 \pi}$, that
	\begin{align}
	  \begin{split}
            F_{\lambda,\alpha}\left(\beta^\frac{1}{2} \Omega\right)
            & =  \beta^\frac{1}{2} \left( 1- \frac{1}{2\pi
                \alpha^2}\right) P(\Omega)  \\ 
            &\quad + \frac{\beta}{4\pi \alpha} \int_{\partial^* \Omega}
            \int_{ H_-(\nu(y)) \Delta \lambda (\Omega - y) } \left|
              \nu(y) \cdot \frac{ z}{|z|} \right|
            \frac{e^{-\beta^{\frac{1}{2}}\alpha|z|}}{|z|} \intd^2 z
            \intd \Hd^1(y)\\ 
            &\geq \beta \left( 1- \frac{1}{2\pi \alpha^2}\right)
            P(\Omega) + \beta^\frac{1}{2} (1 - \beta^\frac{1}{2})
              \left( 1- \frac{1}{2\pi \alpha^2}\right) P(\Omega)   \\ 
            &\quad + \frac{\beta}{4\pi \alpha} \int_{\partial^* \Omega}
            \int_{ H_-(\nu(y)) \Delta \lambda (\Omega - y) } \left|
              \nu(y) \cdot \frac{ z}{|z|} \right|
            \frac{e^{-\alpha|z|}}{|z|} \intd^2 z \intd \Hd^1(y)\\
            & \geq \beta \inf_{\mathcal{A_{\pi}}} F_{\lambda,\alpha}
            + \beta^\frac{1}{2} (1 - \beta^\frac{1}{2}) \left( 1-
                \frac{1}{2\pi \alpha^2}\right) P(\Omega),
	  \end{split}
	\end{align}
	and thus
        $\inf_{\mathcal A_{\beta\pi}} F_{\lambda,\alpha}> \beta
        \inf_{\mathcal{A_{\pi}}} F_{\lambda,\alpha}$.  Similarly, we
        have
        $ \inf_{\mathcal A_{(1-\beta)\pi}} F_{\lambda,\alpha}>
        (1- \beta) \inf_{\mathcal{A_{\pi}}} F_{\lambda,\alpha}$, so
        adding the two inequalities gives the claim
        \eqref{eq:strict_subadd}.
        
        The rest of the statement can be proved as in
        \cite[Proposition 2.1]{knupfer2013isoperimetric}. The
        regularity theory for quasi-minimizers of the perimeter
        implies $C^{1,\beta}$-regularity of $\Omega$ for any
        $\beta \in (0, \frac12)$, see for example \cite[Theorem
        21.8]{maggi} or \cite[Theorem 1.4.9]{rigot00}.  By
        \cite[Theorem 5.2]{gatto}, the potential
        \begin{align}
          v(x) := \frac{1}{ 2 \pi} \int_\Omega
          \frac{e^{-\alpha|x-y|}}{|x-y|} \intd x, \qquad x \in
          \R^2,
        \end{align}
        is of class $C^{0,\alpha}(\R^2)$ for any $\alpha \in (0,1)$.
        Therefore, the Euler-Lagrange equation
         \begin{align}
           \kappa(x) +  v(x) = \mu,\qquad x \in
           \partial \Omega,
        \end{align}
        where $\mu \in \R$ is the Lagrange multiplier for the
        mass constraint, holds
        in the weak sense in a local Cartesian frame in which
          $\partial \Omega$ is a $C^{1,\beta}$ graph.  Consequently,
          $\partial \Omega$ is of class $C^{2,\alpha}$ for any
          $\alpha \in (0,1)$. In particular, $\Omega$ is a bounded
          open set with finitely many holes. Finally, as the kernel is
          repulsive, any minimizer $\Omega$ must be connected, since
          otherwise moving different connected components far apart
          lowers the energy.
\end{proof}
	
\begin{proof}[Proof of Lemma \ref{lem:energy_on_boundary}]
  As in \cite{muratov2018nonlocal}, the proof relies on the
  application of the Gauss-Green theorem to the double integral in
  \eqref{energy_large_mass}. However, due to a mild singularity of the
  kernel absent in \cite{muratov2018nonlocal} we need an additional
  approximation argument to express the non-local term as an integral
  over the interior and the reduced boundary of the set of finite
  perimeter $\Omega$.  To that end, let $\eta \in C^\infty(\R)$ be a
  cutoff function with $\eta' \leq 0$ such that $\eta(t) = 1$ for all
  $t \leq \frac12$ and $\eta(t) = 0$ for all $t \geq 1$. For
  $\eps > 0$ and $R > 0$ we define a short-range cutoff
  $\eta_\eps(t) = \eta(t / \eps)$ and a long-range cutoff
  $\eta_R(t) = \eta(t / R)$, respectively, and observe that by the
  monotone convergence theorem we have
    \begin{align}\label{eq:cutoff_approx}
      \int_\Omega \int_{\stcomp{\Omega}} \frac{e^{-\lambda
      \alpha |x-y|}}{|x-y|} \intd^2 y \intd^2 x = \lim_{\eps \to 0, \, R \to
      \infty} \int_\Omega \int_{\stcomp{\Omega}} (1 - \eta_\eps(|y -
      x|)) \eta_R(|y - x|) \frac{e^{-\lambda
      \alpha |y-x|}}{|y-x|} \intd^2 y \intd^2 x.
    \end{align}
    Then recalling the definition of $\Phi_\alpha$ and integrating by
    parts in $y$, which is now justified \cite{ambrosio}, with the
    help of Fubini's theorem we obtain
    \begin{align} \label{bigints}
      \begin{split}
        & \quad \int_\Omega \int_{\stcomp{\Omega}} (1 - \eta_\eps(|y -
        x|)) \eta_R(|y - x|) \frac{e^{-\lambda
            \alpha |y-x|}}{|y-x|} \intd^2 y \intd^2 x \\
        & = \int_\Omega \int_{\stcomp{\Omega}} (1 - \eta_\eps(|y -
        x|)) \eta_R(|y - x|) \Delta_y \Phi_{\lambda
          \alpha}(|y - x|) \intd^2 y \intd^2 x \\
        & = - \int_\Omega \int_{\partial^* \Omega} (1 - \eta_\eps(|y -
        x|)) \eta_R(|y - x|) \, \nu(y) \cdot \nabla_y \Phi_{\lambda
          \alpha}(|y - x|) \intd \mathcal H^1(y) \intd^2 x \\
        &\qquad + \int_\Omega \int_{\stcomp{\Omega}} \eta_R(|y - x|)
        \nabla_y \eta_\eps(|y - x|) \cdot \nabla_y \Phi_{\lambda
          \alpha}(|y - x|) \intd^2 y \intd^2 x \\
        &\qquad - \int_\Omega \int_{\stcomp{\Omega}} (1 - \eta_\eps(|y
        - x|)) \nabla_y \eta_R(|y - x|) \cdot \nabla_y \Phi_{\lambda
          \alpha}(|y - x|) \intd^2 y \intd^2 x .
      \end{split}
    \end{align}

    Notice that from \eqref{Phi_prime} we have
    \begin{align}\label{fubini}
      |\nu(y) \cdot \nabla_y \Phi_{\lambda
      \alpha}(|y-x|)| \leq \frac{e^{-\alpha \lambda |x -
      y|}}{\alpha \lambda |x-y|}, 
    \end{align}
    which is integrable over
    $(x,y) \in \Omega \times \partial^* \Omega$ by Fubini's
    theorem. Hence applying the dominated convergence theorem, we
    obtain
    \begin{align}\label{eq:main_term}
      \begin{split}
        \lim_{\eps \to 0, \, R \to \infty} \int_\Omega
        \int_{\partial^* \Omega} (1 - \eta_\eps(|y - x|)) \eta_R(|y -
        x|) \, \nu(y) \cdot \nabla_y \Phi_{\lambda
          \alpha}(|y - x|) \intd \mathcal H^1(y)  \intd^2 x  \\
        = \int_\Omega \int_{\partial^* \Omega} \nu(y) \cdot \nabla_y
        \Phi_{\lambda
          \alpha}(|y - x|) \intd \mathcal H^1(y)  \intd^2 x \\
        = - \int_{\partial^* \Omega} \int_{\Omega} \nu(y) \cdot
        \nabla_x \Phi_{\lambda \alpha}(|y-x|)\intd^2 x \intd
        \Hd^1(y).
     \end{split}
    \end{align}
    Similarly, the last term in the right-hand side of \eqref{bigints}
    vanishes in the limit.
    
    Thus, it remains to evaluate the second integral on the right-hand
    side of \eqref{bigints}, which for $R > 2 \eps$ can be written as
    \begin{align}
    \begin{split}
      \int_\Omega \int_{\stcomp{\Omega}} 
      \eta_R(|y - x|) \nabla_y \eta_\eps(|y - x|) \cdot \nabla_y
      \Phi_{\lambda 
      \alpha}(|y - x|) \intd^2 y \intd^2 x      \\
      = {1 \over \eps} \int_\Omega \int_{B_\eps(x)}  
      |\eta'(\eps^{-1} |y - x|) \Phi'_{\lambda \alpha}(|y - x|)|
      \intd^2 y \intd^2 x .
     \end{split}
    \end{align}
    Therefore, by \eqref{Phi_prime} and defining
    $\phi_\eps(x) := \frac{1}{ 2 \pi \eps |x|} |\eta'(|x|/\eps)|$ for
    $x\in \R^2$ we have
    \begin{align}
     \begin{split}
     & \quad \frac{1}{2 \pi} \left| \int_\Omega \int_{\stcomp{\Omega}} 
      \eta_R(|y - x|) \nabla_y \eta_\eps(|y - x|) \cdot \nabla_y
      \Phi_{\lambda 
      \alpha}(|y - x|) \intd^2 y \intd^2 x \right|      \\
     & \leq  \frac{1}{ \lambda \alpha } \int_\Omega \int_{B_\eps(x)}  \frac{
      |\eta'(\eps^{-1} |y - x|)|}{2
      \pi \eps |y - x|} \intd^2 y \intd^2 x\\
      &  = \frac{1}{ \lambda \alpha}
      \int_\Omega \int_{\stcomp{\Omega}} \phi_\eps(x - y) \intd^2 y
      \intd^2 x.
      \end{split}
    \end{align}
    The function $\phi_\eps$ is non-negative with $\phi_\eps(x) = 0$
    for all $x \in \R^2$ such that $|x| > \eps$, and
    $\int_{\R^2} \phi_\eps(x) \intd^2 x = 1$, so that it is an
    approximation of a Dirac delta as $\eps \to 0$. Thus by the
    standard approximation argument for the characteristic function of
    $\Omega$ in $L^1(\R^2)$ by uniformly bounded smooth functions with
    compact support we have
    \begin{align}\label{eq:vanishes}
      \lim_{\eps \to 0} \int_\Omega \int_{\stcomp{\Omega}}
      \phi_\eps(|x - y|) \intd^2 y \intd^2 x = 0.
    \end{align}
    
    Putting together equations \eqref{eq:cutoff_approx},
    \eqref{bigints}, and \eqref{eq:main_term}--\eqref{eq:vanishes} we
    have
	\begin{align}
	  \begin{split}
            \int_{\Omega} \int_{\stcomp{\Omega}} \frac{e^{-\lambda
                \alpha |x-y|}}{|x-y|} \intd^2 y \intd^2 x
            & = \int_{\partial^* \Omega} \int_{\Omega} \nu(y) \cdot
            \nabla_x \Phi_{\lambda \alpha}(|y-x|)\intd^2 x \intd
            \Hd^1(y).
	  \end{split}
	\end{align}
	Therefore, we can write
	\begin{align}\label{eq:splitup}
	  \begin{split}
            & \quad - \int_{\Omega} \int_{\stcomp{\Omega}}
            \frac{e^{-\lambda \alpha |x-y|}}{|x-y|} \intd^2 y  \intd^2
            x  \\ 
            & =  \quad  \int_{\partial^* \Omega}  \int_{\Omega} \nu(y)
            \cdot \frac{ x-y}{|x-y|}  |\Phi'_{\lambda
              \alpha}(|y-x|)|\intd^2 x \intd \Hd^1(y) \\ 
            & = \int_{\partial^* \Omega}  \int_{H_-(y)} \nu(y) \cdot
            \frac{ x-y}{|x-y|}  |\Phi'_{\lambda \alpha}(|y-x|)|\intd^2
            x \intd \Hd^1(y) \\ 
            & \quad + \int_{\partial^* \Omega}  \int_{\Omega\setminus
              H_-(y)} \nu(y) \cdot \frac{ x-y}{|x-y|}  |\Phi'_{\lambda
              \alpha}(|y-x|)|\intd^2 x \intd \Hd^1(y) \\ 
            & \quad -\int_{\partial^* \Omega} \int_{H_-(y) \setminus
              \Omega} \nu(y) \cdot \frac{ x-y}{|x-y|} |\Phi'_{\lambda
              \alpha}(|y-x|)|\intd^2 x \intd \Hd^1(y).
	  \end{split}
	\end{align}
	For every $y \in \partial^* \Omega$, we compute
	\begin{align}
	  \begin{split}
            \int_{H_-(y)} \nu(y) \cdot \frac{ x-y}{|x-y|}
            |\Phi'_{\lambda \alpha}(|y-x|)|\intd^2 x & =
            -\int_{-\frac{\pi}{2}}^{\frac{\pi}{2}} \int_0^\infty
            \frac{e^{-\lambda \alpha r} \cos
                \theta}{\lambda \alpha} \intd r \intd \theta \\
              & = -\frac{2}{\lambda^2 \alpha^2}
              \label{eq:2lam2al2}
	 \end{split}
	\end{align}
	Together with the combinatorics of the sign of
        $\nu(y) \cdot \frac{ x-y}{|x-y|}$ for $x \in H_-(y)$ and
        $x \not \in H_-(y)$, \eqref{eq:splitup} and
          \eqref{eq:2lam2al2} thus result in
	\begin{align}
	  \begin{split}
            & \quad - \int_{\Omega} \int_{\stcomp{\Omega}}
            \frac{e^{-\lambda \alpha |x-y|}}{|x-y|} \intd^2 y  \intd^2 x\\
            & = -\frac{2}{\lambda^2 \alpha^2} P(\Omega) +
            \int_{\partial^* \Omega} \int_{ H_-(y) \Delta \Omega }
            \left| \nu(y) \cdot \frac{ x-y}{|x-y|} \right|
            |\Phi'_{\lambda \alpha}(|y-x|)|\intd^2 x \intd \Hd^1(y)\\
            & = -\frac{2}{\lambda^2 \alpha^2} P(\Omega) +
            \frac{1}{\lambda^2 \alpha} \int_{\partial^* \Omega} \int_{
              H^0_-(\nu(y)) \Delta \lambda(\Omega - y) } \left|
              \nu(y) \cdot \frac{ z}{|z|} \right| \frac{e^{-\alpha
                |z|}}{|z|}\intd^2 z \intd \Hd^1(y),
	  \end{split}
	\end{align}
	which proves equation \eqref{eq:representation_lower_bound}.

	Finally, we calculate
	\begin{align}
	  \begin{split}
            &\quad \int_{ H^0_- (\nu(y)) \Delta \lambda(\Omega -
              y) } \left| \nu(s) \cdot \frac{ z}{|z|} \right|
            \frac{e^{-\alpha |z|}}{|z|} \intd^2 z \\
            & = \frac{1}{\lambda^2} \int_{ H^0_- (e_2) \Delta
              A_\lambda R_{\nu(y)} (\Omega - y) } \left| z_2 \right|
            \frac{e^{-\alpha \sqrt{z_1^2 + \frac{z_2^2}{\lambda^2}
                }}}{z_1^2 + \frac{z_2^2}{\lambda^2}} \intd^2 z,
	  \end{split}
	\end{align}
	giving equation \eqref{eq:representation_revised}.
\end{proof}

\section{Compactness}
\label{sec:compact}

\subsection{Single boundary curves}
We start out by proving compactness for a single sequence of boundary
curves of a finite energy sequence.  By density of regular sets
in the sets of finite perimeter, we may as well assume that the
sequence consists of regular sets.  The main point here is to prove
that the limit is sufficiently regular to have curvature in $L^2$.
       
To this end, the first step is to obtain a discrete $H^1$ estimate for
the normals along the sequence, that is, for fixed $\lambda$ and
$\alpha$.  In order to control the geometry of the curves in the lower
bound, we also need an estimate for how often two boundary points (be
they from the same boundary curve doubling up on itself or from two
different boundary curves) with wildly different tangents can be close
to each other. Hence we also record a consequence of the arguments
pertaining to two mismatched, close-by normals regardless of
which boundary curve they belong to.
        
        \begin{lemma}\label{lemma:discrete}
          Let $\alpha_0 > 0$ and $K > 0$. Then there exist
          $C, C', C'' > 0$ with the following property:
          If $\lambda > 0$, $\Omega \in \mathcal{A}_\pi$ is regular
          and $\gamma :[0,1] \to \R^2$ is a smooth Jordan boundary
          curve of $\Omega$ parametrized by constant speed, then for
          all $\alpha \in (0, \alpha_0)$ and $s \in [-K,K]$ we have
          the estimate
	  \begin{align}\label{eq:norms_boundedness}
	  \begin{split}
            & \quad L(\gamma) \int_0^1 \left|\nu \left(t+ (L(\gamma)
                \lambda)^{-1} s \right)- \nu(t)\right|^2
            \intd t \\
            & \leq C \int_{\Gamma} \int_{ H^0_- (\nu(y)) \Delta
              \lambda(\Omega - y) } \left| \nu(y) \cdot \frac{ z}{|z|}
            \right| \frac{e^{-\alpha |z|}}{|z|} \intd^2 z
            \intd \Hd^1(y) \\
            & \leq C' F_{\lambda,\alpha}(\Omega).
	  \end{split}
          \end{align}
          Furthermore, for
	\begin{align}
	  \begin{split}
            Z_{K,\lambda} & := \big\{ (y_1,y_2) \in \partial \Omega
            \times \partial
            \Omega : \\
            & \qquad \qquad|y_1-y_2| \leq K \lambda^{-1}, \
            \nu(y_1)\cdot \nu (y_2) \leq 0 \big\}.
	  \end{split}
	\end{align}
	we have
	\begin{align}\label{eq:weak_estimate}
          \mathcal{H}^2(Z_{K,\lambda})  \leq  C ''
          P(\Omega)F_{\lambda,\alpha}(\Omega).
	\end{align}
\end{lemma}

Using this information, we can prove the compactness statement for
sequences of single boundary curves.  As the $L^1$-topology disregards
sets with vanishing mass, we only have to consider sequences of curves
whose length does not converge to zero in the limit. We exclude
  the natural lack of compactness due to the translational symmetry of
  the problem by pinning one point on each curve.

\begin{lemma}\label{lemma:compactness_single_curve}
  Let $\lambda_n > 0$ and $\alpha_n > \frac{1}{\sqrt{2\pi}}$ be
    such that $\lambda_n \to \infty$ as $n \to \infty$ and such that
  $\sigma_n := \lambda_n^2 \left( 1- \frac{1}{2\pi \alpha_n^2}
  \right)$ satisfies
  \begin{align}
    \label{eq:siglim}
          \lim_{n\to \infty} \sigma_n = \sigma >0.
	\end{align}
	Let $(\Omega_{n}) \subset \mathcal{A}_\pi$ be a sequence of
        regular sets such that
	\begin{align}\label{eq:compactness_boundedness}
		\limsup_{n \to \infty} \lambda_n^2 F_{\lambda_n,
          \alpha_n} (\Omega_n) <\infty. 
	\end{align}
        Let $\gamma_{n} :[0,1] \to \R^2$ be a smooth Jordan boundary
        curve of $\Omega_{n}$, and assume that all $\gamma_{n}$ are
        parametrized with constant speed, with
	\begin{align} \label{eq:liminfLn}
		 \liminf_{n \to \infty} L(\gamma_{n})>0.
	\end{align}
        Then there exists a subsequence $(\gamma_{n_k})$ of
          $(\gamma_n)$ and
        $\gamma_\infty \in H^2(\Sph^1; \R^2)$ such that
        \begin{align}
          \gamma_{n_k} - \gamma_{n_k}(0) \to \gamma_\infty
          \qquad 
          \text{in} \quad
          H^1(\Sph^1;\R^2),
        \end{align}
        as $k \to \infty$.  In particular, we have
        $\lim_{k \to \infty} L(\gamma_{n_k}) =
        L(\gamma_\infty) >0$.  Furthermore, there exists a universal
        constant $C>0$ such that
        along a further subsequence (not relabeled) we have
	\begin{align}\label{eq:up_to_constants}
	  \begin{split}
            &\quad  \hat F_{\infty,\sigma}(\gamma_\infty)\\
            & \leq C \liminf_{k \to \infty} \Bigg(
            \sigma_{n_k} L(\gamma_{n_k})\\
            & \qquad + \lambda_{n_k}^2 \int_{\Gamma_{n_k}} \int_{
              H^0_- (\nu_{n_k}(y)) \Delta
              \lambda_{n_k}(\Omega_{n_k} - y) } \left| \nu_{n_k}(y)
              \cdot \frac{ z}{|z|} \right|
            \frac{e^{-\alpha_{n_k}|z|}}{|z|} \intd^2 z \intd \Hd^1(y)
            \Bigg).
	  \end{split}
	  \end{align}
\end{lemma}

\begin{proof}[Proof of Lemma \ref{lemma:discrete}]
  For $x\in \R^2$ and $i=1,2$, we abbreviate
	\begin{align}
          g_i(x)& := \nu(y_i) \cdot (x-\lambda y_i),\\
          \mu_i (x) &:= \left|g_i(x) \right|
                      \frac{e^{-\alpha|x-\lambda
                      y_i|}}{|x-\lambda y_i|^2}.
	\end{align}
	\textit{Step 1: We claim that for all $y_1, y_2 \in \partial
          \Omega$ with $\lambda | y_1- y_2| \leq K$, we have
	\begin{align}
	  \begin{split}
            |\nu(y_1 ) - \nu(y_2)|^2 & \leq C \int_{\R^2} \left|
              \chi_{\lambda H_-(y_1)} - \chi_{\lambda
                \Omega}
            \right| \mu_1 \intd^2 x \\
            & \quad + C \int_{\R^2} \left| \chi_{\lambda H_-(y_2)} -
              \chi_{\lambda \Omega} \right| \mu_2 \intd^2 x,
	  \end{split}
	\end{align}
        for some $C > 0$ depending only on $K$ and $\alpha_0$.}
      
      With the goal of comparing tangent spaces at $y_1$ and $y_2$, we
      have by the triangle inequality
	\begin{align}\label{eq:compactness_triangle}
	   \begin{split}
             & \int_{\R^2} \left| \chi_{\lambda H_-\left(y_1
                 \right)} - \chi_{\lambda \Omega} \right| \mu_1
             \intd^2 x + \int_{\R^2} \left| \chi_{\lambda H_-(y_2)}
               -
               \chi_{\lambda \Omega} \right| \mu_2 \intd^2 x\\
             & \geq \int_{\R^2} \left| \chi_{\lambda
                   H_-\left(y_1\right)} - \chi_{\lambda
                   H_-(y_2)} \right|
             \min\left(\mu_1(x),\mu_2\left(x\right)\right) \intd^2 x.
	  \end{split}
	\end{align}
	Let $\bar y := \frac{\lambda (y_1 + y_2)}{2}$.  By the
        assumption $\lambda |y_1 -y_2| \leq K$, for $i=1,2$ we
        have
	\begin{align}\label{eq:distance_blowup_points}
          |\lambda y_i - \bar y| \leq \frac{K}{2}.
	\end{align}
	Therefore, for all $x \in B_{ K}(\bar y)$ we obtain
	\begin{align}
          \max\left\{ | x- \lambda  y_1| , |x-\lambda y_2|  \right\}
          &            \leq \frac{3}{2}  K 
	\end{align}
	so that we have
	\begin{align}\label{eq:compactness_kernel_bound}
          \min\left(\frac{e^{-\alpha|x-\lambda
          y_1|}}{|x-\lambda y_1|^2},  \frac{e^{-\alpha|x-\lambda
          y_2|}}{|x-\lambda y_2|^2} \right) \geq \frac{4 e^{-\frac32
          \alpha_0 K}}{ 9 K^2}. 
	\end{align} 
	Together with \eqref{eq:compactness_triangle}, we arrive at
	\begin{align}\label{eq:compactness_terrible}
	   \begin{split}
             & \int_{\R^2} \left| \chi_{\lambda H_-\left(y_1 \right)}
               - \chi_{\lambda \Omega} \right| \mu_1 \intd^2 x +
             \int_{\R^2} \left| \chi_{\lambda H_-(y_2)} -
               \chi_{\lambda \Omega} \right| \mu_2 \intd^2 x\\
             & \geq C^{-1} \int_{\left(  \lambda H_- \left(y_1
                   \right) \Delta \lambda H_-(y_2) \right) \cap B_{
                 K}(\bar y)} \min\left\{|g_1(x)|,|g_2(x)| \right\}
             \intd^2 x,
	  \end{split}
	\end{align}
        for some $C > 0$ depending only on $K$ and $\alpha_0$.
	
	\begin{figure}
		\centering
	\includegraphics{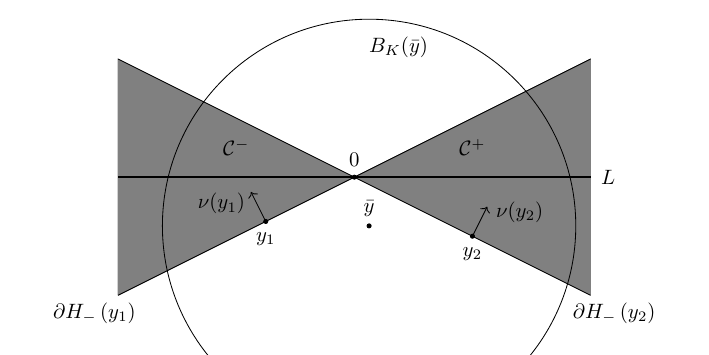}
		\caption{\label{fig:cones} Sketch of the cones
                  $\mathcal{C}^\pm$, the line $L$, the points $y_1$,
                  $y_2$, $\bar y$, the normals $\nu_1$ and $\nu_2$,
                  and the ball $B_K(\bar y)$.}
	\end{figure}
	
	We now interpret the integral on the right-hand side if
        \eqref{eq:compactness_terrible} as the volume of a
        three-dimensional body, aiming to estimate
        $|\nu(y_1) - \nu(y_2)|^2$ from above.  Therefore, we may
        assume that $\nu(y_1) \neq \nu(y_2)$ and, without loss of
        generality, that $\partial H_-(y_1 )$ and $\partial H_-(y_2)$
        intersect at the origin.  Thus there exists a closed cone
        $\mathcal{C}^+ \subset \R^2$ with vertex at $0$ and half-angle
        $\theta \in [0,\frac{\pi}{2})$, see Figure \ref{fig:cones},
        such that
	\begin{align}
          \mathcal{C}^+ \cup \mathcal{C}^-=
          \overline{\lambda H_-( y_1 ) \, \Delta \, \lambda
          H_-(y_2)}, 
	\end{align}
	where $\mathcal{C}^- := -\mathcal{C}^+$.
	Now recall that by its definition $|\lambda^{-1} g_i(x)|$
          is the distance from the point $\lambda^{-1} x$ to
          $\partial H_-(y_i)$. Hence the set
	\begin{align}
          L := \{ x \in \mathcal{C}^+ \cup \mathcal{C}^-: |g_1(x)| = |g_2(x)| \}
	\end{align}
	defines a line that bisects $\mathcal{C}^\pm$, that is, it
        separates $\mathcal{C}^+$ and $\mathcal{C}^-$ into two
        respective sub-cones with apertures
        $\theta \in [0,\frac{\pi}{2})$.
	
	Let $x \in L$.  Then for $i=1,2$, we have
        $x - g_i(x) \nu(y_i) \in \partial H_-(y_i)$ and
	\begin{align}\label{eq:expr_graph}
          | g_i(x) | = |x| \sin \theta
	\end{align}
	Let $\tau \in L$ be such that $|\tau| =1$ and
        $\rho \tau \in \mathcal{C}^+ \subset \R^2$ for all
        $\rho >0$.  We define the three-dimensional sets
	\begin{align}
          \tilde L^\pm
          & := \left\{\left( \pm  \rho  \tau ,
            \rho  \sin \theta \right)
            \in \R^3 :
            \rho  \in (0,\infty)  \right\} ,\\ 
          \mathcal{\tilde C}^\pm
          & := \operatorname{conv} \left(
            \left(\mathcal{C}^\pm\times\{0\} \right) \cup \tilde L^\pm
            \right). 
	\end{align}
	In particular, $ \mathcal{\tilde C}^\pm$ are two
        three-dimensional cones, see an illustration in Figure
          \ref{fig:3d_cone}.

	By linearity of $g_i$ for $i=1,2$, the definition of $L$, and
        the identity \eqref{eq:expr_graph}, we can interpret the
        integral on the right hand side of
        \eqref{eq:compactness_terrible} as
	\begin{align}\label{eq:compactness_geometry}
	  \begin{split}
            & \quad \int_{\left( (\lambda H_-\left(y_1 \right) )
                \Delta (\lambda H_-(y_2) ) \right) \cap B_{
                K}(\bar
              y)}  \min\left\{|g_1(x)|,|g_2(x)| \right\} \intd^2 x \\
            &= \left|\left(\mathcal{\tilde C}^+ \cup \mathcal{\tilde
                  C}^-\right) \cap \left( B_{ K}(\bar y) \times \R
              \right) \right|.
	  \end{split}
	\end{align}
	As a result of estimate \eqref{eq:distance_blowup_points}, we
        have $B_{\frac{K}{2}}(\lambda y_i) \subset B_{ K}(\bar y)$ for
        $i=1,2$, so that
	\begin{align}
	  \begin{split}
            \left|\left(\mathcal{\tilde C}^+ \cup \mathcal{\tilde
                  C}^-\right) \cap \left( B_{ K}(\bar y) \times
                \R \right) \right| \geq 
              \left|\left(\mathcal{\tilde C}^+ \cup \mathcal{\tilde
                    C}^-\right) \cap \left( B_{ \frac{K}{2}}(\lambda y_i)
                  \times \R \right) \right|.
	  \end{split}
	\end{align}
	Without loss of generality, we may assume
          $y_2 \in \mathcal{C}^+$, as in Figure \ref{fig:cones}, so
        that
	\begin{align}\label{eq:compactness_symmetry}
	  \begin{split}
            \left|\left(\mathcal{\tilde C}^+ \cup \mathcal{\tilde
                  C}^-\right) \cap \left( B_{ K}(\bar y) \times \R
              \right) \right| \geq  \left|\mathcal{\tilde C}^+
                \cap \left( B_{ \frac{K}{2}}(\lambda y_2) \times \R \right)
              \right|.
	  \end{split}
	\end{align}
	In turn, by monotonicity of the right-hand side of
        \eqref{eq:compactness_symmetry} with respect to sliding the
        ball center along $\partial H_-(y_2)$, we have
	\begin{align}
          \left|\mathcal{\tilde C}^+
          \cap \left( B_{ \frac{K}{2}}(\lambda y_2) \times \R \right)
          \right| \geq \left|\mathcal{\tilde C}^+
          \cap \left( B_{ \frac{K}{2}}(0) \times \R \right)
          \right| .
	\end{align}
      
	\begin{figure}
		\centering

		\includegraphics{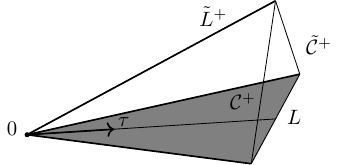}
		\caption{\label{fig:3d_cone} Sketch of the
                  three-dimensional cone $\mathcal{\tilde C}^+$.}
	\end{figure}

        As $\mathcal{\tilde C}^+$ is a cone with opening half-angle
        $\theta$ in the horizontal direction and aperture $\theta$ in
        the vertical direction, we have
	\begin{align}\label{eq:minimizing_point}
          \left|\mathcal{\tilde C}^+ \cap  \left( B_{ \frac{K}{2}}(0)
          \times \R \right)   \right| \geq C K^3 \sin^2
         \theta, 
	\end{align}
        for some $C > 0$ universal.  Finally, by the observation
        that
	\begin{align}
          |\nu(y_1)- \nu(y_2) |= 2 \sin\theta,
	\end{align}
	and estimates
        \eqref{eq:compactness_terrible} and
        \eqref{eq:compactness_geometry}--\eqref{eq:minimizing_point},
        we obtain
	\begin{align}
	  \begin{split}
            |\nu(y_1) - \nu(y_2)|^2 & \leq C \int_{\R^2} \left|
              \chi_{\lambda H_-(y_1)} - \chi_{\lambda \Omega}
            \right| \mu_1 \intd^2 x \\
            & \quad + C \int_{\R^2} \left| \chi_{\lambda
                H_-(y_2)} - \chi_{\lambda \Omega} \right| \mu_2
            \intd^2 x,
	  \end{split}
	\end{align}
        for some $C > 0$ depending only on $K$ and $\alpha_0$,
        proving the claim.

          \textit{Step 2: Estimate the normals along a curve.}  As
          $\gamma$ is parametrized by constant speed, for all
          $t \in [0,1]$ and $s \in [-K,K]$ we have
	\begin{align}
          \lambda \left|\gamma(t) - \gamma\left(t+
          (L(\gamma)\lambda)^{-1} s \right) \right| \leq K. 
	\end{align}
	The first estimate in \eqref{eq:norms_boundedness} then
        follows by taking $y_1 = \gamma(t)$ and
        $y_2 = \gamma\left(t+ (L(\gamma)\lambda)^{-1} s \right)$ in
        Step 1 and integrating in $t$, while the second one is
          obtained with the help of Lemma
          \ref{lem:energy_on_boundary}.

        \textit{Step 3: Estimate mismatched normals.}  For
        $(y_1,y_2) \in Z_{K,\lambda}$, Step 1 implies
	\begin{align}
	  \begin{split}
            2& \leq |\nu(y_1) - \nu(y_2)|^2\\
            & \leq C \int_{\R^2} \left| \chi_{\lambda H_-(y_1)} -
              \chi_{\lambda \Omega} \right| \mu_1
            \intd^2 x  \\
            & \quad + C \int_{\R^2} \left| \chi_{\lambda
                H_-(y_2)} - \chi_{\lambda \Omega} \right| \mu_2
            \intd^2 x.
	  \end{split}
	\end{align}
	Integrating jointly in $y_1$ and $y_2$ over
          $Z_{K,\lambda} \subset \Gamma \times \Gamma$, as in
          Step 2 we obtain
	\begin{align}
          \mathcal{H}^2(Z_{K,\lambda})  \leq  C' P(\Omega)
          F_{\lambda,\alpha}(\Omega), 
	\end{align}
	for some $C' > 0$ depending only on $K$ and
          $\alpha_0$. This concludes the proof.
\end{proof}

\begin{proof}[Proof of Lemma \ref{lemma:compactness_single_curve}]
  Throughout the following, we will never relabel the sequences after
  passing to subsequences. Without loss of generality, we may also
  assume that $\gamma_n(0) = 0$ for all $n \in \mathbb N$. We
    recall that \eqref{eq:siglim} implies that
    $\alpha_n \to {1 \over \sqrt{2 \pi}}$ as $n \to \infty$.
  
  \textit{Step 1: Establish the limit behavior of the normals.}
  We begin by observing that by
    \eqref{eq:compactness_boundedness}, \eqref{eq:liminfLn} and Lemma
    \ref{lem:energy_on_boundary} the limit
  $L_\infty:= \lim_{n \to \infty} L(\gamma_{n}) \in (0,\infty)$ exists
  along a subsequence.
  
  Let $\phi : \R \to \R$ be the standard mollifier and let
  $\phi_\delta = \frac{1}{\delta} \phi \left( \frac{\cdot}{\delta}
  \right)$ be an approximation for the Dirac delta in one dimension
  with support in $(-\delta, \delta)$ for $\delta \to 0$.  Consider
  the maps
  $\tilde \nu_{n} := \phi_{(L(\gamma_{n}) {\lambda_n})^{-1}} \ast
  \nu_{n}$ as 1-periodic convolutions, where $\nu_n$ is the outward
  normal to $\gamma_n$ defined in \eqref{eq:chosen_normal}. In
  particular, we have $|\tilde \nu_n| \leq 1$. Note that we are
  mimicking convolution in arc-length coordinates on the fixed domain
  $[0,1]$.  For $L(\gamma_{{n}}) \lambda_{n} > 2$, which holds for $n$
  large enough due to $L_\infty>0$, we may apply the Cauchy-Schwarz
  inequality to split off $\phi$ and get
          \begin{align}\label{eq:rigidity_close}
	  \begin{split}
            \int_0^1 |\tilde \nu_{n} - \nu_{n}|^2 \intd t & = \int_0^1
            \left|\int_{-\frac{1}{2}}^{\frac{1}{2}}
              \phi_{(L(\gamma_{n}){\lambda_n})^{-1}} (\tau) (\nu_{n}(t
              + \tau) - \nu_{n}(t)) \intd \tau \right|^2 \intd t
            \\
            & \leq \int_{-\frac{1}{2}}^{\frac{1}{2}}
            \phi_{(L(\gamma_{n}) {\lambda_n})^{-1}} (\tau) \int_0^1
            \left|\nu_{\lambda_n}(t + \tau) -
              \nu_{\lambda_n}(t) \right|^2 \intd t \intd \tau \\
            & = \int_{-1}^{1} \phi (s) \int_0^1
              \left|\nu_{\lambda_n}(t + (L(\gamma_n) \lambda_n)^{-1}
                s) - \nu_{\lambda_n}(t) \right|^2 \intd t \intd s .
	  \end{split}
	\end{align}
	Applying the $L^\infty$-type estimate
          \eqref{eq:norms_boundedness} to the last term in
          \eqref{eq:rigidity_close} for $K= 1$, we get
        \begin{align}
	  \begin{split}
            & \quad L(\gamma_{n})  \int_0^1
            |\tilde
            \nu_{n} - \nu_{n}|^2 \intd t \\
            & \leq C \int_{\Gamma_{n}} \int_{
              H^0_-(\nu_{n}(y)) \Delta{\lambda_n}(\Omega - y) }
            \left| \nu_{n}(y) \cdot \frac{ z}{|z|} \right|
            \frac{e^{-\alpha |z|}}{|z|} \intd^2 z \intd \Hd^1(y),
	  \end{split}
	\end{align}
	for some $C > 0$ universal and all $n$ large enough. In
        particular, by Lemma \ref{lem:energy_on_boundary} this ensures
        tha
	\begin{align}\label{eq:rigity_convergence}
		\tilde \nu_{n} - \nu_{n} \to 0
	\end{align}
	in $L^2 (\Sph^1; \R^2)$ as $n \to \infty$.
	
        Similarly, due to
        $\int_{-\frac{1}{2}}^{\frac{1}{2}}
        \phi_{(L(\gamma_{n}){\lambda_n})^{-1}}' \intd t =0$ for
        $n \in \N $ sufficiently large and the Cauchy--Schwarz
        inequality to split off $|\phi '|$, we obtain
	\begin{align}\label{eq:rigidity_derivative}
	  \begin{split}
            \int_0^1 |\tilde \nu_{n}'|^2 \intd t & = \int_0^1
            \left|\int_{-\frac{1}{2}}^{\frac{1}{2}}
              \phi_{(L(\gamma_{n}) {\lambda_n})^{-1}} '(\tau)
              (\nu_{n}(t + \tau) - \nu_{n}(t)) \intd \tau \right|^2
            \intd t
            \\
            & \leq \| \phi_{(L(\gamma_{n}){\lambda_n})^{-1}}'\|_{L^1}
            \int_{-\frac{1}{2}}^{\frac{1}{2}} \left|
              \phi_{(L(\gamma_{n}) {\lambda_n})^{-1}} '(\tau) \right|
            \int_0^1 \left|\nu_{n}(t + \tau) - \nu_{n}(t) \right|^2
            \intd t \intd \tau.
	  \end{split}
	\end{align}
	Combining the fact that
        $\|\phi'_{(L(\gamma_{n}){\lambda_n})^{-1}}\|_{L^1} \leq
        CL(\gamma_{n}) {\lambda_n}$ for some $C > 0$ universal with
        the $L^\infty$-type estimate \eqref{eq:norms_boundedness}
        for $K= 1$, we furthermore get
	\begin{align}\label{eq:derivative_normal_bound}
	  \begin{split}
            & \quad \frac{1}{L (\gamma_{n})} \int_0^1 |\tilde
            \nu_{n}'|^2 \intd t \\ 
            & \leq C \lambda_n^2 \int_{\Gamma_{n}} \int_{
              H^0_-(\nu_{n}(y)) \Delta \lambda(\Omega_n - y) }
            \left| \nu_{n}(y) \cdot \frac{ z}{|z|} \right|
            \frac{e^{-\alpha |z|}}{|z|} \intd^2 z \intd \Hd^1(y),
	  \end{split}
	\end{align}
        again, for some $C > 0$ universal and all $n$ large
          enough.  In particular, by
          \eqref{eq:norms_boundedness},
          \eqref{eq:compactness_boundedness} and Lemma
          \ref{lem:energy_on_boundary} we have that $\tilde \nu'_{n}$
        is uniformly bounded in $L^2(\Sph^1; \R^2)$, and so by
          uniform boundedness of $|\tilde \nu_n|$ there exists
        $\tilde \nu_\infty \in H^{1}(\Sph^1;\R^2)$ such that upon
        extraction of a subsequence
        $\tilde \nu_{n} \rightharpoonup \tilde \nu_\infty$ in
        $H^1(\Sph^1; \R^2)$ as $n \to \infty$. In turn, by the
        Rellich-Kondrachov theorem we get
        $\tilde \nu_{n} \to \tilde \nu_\infty$ strongly in
        $L^2(\Sph^1; \R^2)$ as $n \to \infty$.  Hence, due to
        convergence \eqref{eq:rigity_convergence}, we also get
        \begin{align} \label{eq:nunnuinfty}
          \nu_{n} \to \tilde \nu_\infty \qquad \text{strongly in} \quad
          L^2(\Sph^1; \R^2). 
        \end{align}
        Furthermore, since
        $|\nu_n| = 1$, we obtain that $|\tilde \nu_\infty| =1$ as well
        (recall that
        $\tilde \nu_\infty \in C^\frac{1}{2}(\Sph^1; \R^2)$ by the
        corresponding Sobolev embedding).

        \textit{Step 2: Construct a limit curve
          $\gamma_\infty \in H^2(\Sph^1; \R^2)$.}  By
        \eqref{eq:nunnuinfty}, \eqref{eq:chosen_normal} and the fact
        that $|\gamma_n'| = L(\gamma_n)$ we have
          \begin{align} \label{eq:gammanpL} \gamma_{n}' =
            L(\gamma_n) \nu_{n}^\perp \to L_\infty \tilde
            \nu_\infty^\perp \qquad \text{strongly in} \quad
          L^2(\Sph^1; \R^2).
        \end{align}
        At the same time, since we assumed without loss of generality
        that $\gamma_n(0) = 0$, upon extraction of a subsequence we
        get $\gamma_{n} \rightharpoonup \gamma_\infty$ weakly in
        $H^1(\Sph^1; \R^2)$ for some
        $\gamma_\infty \in H^1(\Sph^1; \R^2)$ with
          $\gamma_\infty(0) = 0$ as $n \to \infty$. In particular, by
        \eqref{eq:gammanpL} we have
        \begin{align} \label{eq:gammainftynuperp}
          \gamma_\infty' = L_\infty \tilde \nu_\infty^\perp,
        \end{align}
        and $\gamma_n \to \gamma_\infty$ strongly in
        $H^1(\Sph^1; \R^2)$.  From the strong convergence, it follows
        that
        $L_\infty = \lim_{n \to \infty} L(\gamma_{n}) =
        L(\gamma_\infty)$. From \eqref{eq:gammainftynuperp} and the
        fact that $|\tilde \nu_\infty| = 1$ we thus obtain that
        $|\gamma_\infty'| = L(\gamma_\infty)$, i.e., that
        $\gamma_\infty$ is a closed curve parametrized with constant
        speed. Finally, we conclude that
        $\gamma_\infty \in H^2(\Sph^2;\R^2)$ from
        \eqref{eq:gammainftynuperp} and the fact that
        $\tilde \nu_\infty \in H^1(\Sph^1; \R^2)$.
	
        \textit{Step 3: Estimate the elastica energy up to constants.}
        Observe that by the identities \eqref{eq:kappai} and
        \eqref{eq:gammainftynuperp}, together with the constant
          speed parametrization, we have
      \begin{align}
      	|\kappa_\infty| = \frac{|\gamma_\infty''|}{
        L^2(\gamma_\infty)} = \frac{|\tilde \nu_\infty'|}{
        L(\gamma_\infty)} \in 
        L^2(\Sph^2).
      \end{align}
      By $\lim_{n \to \infty} L(\gamma_n) = L(\gamma_\infty)$, weak
      convergence of $\tilde \nu_n'$, and lower semi-continuity of the
      norms, we therefore obtain (recall Definition
        \ref{def:H2curves})
	\begin{align}
	  \begin{split}
            \hat{F}_{\infty,\sigma}(\gamma_\infty) & =
            L(\gamma_\infty) \left(\sigma + \frac{\pi}{2}
              \int_0^1 \kappa_\infty^ 2 \intd t\right)
              \leq 
              \liminf_{n \to \infty} \left( \sigma_n L(\gamma_n) +
                \frac{\pi}{ 2 L(\gamma_n)} \int_0^1 |\tilde \nu'_n|^2
                \intd t \right).
             \end{split}
           \end{align}      
           Estimate \eqref{eq:derivative_normal_bound} then yields
           \eqref{eq:up_to_constants}, concluding the proof.
\end{proof}

\subsection{Identifying the asymptotic system of boundary curves}
	
The main issues in compactness for all boundary curves are,
first, proving that there may only exist finitely many limit
curves and, second, that their limits are a system of boundary curves
to the limiting set.

For the first part, we use the estimate \eqref{eq:up_to_constants}
together with the Gauss-Bonnet theorem for closed curves to show that
the limit energy of curves blows up as their length approaches zero.
The isoperimetric inequality ensures that curves whose lengths vanish
in the limit $\lambda_n \to \infty$ do not contribute to the
$L^1$ $\Gamma$-limit. Combined, these two facts also ensure that
  not all boundary curves would vanish in the limit.
  
  \begin{lemma}\label{lemma:finite}
    Let $\lambda_n > 0$ and $\alpha_n >\frac{1}{\sqrt{2\pi}}$ be
      such that $\lambda_n \to \infty$ as $n \to \infty$ and such
      that
    $\sigma_n = \lambda_n^2 \left( 1- \frac{1}{2\pi \alpha_n^2}
    \right)$ satisfies
	\begin{align}
         \lim_{n\to \infty} \sigma_n = \sigma >0.
	\end{align}
	Let $(\Omega_{n}) \subset \mathcal{A}_\pi$ be a sequence of
        regular sets such that
        \begin{align}
          \label{eq:Flam2M}
                  M := \limsup_{n \to \infty} \lambda_n^2 F_{\lambda_n,
                  \alpha_n} (\Omega_{n}) <\infty,
                \end{align}
                and for $n \in \N$, let $N_n \in \N$ be the number of
                constant speed boundary curves
                $\{\gamma_{n,i}\}_{i=1}^{N_n}$ of the
                set $\Omega_n$, enumerated by
                decreasing length.  Finally, let $N \geq 0$ be the
                  number of non-vanishing boundary curves as
                  $n \to \infty$:
	\begin{align}\label{eq:N}
          N:= \sup \left( \{0\} \cup \left\{ i \in \N: \limsup_{n \to \infty} L(
          \gamma_{n,i}) >0 \right\} \right),
	\end{align}
	with the convention that $L(\gamma_{n,i}) = 0$ if $i > N_n$.
        Then the following holds:
        \begin{enumerate}[i)]
        \item For the non-vanishing curves, we have
          \begin{align}\label{eq:def_n_max}
            1 \leq N\leq C M\sigma^{-\frac{1}{2}},
          \end{align}
          for some $C > 0$ universal. 
                
        \item For the vanishing curves, we have
	\begin{align}\label{eq:vanishing_curves}
		\lim_{n \to \infty} \sup_{i> N} L(\gamma_{n,i}) = 0
	\end{align}
	and
	\begin{align}\label{eq:convergence_short_curves}
          \lim_{n\to \infty} \sum_{i > N} | \operatorname{int}(\gamma_{n,i})|= 0.
	\end{align}
      \end{enumerate}
      
    \end{lemma}

    We remark that for sequences of sets whose energy is
      comparable (within a universal constant) to that of the
      minimizers we have
      \begin{align}
        2 \pi \sigma \leq M \leq C (\sigma  + \sqrt{\sigma}),
      \end{align}
      for some $C > 0$ universal: The lower bound is due to the
      isoperimetric inequality and the upper bound is obtained by
      testing with either a disk for $\sigma \geq 1$ or an annulus for
      $\sigma < 1$, see estimate \eqref{eq:annuli_estimate}.
      Therefore, for such sets the upper estimate in
      \eqref{eq:def_n_max} translates into
      \begin{align}
        \label{eq:NCMsig}
        N \leq C \left(1 + \sqrt{\sigma} \right), 
      \end{align}
      for some $C > 0$ universal. In particular,
      counterintuitively, this estimate shows that the number of
      non-vanishing boundary curves in a sequence of sets under
      consideration remains uniformly bounded for $\sigma \lesssim 1$
      as $n \to \infty$. At the same time, for $\sigma \gg 1$ the
      number of non-vanishing boundary curves could be large as
      $n \to \infty$, as can be seen from an example of a
      configuration consisting of one disk of $O(1)$ radius and
      $N = O(\sigma^{1/2})$ small disks of radius
      $r = O(\sigma^{-1/2})$ far apart. Thus, for such sequences of
      sets the estimate in \eqref{eq:def_n_max} is sharp.

      It remains to prove that the non-vanishing curves asymptotically
      provide a system of boundary curves for an admissible limiting
      set.  In order to handle the technical issue that even
      relatively long boundary curves may escape to infinity, we
      only take those curves that stay close to the origin, resulting
      in a further restriction on the set of indices $I$ to consider
      in the system of boundary curves.

\begin{prop}\label{prop:compactness}
  Under the assumptions of Lemma \ref{lemma:finite}, let
    $\chi_{\Omega_n} \to \chi_{\Omega_\infty}$ in $L^1(\R^2)$ for
    some $\Omega_\infty \in \mathcal{A}_\pi$ as
      $n \to \infty$. Then there exist a subsequence and a family
  $\{\gamma_{\infty,i}\}_{i\in I} \in G(\Omega_\infty)$ of
  $H^2$-regular, constant speed curves such that for all $i\in I$ we
  have
	\begin{align}
          \gamma_{n,i} \to
          \gamma_{\infty,i} \qquad \text{in} \quad H^1(\Sph^1; \R^2),
	\end{align}
	 where $\gamma_{n,i}$ for $i \in I$ is
        some sub-collection of curves from the decomposition of
        $\Omega_{n}$ into its boundary curves (modulo
          re-indexing).
\end{prop}

\begin{proof}[Proof of Lemma \ref{lemma:finite}]
  We begin by observing that by \eqref{eq:FlaP} and
    \eqref{eq:Flam2M} we have that $P(\Omega_n)$ is uniformly bounded
    and, in particular, so are $L(\gamma_{n,i})$. By the fact that
  our enumeration of $\gamma_{n,i}$ in $i$ has decreasing lengths for
  each $n$, for any subsequence $n_m$ with $m \in \N$ and any
  $j\leq k$, $j,k\in \N$ we have
	\begin{align}\label{eq:decreasing}
          \limsup_{m\to \infty} L( \gamma_{n_m,j}) \geq \limsup_{m\to
          \infty} L( \gamma_{n_m,k}). 
	\end{align}

        \textit{Step 1: Bound on the number of long curves.}  If
        $N=0$, there is nothing to prove. Therefore, we may assume
        that $N \geq 1$ and let $i \in \N$ be such that
        $i \leq N \leq \infty$.  Going to a subsequence $n_m$,
        $m \in \N$, that depends on $i$, we may assume that
	\begin{align}
          \label{eq:Lnilimsuplim}
          \limsup_{n \to \infty} L( \gamma_{n,i}) = \lim_{m \to
          \infty} L( \gamma_{n_m,i}) >0 
	\end{align}
        exists.  By the inequality \eqref{eq:decreasing}, we may
        repeatedly apply Lemma
        \ref{lemma:compactness_single_curve}
	to get a further, non-relabeled subsequence obeying
          \eqref{eq:Lnilimsuplim} and closed limit curves
        $ \gamma_{\infty,j} \in H^2(\Sph^1;\R^2)$ with constant speed
        and
        $L( \gamma_{\infty,j}) = \lim_{m \to \infty}
          L(\gamma_{n_m,j})$ for all $j =1,\ldots,i$ with the
          following property: We have 
	\begin{align}
          \gamma_{n_m,j} - \gamma_{n_m,j}(0) \to
          \gamma_{\infty,j} \qquad  \text{in} \quad H^1(\Sph^1;
          \R^2),
	\end{align}
       as $m \to \infty$, and
	\begin{align}
	  \begin{split}
            & \quad \sum_{j=1}^i L( \gamma_{\infty,j}) \left(\sigma
              +\frac{\pi}{2} \int_0^1  \kappa_{\infty,j}^ 2 \intd
              t\right)\\ 
            & \leq C \sum_{j=1}^i \liminf_{m \to \infty} \Bigg(
            \sigma_{n_m} L( \gamma_{n_m,j})\\
            & \qquad \qquad \qquad \qquad + \lambda_{n,m}^2 \int_{
              \Gamma_{n_m,j}} \int_{ H_-(\nu(y)) \Delta
              \lambda_{n_m}(\Omega_{n_m} - y) } \left| \nu(y) \cdot
              \frac{ z}{|z|} \right| \frac{e^{-|z|}}{|z|} \intd^2 z
            \intd \Hd^1(y) \Bigg),
           \end{split}
         \end{align}
         for some $C > 0$ universal, where $ \kappa_{\infty,j}$
         is the curvature of $ \gamma_{\infty,j}$.  Consequently, we
         have
         \begin{align}\label{eq:sum_compactness_estimate}
           \sum_{j=1}^i L( \gamma_{\infty,j}) \left(\sigma  +
           \frac{\pi}{2} \int_0^1 \kappa_{\infty,j}^2
           \intd t\right) 
           \leq C \limsup_{n \to \infty}
           \lambda_n^2 F_{\lambda_n,\alpha_n}(\Omega_n) =
           C M,
	\end{align}
	for some $C > 0$ universal.
	
	By Fenchel's theorem \cite[Theorem 3 and Remark 5, Section
        5.7]{docarmo}, together with a straightforward approximation
        argument to remove the regularity assumption therein, for
          each $j = 1, \ldots, i$ we have
	\begin{align}
          \int_{ \Gamma_{\infty,j}} |\kappa_{\infty,j}| \intd \Hd^1
          \geq 2\pi. 
	\end{align}
	As a result of Jensen's inequality, we therefore obtain
		\begin{align}
          L( \gamma_{\infty,j}) \int_0^1  \kappa_{\infty,j}^ 2 \intd t
          \geq \frac{1}{ L( \gamma_{\infty,j})} \left(\int_{
          \Gamma_{\infty,j}} |\kappa_{\infty,j}| \intd \Hd^1 \right)^2
          = \frac{4\pi^2}{ L( \gamma_{\infty,j}) }.
	\end{align}
	Combining this with estimate
        \eqref{eq:sum_compactness_estimate} and Young's inequality, we
        obtain
	\begin{align}
          \left(8\pi^3 \sigma\right)^{\frac{1}{2}}i \leq \sum_{j=1}^i
          \left(\sigma L( \gamma_{\infty,j}) + \frac{2\pi^3}{
          L( \gamma_{\infty,j})} \right) \leq C M,
	\end{align}        
        for some $C > 0$ universal. In particular, in view of the
        arbitrariness of $i$ we have $N < \infty$, and the upper bound
        in \eqref{eq:def_n_max} holds.

        \textit{Step 2: Estimates for vanishingly short curves.}  To
        handle the boundary curves $\gamma_{\lambda,i}$ for $i > N$,
        assume towards a contradiction that there exists a sequence
        of $i_n \in \N$ with $i_n > N$ such that
	\begin{align}
		\limsup_{n\to \infty} L(\gamma_{n,i_n}) >0.
	\end{align}
	By discreteness of $\N$, we have $i' := \min_{n>0} i_n \in \N$ and $i' >N$.
	As the curves are ordered by decreasing length, we therefore also get
	\begin{align}
          \limsup_{n \to \infty} L(\gamma_{n,i'}) \geq
          \limsup_{n\to \infty} L(\gamma_{n,i_n}) >0, 
	\end{align}
	which by way of definition \eqref{eq:N} would imply a
        contradiction.  This yields \eqref{eq:vanishing_curves}.
        
        As a result of the isoperimetric inequality and
          \eqref{eq:FlaP} we have
	\begin{align}
          \sum_{i > N} | \operatorname{int}(\gamma_{n,i})|
          \leq
          \frac{1}{4\pi} \sum_{i > N} L^2(\gamma_{n,i})
          \leq
          \frac{ \sup_{i> N} L(\gamma_{n,i}) }{ 4 \pi} \sum_{i > N}
          L(\gamma_{n,i})  \leq \frac{P(\Omega_n)}{4\pi}  \sup_{i> N}
          L(\gamma_{n,i}) \to 0  
	\end{align}
	as $n \to \infty$, proving
        \eqref{eq:convergence_short_curves}.

        Lastly, if $N = 0$ then
          \eqref{eq:convergence_short_curves} would imply
          $|\Omega_n| \to 0$ as $n \to \infty$, contradicting the fact
          that $\Omega_n \in \mathcal A_\pi$ for all $n$ large
          enough. This concludes the proof.
\end{proof}

\begin{proof}[Proof of Proposition \ref{prop:compactness}]
\textit{Step 1: Construct the limiting curves.}

Let $N \in \N$ be as in Lemma \ref{lemma:finite}.  We take a
non-relabeled subsequence such that for all $i =1,\ldots, N$ we have
\begin{align}
	\limsup_{n \to \infty} L( \gamma_{n,i}) =
\lim_{n \to \infty} L( \gamma_{n,i})
\end{align}
and such that we either have
$\limsup_{n \to \infty} \| \gamma_{n,i}\|_\infty < \infty$ or
$\lim_{n \to \infty} \| \gamma_{n,i}\|_\infty = \infty$.  Let
$A \subset \{1,\ldots , N\}$ be the set of indices such that the
former alternative holds, i.e., such that $\gamma_{n,i}$ is uniformly
bounded in $n$ for all $i \in A$. This set is non-empty, as otherwise
by Lemma \ref{lemma:finite} and uniform boundedness of
  $L(\gamma_{n,i})$ this would contradict the $L^1$-convergence of
the characteristic functions of $\Omega_n$ to a set of positive
Lebesgue measure.
Thus we may apply Lemma \ref{lemma:compactness_single_curve} to get
another subsequence such that for all $i \in A$ there exist closed
limit curves $ \gamma_{\infty,i} \in H^2(\Sph^1;\R^2)$ with
constant speed and $L( \gamma_{\infty,i}) >0$ such that
$ \gamma_{n,i} \to \gamma_{\infty,i}$ in $H^1(\Sph^1; \R^2)$ as
$n \to \infty$.  

\medskip

In the following, we use the abbreviations
$\gamma_\infty = \{ \gamma_{\infty,i} \}_{i \in A}$ and
$\Gamma_\infty = \bigcup_{i \in A} \gamma_{\infty,i}([0,1])$. \medskip

\textit{Step 2: Prove $\partial^* \Omega_\infty \subset \Gamma_\infty$.}

Let $x \in \partial^* \Omega_\infty$.  Then there exists
$r_0 \in (0,1)$ such that for all $r\in (0,r_0)$ we have
\cite{ambrosio}
	\begin{align}\label{eq:density_bound}
		\frac{1}{3} < \frac{|\Omega_\infty \cap B_r(x)|}{\pi r^2} < \frac{2}{3}.
	\end{align}
        As $\chi_{\Omega_n} \to \chi_{\Omega_\infty}$ in $L^1$, we
        have for $n$ sufficiently big depending on $r$ that also
	 \begin{align}
		\frac{1}{3} < \frac{|\Omega_n \cap B_r(x)|}{\pi r^2} < \frac{2}{3},
	\end{align}
	so that the relative isoperimetric inequality implies
	\begin{align}
          \label{eq:H1den}
          \frac{\mathcal{H}^1(\partial \Omega_n \cap B_r(x) )}{r}
          > \frac{1}{C}, 
	\end{align}
        for some $C > 0$ universal.  Therefore, from
        \eqref{eq:H1den} and Lemma
        \ref{lemma:finite}, we get for $n$ sufficiently large
        that
	\begin{align}
		\frac{\mathcal{H}^1\left(\left(\bigcup_{ i \in A} \Gamma_{n,i} \right) \cap B_r(x) \right)}{r} > \frac{1}{C}.
	\end{align}
	Consequently, there exists a subsequence $n_k$ for $k \in \N$
        with $\frac{1}{k} < r_0$, $i \in A$, and
        $t_k, t_\infty \in [0,1]$ such that
        $|\gamma_{n_k,i}(t_k) -x| < \frac{1}{k}$ and
        $t_k \to t_\infty$ as $k \to \infty$.  Because the curves
        $\gamma_{n_k,i}$ converge in $C^0(\Sph^1;\R^2)$, we get
        $\gamma_{\infty,i}(t_\infty) = x$.

        \textit{Step 3:
        For almost all $x \in \R^2$ we have $ x \not \in \Gamma_\infty$ and
        \begin{align} 
          \mathcal{I}(\gamma_\infty, x) \equiv
          \chi_{\Omega_\infty}(x) \pmod{2}. 
        \end{align}}

      As $\Gamma_n$ for $n \in \N \cup \{\infty\}$ has Hausdorff
      dimension $1$, we have
      $\left| \bigcup_{n\in \N \cup \{\infty\}} \Gamma_n \right| =0$.
      Let $ x \in \R^2 \setminus \Gamma_\infty$.  Since the curves
      $\gamma_{n,i}$ for $i\in A$ converge in $H^1(\Sph^1; \R^2)$ and
      $C^0(\Sph^1; \R^2)$ and since the set
      $\R^2 \setminus \Gamma_\infty$ is open, we have
	\begin{align}
		\mathcal{I}(\gamma_\infty, x) = \lim_{n \to \infty} \sum_{i\in A} \mathcal{I}(\gamma_{n,i},  x).
	\end{align}
	As for $i \in \N$ with $i \leq N$ and $i \not \in A$ the
        curves $\gamma_{n,i}$ run off to infinity, we must have
        $\mathcal{I}(\gamma_{n,i}, x)=0$ for sufficiently large $n$
        and those values of $i$.  We thus get
	\begin{align}
		\mathcal{I}(\gamma_\infty,  x) = \lim_{n \to \infty} \sum_{i=1}^{N} \mathcal{I}(\gamma_{n,i},  x) 
	\end{align}
	
	Due to the convergence \eqref{eq:convergence_short_curves}, we
        can choose a non-relabeled subsequence
        such that
	\begin{align}
		\sum_{n=1}^\infty \sum_{i > N} \left| \overline{\operatorname{int}(\gamma_{n,i})}\right| < \infty. 
	\end{align}
	The Borel-Cantelli Lemma \cite[p.\ 42]{stein-shakarchi}
        therefore implies that for almost all $x \in \R^2 $ and
        $n_0(x)$ sufficiently big we have
        $ x \not \in \overline{\operatorname{int}(\gamma_{n,i})}$ for
        all $i >N$ and all $n> n_0(x)$.  In particular, for almost all
        $ x \in \R^2$ we get
	\begin{align}\label{eq:index_converges}
		\mathcal{I}(\gamma_\infty,  x) = \lim_{n \to \infty}
          \sum_{i=0}^{N_{n}} \mathcal{I}(\gamma_{{n},i},  x),
	\end{align}
	since eventually, the sum only has at most $N$ non-zero
        terms.
	
	According to the representation \eqref{eq:omega_via_winding},
        for almost all $ x \in \R^2$ we have \begin{align}
          \sum_{i=1}^{N_{n}} \mathcal{I}(\gamma_{n,i}, x) \equiv
          \chi_{\Omega_n}(x) \pmod{2}.
	\end{align}
	By $| \Omega_\infty \Delta \Omega_{n}| \to 0$ as $n \to
        \infty$, we may choose another, non-relabeled subsequence such
        that $\chi_{\Omega_n} \to \chi_{\Omega_\infty}$ pointwise
        almost everywhere.  Combining these insights with the
        convergence \eqref{eq:index_converges}, for almost all
        $x \in \R^2$ we get
        \begin{align}
          \mathcal{I}(\gamma_\infty,  x)  \equiv \chi_{\Omega_\infty}
          (x) \pmod{2}. 
        \end{align}

        \textit{Step 4: Prove
          $\Omega_\infty^*=
          \operatorname{int}\left(A_{\gamma_\infty}^o \cup
            \Gamma_\infty \right)$.}
	
        We recall that $A_{\gamma_\infty}^o$ and
          $\Omega_\infty^*$ are defined via
          \eqref{eq:winding_interior} and \eqref{eq:Omstar},
          respectively. We first prove the inclusion
        $\Omega_\infty^* \subset
        \operatorname{int}\left(A_{\gamma_\infty}^o \cup
          \Gamma_\infty \right)$. Notice that since the set
          $\Omega_\infty^*$ is open, it is enough to show that
          $\Omega_\infty^* \subset A_{\gamma_\infty}^o \cup
          \Gamma_\infty$. So, let $x\in \Omega_\infty^*$.  If
        $x\in \Gamma_\infty$, there is nothing to prove. Otherwise,
        there exists $r \in (0,1)$ such that
        $|B_r(x) \setminus \Omega_\infty |=0$ and
        $B_{ r}(x) \cap \Gamma_\infty = \emptyset$.  By Step 3, for
        almost all $\tilde x \in B_r(x)$ we have
	\begin{align}
          \mathcal{I}(\gamma_\infty, \tilde x) \equiv 1 \pmod{2}.
	\end{align}
	Continuity of the index in $B_{ r}(x)$ then gives
	\begin{align}
          \mathcal{I}(\gamma_\infty, x) \equiv 1 \pmod{2}.
	\end{align}
	Consequently, we get $x \in A_{\gamma_\infty}^o$.
	
	Let now $x \not\in \Omega_\infty^*$.
	Then for every $r \in (0,1)$ we have
	\begin{align}
		 |B_r(x)\setminus (\Omega_\infty \cup \Gamma_\infty) | >0.
	\end{align}
	Again by Step 3, for all $r \in (0,1)$ and almost all
        $\tilde x \in B_r(x)\setminus ( \Omega_\infty \cup
        \Gamma_\infty) $ we get that
	\begin{align}
		\mathcal{I}(\gamma_\infty, \tilde x) \equiv 0  \pmod{2},
	\end{align}
	so that there exists a sequence of points
        $\tilde x_k \in B_{k^{-1}}(x)\setminus
        \left(A_{\gamma_\infty}^o \cup \Gamma_\infty\right)$ for
        all $k \in \N$. Therefore, we have
        $x \not\in \operatorname{int}\left( A_{\gamma_\infty}^o
          \cup \Gamma_\infty\right)$, proving the claim of this
        step.
	
	Finally, combining the statements of Steps 2 and 4 we
          obtain that
        $\{\gamma_{\infty,i}\}_{i \in A} \in G(\Omega_\infty)$,
          concluding the proof.
\end{proof}
	
\section{$\Gamma$-convergence} \label{sec:Gamma}
\subsection{Upper bound}

Both the upper and the lower bounds crucially depend on the
representation \eqref{eq:representation_revised}.  In the upper bound
presented in Lemma \ref{lem:expansion}, we observe that the
anisotropic blowup $A_\lambda R_{\nu(y)} (\Omega - y)$ of a
sufficiently regular set $\Omega \subset \R^2$ converges to the
subgraph of the parabola
$z_1 \mapsto -\frac{1}{2} \kappa(y) z_1^2$.  The integral in
the non-local term in representation
\eqref{eq:representation_revised} can then be explicitly calculated in
the limit using Fubini's theorem, the integral over $z_2$ giving
the dependence of the energy on $\kappa^2(y)$.

\begin{lemma}\label{lem:expansion}
  Let $\alpha > 0$ and let $\Omega$ be a regular set. Then, as
  $\lambda \to \infty$ we have
	\begin{align*}
          F_{\lambda,\alpha}(\Omega) \leq \left( 1- \frac{1}{2\pi
          \alpha^2}\right) P(\Omega) + \frac{1}{8\pi \alpha^4
          \lambda^2}   \int_{\partial \Omega} \kappa^2 \intd \Hd^1 + 
          o\left(\lambda^{-2} \right).
	\end{align*}
   \end{lemma}

   \begin{proof}
     All constants in this proof may depend on $\Omega$ and
       $\alpha$, in contrast to the rest of the paper.
	
     By the identity \eqref{eq:representation_revised}, we only need
     to compute the non-local term therein.  Let
     $y \in \partial \Omega$.  For the sake of convenience, in this
     proof we parametrize $\mathbb{S}^1$ on
     $[-\frac{1}{2},\frac{1}{2}]$ instead of the usual parametrization
     on the unit interval. Let
     $\gamma : [-\frac{1}{2},\frac{1}{2}] \to \R^2$ be a constant
     speed parametrization of the connected component of
     $\partial \Omega$ containing $y$ and such that $\gamma(0) = y$.
		
     For $\lambda>1$ and $s \in \R$, let
     $T_\lambda s := (L(\gamma) \lambda)^{-1} s$, so that $s$ plays
     the role of an arc length parameter after blowup by
       $T_\lambda^{-1} = L(\gamma) \lambda$.  Let $\tau(t) :=
       \gamma'(t) / L(\gamma)$ be the unit tangent vector to $\gamma$
       at point $\gamma(t)$, and let
	\begin{align}
          g_{\lambda}(s) & := \lambda
                           \tau(0) \cdot
                           \left(\gamma(T_\lambda s) - \gamma(0
                           )\right), \label{def:g_upper}\\ 
          h_{\lambda}(s) & := \lambda^2 \nu(0) \cdot
                           \left(\gamma(T_\lambda s ) - \gamma(0
                           )\right).\label{def:h_upper} 
	\end{align}
	be the local Cartesian coordinates of $\gamma(s)$ with respect
        to the orthonormal basis $\{ \tau(0) ,\nu(0)\}$ after
        anisotropic blowup, see Figure \ref{fig:smooth_coords}.  By
        Taylor expansion and identity \eqref{eq:second_derivative},
        for $s \in \R$ we have
	\begin{align}
          g_{\lambda}(s) & = s +
                           O(\lambda^{-2}s^3),\label{eq:g_upper}\\ 
          g'_{\lambda}(s) & = 1 + O(
                            \lambda^{-2}s^2),\label{eq:gprime_upper}\\ 
          h_{\lambda}(s) & = - \frac{ \kappa(y)}{2}s^2+
                           O(\lambda^{-1}s^3),\label{eq:h_upper}  
	\end{align}
	as well as
	\begin{align}\label{eq:gprime_mod}
		|g'_{\lambda}(s)| \leq 1
	\end{align}
        for all $s$.  In particular, if $\lambda^{-1}s$ is
        sufficiently small, the map $g_\lambda$ is monotone increasing
        with $g'_\lambda(s) \geq \frac{1}{2}$ and is therefore
        invertible.  Thus for all $\eps >0$ sufficiently small we have
	\begin{align}
	  \begin{split}
            & \quad A_\lambda \left(  R_{\nu(y)} (\Omega- y)  \cap  (-\eps,\eps)^2\right)\\
            &= \left\{ (z_1,z_2) : z_1 \in (-\lambda \eps ,\lambda
              \eps) , z_2 \in \left(-\lambda^2 \eps , h_\lambda (
                g_\lambda^{-1}(z_1))\right) \right\}.
	  \end{split}
	\end{align}

	\begin{figure}
		\centering
		\includegraphics{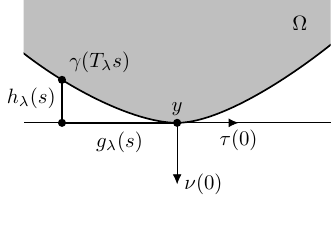}
		\caption{\label{fig:smooth_coords} Sketch of the
                  anisotropic blowup of $\gamma(T_\lambda s)$ in the
                  basis $\{\tau(0),\nu(0)\}$ with components
                  $g_\lambda(s)$ and $h_\lambda(s)$. }
	\end{figure}	

	An explicit calculation in polar coordinates for
          $z = (z_1, z_2)$ gives
	\begin{align}
          \int_{ \R^2 \backslash A_\lambda (-\eps,\eps)^2 }
          \left| z_2  \right| \frac{e^{-\alpha \sqrt{z_1^2 +
          \frac{z_2^2}{\lambda^2} }}}{z_1^2 + \frac{z_2^2}{\lambda^2}}
          \intd^2 z \ \leq \ \lambda^2
          \int_{\stcomp{B}_{\lambda \eps}(0)} \frac{e^{-\alpha
          |z|}}{|z|} \intd^2 z = O\left(\frac{\lambda^2}{\alpha} e^{-
          \eps \alpha \lambda}\right). 
	\end{align}
	By Fubini's theorem, we therefore have
	\begin{align}\label{eq:terror_term_1}
	  \begin{split}
            &\quad \int_{ H^0_-(e_2) \Delta A_\lambda R_{\nu(y)} 
              (\Omega - y) } \left| z_2 \right| \frac{e^{-\alpha
                \sqrt{z_1^2 + \frac{z_2^2}{\lambda^2} }}}{z_1^2 +
              \frac{z_2^2}{\lambda^2}} \intd^2 z \\
		& = \int_{-\lambda \eps}^{\lambda \eps}
                \int_0^{\left| h_\lambda(g_\lambda^{-1}(z_1))\right|} \left| z_2
                \right| \frac{e^{-\alpha \sqrt{z_1^2 +
                      \frac{z_2^2}{\lambda^2} }}}{z_1^2 +
                  \frac{z_2^2}{\lambda^2}} \intd z_2 \intd z_1 +
                O\left(\frac{\lambda^2}{\alpha} e^{- \eps \alpha
                    \lambda}\right).
	  \end{split}
	\end{align}
	By monotonicity of the exponential function and inverse powers, as well as estimate \eqref{eq:gprime_mod}, we get
	\begin{align}\label{eq:terror_term_2}
	  \begin{split}	
            & \quad \int_{-\lambda \eps}^{\lambda \eps}
            \int_0^{\left|h_\lambda(g_\lambda^{-1}(z_1))\right|}
            \left| z_2 \right| \frac{e^{-\alpha \sqrt{z_1^2 +
                  \frac{z_2^2}{\lambda^2} }}}{z_1^2 +
              \frac{z_2^2}{\lambda^2}} \intd z_2 \intd z_1\\
            &  \leq \int_{-\lambda \eps }^{\lambda \eps}
            \int_0^{\left|h_\lambda(g_\lambda^{-1}(z_1))\right|}
            \left| z_2  \right| \frac{e^{-\alpha |z_1|}}{z_1^2 } \intd
            z_2  \intd z_1 \\ 
            & \leq \frac{1 }{2} \int_{g_\lambda^{-1}(-\lambda \eps)
            }^{g_\lambda^{-1}(\lambda \eps)} h^2_\lambda(s)
            \frac{e^{-\alpha |g_\lambda(s)|} }{g^2_\lambda(s) } \intd
            s.
	  \end{split}
	\end{align}
	Using the expansions \eqref{eq:g_upper} and \eqref{eq:h_upper},
        we get
	\begin{align}\label{eq:first_expansion}
	  \begin{split}
            & \quad \frac{1}{2}\int_{g_\lambda^{-1}(-\lambda \eps)
            }^{g_\lambda^{-1}(\lambda \eps)} h^2_\lambda(s)
            \frac{e^{-\alpha |g_\lambda(s)|}}{g^2_\lambda(s) } \intd
            s\\ 
            & = \frac{1 }{8} \int_{g_\lambda^{-1}(-\lambda \eps)
            }^{g_\lambda^{-1}(\lambda \eps)} s^2 \left( \kappa^2(y) +
              O(\lambda^{-1} |s|)\right) e^{-\alpha |s|(1 +
              O(\lambda^{-2}s^2)} \intd s.
	  \end{split}
	 \end{align}
	 
	 As $g_\lambda'$ is strictly positive on
         $(-\lambda \eps, \lambda \eps)$, we have
         $\lambda^{-1} |s| \leq C \eps$ for all
         $s \in \left(g_\lambda^{-1}(-\lambda \eps) ,
           g_\lambda^{-1}(\lambda \eps) \right)$.  For $\eps>0$
         sufficiently small depending only on $\partial \Omega$, the
         first error term in identity \eqref{eq:first_expansion} is
         estimated by
	 \begin{align}
           \int_{g_\lambda^{-1}(-\lambda \eps)
           }^{g_\lambda^{-1}(\lambda \eps)} \lambda^{-1} |s|^3
           e^{-\frac{\alpha}{C}|s|} \intd s \leq \lambda^{-1}
           \int_{\R} |s|^3e^{-\frac{\alpha}{C}|s|} \intd s =
           O(\alpha^{-4}\lambda^{-1}). 
	 \end{align}
	 Similarly, for all
         $s \in \left(g_\lambda^{-1}(-\lambda \eps) ,
           g_\lambda^{-1}(\lambda \eps) \right)$ we have 
	 \begin{align}
           e^{-\alpha |s|\left(1 + O(\lambda^{-2} s^2)\right)} \leq e^{-(1-C\eps)\alpha |s|},
	 \end{align}
	 so that the identity \eqref{eq:first_expansion} can be
         estimated from above as
	\begin{align}\label{eq:terror_term_3}
	  \begin{split}
            & \quad \frac{1 }{2}\int_{g_\lambda^{-1}(-\lambda \eps)
            }^{g_\lambda^{-1}(\lambda \eps)} h^2_\lambda(s)
            \frac{e^{-\alpha |g_\lambda(s)|}}{g^2_\lambda(s) } \intd
            s\\
            & \leq \frac{ \kappa^2(y) }{8}
            \int_{g_\lambda^{-1}(-\lambda \eps)
            }^{g_\lambda^{-1}(\lambda \eps)} s^2 e^{-(1-C\eps)\alpha
              |s|} \intd s + O(\alpha^{-4}\lambda^{-1}).
	  \end{split}
	\end{align}
	Finally, explicit integration gives
	\begin{align}\label{eq:terror_term_4}
	  \begin{split}
            &\frac{ \kappa^2(y) }{8} \int_{g_\lambda^{-1}(-\lambda
              \eps) }^{g_\lambda^{-1}(\lambda \eps)} s^2
            e^{-(1-C\eps)\alpha )|s|} \intd s\\ 
            & \leq \frac{\kappa^2(y)}{8} \int_{\R} s^2
            e^{-(1-C\eps)\alpha |s|} \intd s\\ 
            & = \frac{ \kappa^2(y)}{4\alpha^3(1-C\eps)^3}
            \int_0^\infty s^2 e^{-s} \intd s  \\ 
            & = \frac{ \kappa^2(y)}{2\alpha^3(1-C\eps)^3}.
	  \end{split}
	\end{align}
	
	Combining the estimates \eqref{eq:terror_term_1},
        \eqref{eq:terror_term_2}, \eqref{eq:terror_term_3}, and
        \eqref{eq:terror_term_4}, for all $\eps>0$ sufficiently small
        depending only on $\partial \Omega$ we get
	\begin{align}
	  \begin{split}
            &\quad \int_{ H^0_-(e_2) \Delta A_\lambda R_{\nu(y)}
              (\Omega - y) } \left| z_2 \right| \frac{e^{-\alpha
                \sqrt{z_1^2 + \frac{z_2^2}{\lambda^2} }}}{z_1^2 +
              \frac{z_2^2}{\lambda^2}} \intd^2 z \\
            & \leq \frac{ \kappa^2(y)}{2\alpha^3(1-C\eps)^3} +
            O\left(\alpha^{-4}\lambda^{-1} + \frac{\lambda^2}{\alpha}
              e^{-\eps \alpha \lambda} \right)
	  \end{split}
	\end{align}
	Choosing $\eps = \lambda^{-\frac{1}{2}}$ and integrating over
        $y \in \partial \Omega$, we obtain the statement.
\end{proof}

\subsection{Lower bound}

The argument for the lower bound follows much the same strategy as the
upper bound.  However, we of course have to ensure that the
computation is valid along a sequence of only weakly convergent
objects.  We first point out that the microscopic difference quotients
considered in the proof of Lemma
\ref{lemma:compactness_single_curve} in fact converge to the curvature
in a weak sense.

\begin{lemma}\label{lemma:weak_curvature}
  Under the assumptions of Lemma
    \ref{lemma:compactness_single_curve}, let $(\gamma_{n_k})$ be the
    subsequence and let $\gamma_\infty$ be its limit from the
    conclusion of Lemma \ref{lemma:compactness_single_curve}. Then, if
    $\kappa_\infty$ is the curvature of $\gamma_\infty$ we have
        	\begin{align}\label{eq:conv_anis_scaling}
                  \lambda_{n_k} \left[
                  \nu_{n_k}\left(t+\frac{s}{L(\gamma_{n_k})
                  \lambda_{n_k}}\right) - 
                  \nu_{n_k}(t)) \right]\warr   s
                  \kappa_\infty(t)
                  \frac{\gamma'_\infty(t)}{L(\gamma_\infty)}
	\end{align}
        as $k \to \infty$ in $\mathcal D'(\Sph^1; \R^2)$ for all
        $s\in \R$, as well as in $\mathcal D'(\Sph^1\times \R; \R^2)$.
\end{lemma}

For the proof of the lower bound proper, the convergence
\eqref{eq:conv_anis_scaling} ensures that the limiting object after
the anisotropic blowup is the expected parabola, while the second part
of Lemma \ref{lemma:discrete} allows us to work at points on which the
boundaries along the sequence are not too ill-behaved.  However, the
proof is somewhat heavy on standard, measure-theoretic
details. We recall the definition of $\hat F_{\infty,\sigma}$ in
  \eqref{eq:Fhat}.

\begin{prop}\label{prop:lower_bound}
  Let $\lambda_n \to \infty$ and $\alpha_n >\frac{1}{\sqrt{2\pi}}$ be
  sequences such that
  $\sigma_n = \lambda_n^2 \left( 1- \frac{1}{2\pi \alpha_n^2} \right)$
  satisfies
	\begin{align}
		\lim_{n\to \infty} \sigma_n = \sigma >0.
	\end{align}
	Let $\Omega_{n} \in \mathcal{A}_\pi$ for $n\in \N$ be
        regular sets such that
        $\chi_{n} \to \chi_{\Omega_\infty}$ in $L^1(\R^2)$ for
        $\Omega_\infty \in \mathcal{A}_\pi$ and such that 
	\begin{align}\label{eq:lower_bound_boundedness}
          \limsup_{n \to \infty} \lambda_n^2 F_{\lambda_n,
          \alpha_n} (\Omega_{n}) <\infty. 
	\end{align}
	Furthermore, let there exist $I \subset \N$ finite, and a
        family
        $\gamma_\infty := \{\gamma_{\infty,i}\}_{i\in I} \in
        G(\Omega_\infty)$ of $H^2$-regular, constant speed curves such
        that for all $i\in I$ we have
	\begin{align}
		\gamma_{n,i} \to \gamma_{\infty,i} \qquad \text{in} \
          H^1(\Sph^1; \R^2),
	\end{align}
        as $n \to \infty$, where $\gamma_{n,i}$ for $i \in I$ is some
      sub-collection of curves from the decomposition of $\Omega_{n}$
      into its boundary curves. Then 
	\begin{align}
		\hat F_{\infty,\sigma}(\gamma_\infty)  \leq \liminf_{n
          \to \infty} \lambda_n^2 F_{\lambda_n,\alpha_n} (\Omega_{\lambda_n}). 
	\end{align}
\end{prop}

\begin{proof}[Proof of Lemma \ref{lemma:weak_curvature}]
  Let $s\in \R$.  Let $\xi \in C^\infty(\Sph^1)$ be a smooth, periodic
  test function parametrized by $t \in [0,1]$.  By the strong
  convergence of $\nu_n$ to $\nu_\infty$ obtained in Lemma
  \ref{lemma:compactness_single_curve} with the help of identity
    \eqref{eq:chosen_normal}, we get
	\begin{align}
	  \begin{split}
            & \quad \lim_{n_k \to \infty} \int_0^1 \lambda_{n_k}
            \left[ \nu_{n_k}\left(t+\frac{s}{L(\gamma_{n_k})
                  \lambda_{n_k}}\right) -
              \nu_{n_k}(t)\right] \xi (t)\intd t \\
            & = \lim_{k \to \infty} \int_0^1 \nu_{n_k}(t)
            \lambda_{n_k} \left[\xi\left(t-\frac{s}{L(\gamma_{n_k})
                  \lambda_{n_k} }\right)-
              \xi(t) \right]  \intd t\\
            & = - \int_0^1 \nu_\infty(t) \frac{s}{L(\gamma_\infty)}
            \partial_t \xi(t)\intd t.
	  \end{split}
	\end{align}
      Together with the fact that
      $ \nu_\infty' = \kappa_\infty \gamma_{\infty}'$ a.e., see
      identity \eqref{eq:derivative_normal}, we get the first desired
      convergence in \eqref{eq:conv_anis_scaling}.  To obtain the
      second, simply repeat the above argument after testing with a
      function $\xi \in C^\infty(\Sph^1\times \R)$ with compact
      support in the second variable.
\end{proof}

\begin{proof}[Proof of Proposition \ref{prop:lower_bound}]
	\textit{Step 1: Choose appropriate subsequences.}
	
	We first choose a subsequence (not relabeled) such that
        we may apply Lemma \ref{lemma:compactness_single_curve} to the
        sequences $\gamma_{n,i}$ for all $i \in I$ and such that
	\begin{align}\label{eq:series_finite}
		\sum_{n\in \N} \lambda_n^{-\frac{1}{2}} <\infty,
	\end{align}
        the latter being chosen to be able to subsequently apply the
        Borel-Cantelli lemma in Steps 3 and 4.
	
	Let $i \in I$. Furthermore, let us abbreviate
        $L(\gamma_{n,i})$ by $L_{n,i}$, and recall that
        $|\gamma'_{n,i}| = L_{n,i}$ everywhere.  For all $K \in \N$
        and $t\in [0,1]$ we define
	\begin{align}\label{def_Z_disint}
          Z_{K,n,i,t}:=  \left\{ y \in \partial \Omega \cap
          B_{K\lambda_n^{-1}}( \gamma_{n,i}(t)) : \nu(y) \cdot
          \nu(\gamma_{n,i}(t)) \leq 0 \right\}. 
	\end{align}
	Due to the estimate \eqref{eq:weak_estimate}, Fubini's
        theorem, and estimate \eqref{eq:FlaP}, we have for all
        $K\in \mathbb{N}$ and $i \in I$ that
	\begin{align}
          L_{n,i}  \int_0^1 \mathcal{H}^1(Z_{K,n,i,t})
          \intd t \leq C_K P(\Omega_n) F_{\lambda_n,
          \alpha}(\Omega_n) \leq \frac{C_K}{\sigma_n} \lambda_n^2
          F^2_{\lambda_n, 
          \alpha}(\Omega_n),
	\end{align}
	for some $C_K > 0$ depending only on $K$ and all $n$
          sufficiently large. Therefore, after taking yet another
        subsequence, by \eqref{eq:lower_bound_boundedness} we
        have for almost all $t\in [0,1]$, all $K \in \mathbb{N}$, and
        all $i \in I$ that
	\begin{align}\label{eq:measure_converge}
          \lim_{n \to \infty} \lambda_n^\frac{3}{2}
          \mathcal{H}^1(Z_{K,n,i,t})  = 0,
	\end{align}
	where the exponent $\tfrac32 < 2$ was chosen as
        sufficient to complement estimate \eqref{eq:series_finite}
        when passing from estimate \eqref{eq:sum_H} to estimate
        \eqref{eq:some_series_finite} in what follows.
	
	Let $n \in \N$ and $i \in I$.  For $s \in \R$, let
        $T_{n,i} s := (L_{n,i} \lambda_n)^{-1} s$, so that, as in the
        proof of Lemma \ref{lem:expansion}, the variable $s$ plays the
        role of a microscopic arc length parameter.  In analogy with
        the definitions \eqref{def:g_upper} and \eqref{def:h_upper},
        for $t\in [0,1]$ and $s \in \R$ we define
        $\tau_{n,i}(t) := \gamma'_{n,i} (t) / L_{n,i}$ and the
        functions
	\begin{align}
          g_{n,i,t}(s) & := \lambda_n
                         \tau_{n,i}(t) \cdot
                         \left(\gamma_{n,i}(t+ T_{n,i} s ) -
                         \gamma_{n,i}(t )\right), \label{eq:def_g}\\ 
          h_{n,i,t}(s) & := \lambda_n^2 \nu_{n,i}(t) \cdot
                         \left(\gamma_{n,i}(t +T_{n,i} s ) -
                         \gamma_{n,i}(t )\right), \label{eq:def_h}
	\end{align}
        giving
        \begin{align}\label{eq:g_short}
        	|g_{n,i,t}'(s)| \leq 1.
	\end{align}
	Then for $s\in \R$, we have
	\begin{align}
	  \begin{split}
            & \quad  \int_0^1 \left( 1 - g_{n,i,t}'(s) \right) \intd t\\
            & = \frac{1}{2L_{n,i}^2} \int_0^1
            \left(|\gamma'_{n,i}(t)|^2+ |\gamma'_{n,i}(t +
              T_{n,i}s)|^2 - 2\gamma_{n,i}'(t) \cdot \gamma'_{n,i}(t
              + T_{n,i}s )\right) \intd t \\
            & = \frac{1}{2L_{n,i}^2} \int_{0}^1 | \gamma_{n,i}' (t +
            T_{n,i}s) - \gamma_{n,i}' (t)|^2 \intd t,
	  \end{split}
	\end{align}
	so that together with \eqref{eq:lower_bound_boundedness}
        the bound \eqref{eq:norms_boundedness} implies
	\begin{align}\label{eq:gprime}
          \lim_{n \to \infty} \lambda_n^\frac{1}{2}  \int_0^1 \left( 1
          - g_{n,i,t}'(s) \right) \intd t = 0
	\end{align}
	locally uniformly in $s$, where the exponent $\tfrac12 <
          2$ is chosen to again complement estimate
          \eqref{eq:series_finite} to arrive at
          \eqref{eq:another_fucking_Borel_Cantelli} in what follows.

	Let $K \in \N$.  Using
        $g_{n,i,t}(0) =0$, integrating in $s$ and using Fubini's
        theorem, we obtain
	\begin{align}
          \lambda_n^\frac{1}{2} \int_0^1 \max_{s \in [-K,K]} | s -
          g_{n,i,t}(s) | \intd t \to 0 
	\end{align}
	in the limit $n \to \infty$.  Passing to a subsequence, for
        all $i \in I$ and almost all $t \in [0,1]$ we get
	\begin{align}\label{eq:conv_tangents}
		\lambda_n^\frac{1}{2} \max_{s \in [-K,K]} | s - g_{n,i,t}(s) | \to 0,
	\end{align}
	and together with the summability \eqref{eq:series_finite}
        that 
	\begin{align}\label{eq:another_fucking_Borel_Cantelli}
		\sum_{n\in \N}  \max_{s \in [-K,K]} | s - g_{n,i,t}(s) | < \infty.
	\end{align}	

	Additionally, for $t\in [0,1]$ and all $s \in \R$ we
        calculate
	  \begin{align}
           \begin{split}
             h'_{n,i,t}(s)& = \frac{\lambda_n}{L_{n,i}} \nu_{n,i}(t)
             \cdot \gamma'_{n,i}(t +T_{n,i} s )
              = -\lambda_n \frac{ \gamma_{n,i}'(t)}{L_{n,i}} \cdot (
             \nu_{\lambda_n}(t +T_{n,i} s) - \nu_{n,i}(t))
           \end{split}
	\end{align}
        Consequently, by \eqref{eq:lower_bound_boundedness} and
        Lemma \ref{lemma:discrete} we have
	\begin{align}
          \sup_{n\in \N} \sup_{s\in [-K,K]} \int_{0}^{1}
          |h'_{n,i,t}(s)|^2 \intd t < \infty, 
	\end{align}
	and we get from convergence \eqref{eq:conv_anis_scaling}, a
        weak-times-strong argument, and $|\gamma_{n,i}'| = L_{n,i}$
        that
         \begin{align}\label{eq:conv_double}
           \begin{split}
             h'_{n,i,t}(s)            \warr- s
             \kappa_{\infty,i}(t)
           \end{split}
	\end{align}
	in $L_t^2(0,1)$ for all $s\in \R$, as well as in
        $L_{t,s}^2((0,1)\times (-K,K) )$. Here and in the following,
        the subscripts in the notation for Lebesgue spaces denote the
        variables in which the integration is performed.
   
	Again, using $h_{n,i,t}(0) = 0$ we may apply the
        fundamental theorem of calculus in $s$ and Jensen's inequality
        to obtain 
	\begin{align}\label{eq:est_h}
	\begin{split}
          \sup_{n\in \N} \int_{0}^{1} \sup_{s \in [-K,K]}
          \frac{|h_{n,i,t}(s)|^2}{|s|} \intd t & \leq \sup_{n\in \N}
          \int_{0}^{1} \int_{-K}^{K}
          |h'_{n,i,t}(s)|^2 \intd s  \intd t \\
          & \leq 2 K \sup_{n\in \N}\sup_{s\in [-K,K]}
          \int_{0}^{1}  |h'_{n,i,t}(s)|^2 \intd t \\
          & < \infty.
         \end{split}
	\end{align}
	For all $s \in \R$ and $\xi \in L^2(0,1)$, we also get  
	\begin{align}
          \int_0^1 h_{n,i,t}(s)\xi(t) \intd t = \int_0^s
          \int_0^1 h'_{n,i,t}(s')\xi(t) \intd t \intd s' \to -
          \int_0^1  \frac{s^2}{2} 
          \kappa_{\infty,i}(t) \xi(t) \intd t, 
	\end{align}
	from the weak $L_{t,s}^2((0,1)\times (-K,K) )$
        convergence \eqref{eq:conv_double}.  Therefore,
        for all $s \in \R$ we have
	\begin{align}\label{eq:h_weak}
		 h_{n,i,t}(s)\warr - \frac{s^2}{2}
          \kappa_{\infty,i}(t) 
	\end{align}
	in $L_t^2(0,1)$.  Furthermore, estimate \eqref{eq:est_h}
        implies for all $i \in I$ and $K \in \N$ that
	\begin{align}\label{eq:normal_parts_bound}
          \lambda_n^{-\frac{1}{4}} \max_{s\in [-K,K]} |h_{n,i,t}(s)| \to 0
	\end{align}
	in $L_t^2 (0,1)$ as $n \to \infty$ and thus also for almost
        all $t\in [0,1]$ after passage to one more, final
        subsequence. Here, we chose the exponent $ - \frac{1}{4} < 0$,
        so that we can later deduce the convergence
        \eqref{eq:random_exponent_explained} from the estimate
        \eqref{eq:gprime}.

        \textit{Step 2: Given $i\in I$, identify sets over which the
          curve $\gamma_{n,i}$ is locally a graph over its
          tangent space for sufficiently large $n \in \N$.}

	Let $i \in I$ and $K \in \N$ be fixed throughout this step of
        the proof.
	
        For every $n \in \N$ and $t\in [0,1]$, we consider the
        sets
	\begin{align}\label{eq:defS1}
          S^1_{K,n,i,t}& := \left\{ s \in [-K,K]: |g_{n,i,t}(s)| \geq  K^{-1}\right\},\\
          S^2_{K,n,i,t} & := \left\{ s  \in  [-K,K]:
                          |h_{n,i,t}(s)| \leq
                          \lambda_n^{\frac{1}{4}}\right\}. \label{eq:defS2} 
	\end{align}
	The set $S^1_{K,n,i,t}$ cuts away the origin, while the set
        $S^2_{K,n,i,t}$ makes sure that $h_{n,i,t}$ is not too large,
        see Figure \ref{fig:rough_1}.
        
        	\begin{figure}
		\centering
		\includegraphics{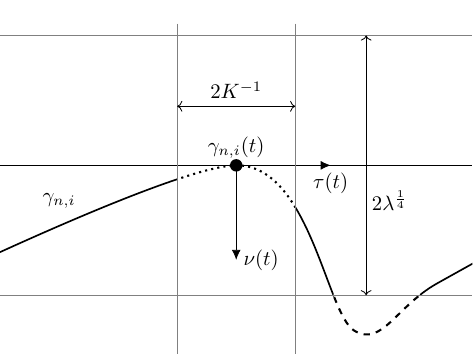}
		\caption{\label{fig:rough_1} Sketch of
                  $\gamma_{n,i}(t+T_\lambda s)$ on
                  $ S^1_{K,n,i,t} \cap S^2_{K,n,i,t}$. The dotted part
                  of $\gamma_{n,i}$ is outside $S^1_{K,n,i,t}$, the
                  dashed part is outside of $S^2_{K,n,i,t}$. }
               	\end{figure}	

                We will want to consider $g_{n,i,t}(s)$ as a
                parametrization for $\gamma_{n,i}$ around its own
                tangent line at $\gamma_{n,i}(t)$. However, there is
                no reason why it should be injective, see Figure
                \ref{fig:rough_2}.  Somewhat abusing notation, we
                therefore define the generalized inverse of
                $g_{n,i,t}$ for $n \in \N$, $t \in [0,1]$, and
                $z_1 \in g_{n,i,t} ([-K,K])$ as
        	\begin{align}
		g^{-1}_{n,i,t}(z_1) :=
		\begin{cases}
                  \inf \{ s' \in [0,K]: g_{n,i,t}(s')\geq z_1 \} &
                  \text{ if } z_1\geq0, \\ 
                  \sup\{ s' \in [-K,0] : g_{n,i,t}(s') \leq z_1 \}
                  & \text{ if } z_1 < 0
		\end{cases}
	\end{align}
	to be left-continuous for positive $s$ and right-continuous
        for negative $s$.  By continuity of $g_{n,i,t}$, we indeed
        have $g_{n,i,t} \circ g^{-1}_{n,i,t}(z_1) = z_1$ for all
        $n \in \N$, $t\in [0,1]$, and
        $z_1 \in g_{n,i,t} ([-K,K])$.  While we may not have
        $g_{n,i,t}^{-1} \circ g_{n,i,t}(s) = s$ for all $s\in [-K,K]$,
        this does hold for all $n \in \N$, $t\in [0,1]$ on
        \begin{align}
          \begin{split}
            S^3_{K,n,i,t} & := \left\{ s \in [0,K]: g_{n,i,t}( s') <
              g_{n,i,t}(s)\, \forall s' \in
              [0,s)\right\} \\
            & \qquad \cup \left\{ s \in [-K,0]: g_{n,i,t}( s') >
              g_{n,i,t}(s)\, \forall s' \in ( s,0]\right\},
         \end{split}
	\end{align}
	by construction. As a result of $g_{n,i,t}(0) = 0$, we have
        $\mathrm{sgn} (g_{n,i,t}(s)) = \mathrm{sgn}(s)$ for all
        $s \in S^3_{K,n,i,t}$.
	
	  	\begin{figure}
		\centering
		\includegraphics{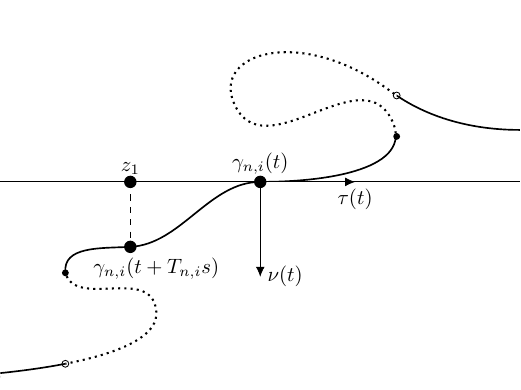}
		\caption{\label{fig:rough_2} Sketch of
                  $\gamma_{n,i}(t+T_\lambda s)$ on $
                  S^3_{K,n,i,t}$. Parts of $\gamma_{n,i}$ outside of
                  $S^3_{K,n,i,t}$ are shown as dotted. The solid
                  points signify the endpoints contained in
                  $S^3_{K,n,i,t}$. For $z_1 \in \R$, the inverse
                  $g_{n,i,t}^{-1}(z_1)$ is the microscopic arc length
                  parameter $s \in S^3_{K,n,i,t}$ such that
                  $g_{n,i,t}(s) = z_1$. }
	\end{figure}

	Finally, when comparing $\Omega$ with its tangent half-plane,
        we have to contend with the possibility of small holes in
        $\Omega$ or small pieces of $\Omega$ lying between
        $\gamma_{n,i}$ and its tangent line at $t$, see Figure
        \ref{fig:rough_3}. Therefore, we also consider the
        parametrized line segment
         \begin{align}\label{eq:line}
           l_{n,i,t,s}(r) :=  g_{n,i,t}(s) e_1 + r h_{n,i,t}(s) e_2
        \end{align}
        for $r\in [0,1]$ and
         the set
	\begin{align}
	\begin{split}\label{eq:s_4}
          S^4_{K,n,i,t}& := \Big\{ s \in[-K,K]: \forall r \in (0,1) : \\
          & \qquad \qquad l_{n,i,t,s}(r) \in
          \left(\overline{H_-^0(e_2)}\setminus A_{\lambda_n}
            R_{\nu_{n,i} (t) } \left(\Omega_{n} -
              \gamma_{n,i}(t )\right) \right) \\
          & \qquad \qquad \qquad \qquad \qquad \cup
          \left(A_{\lambda_n} R_{\nu_{n,i} (t) }
            \left(\overline{\Omega_{n}} - \gamma_{n,i}(t
              )\right)\setminus H_-^0(e_2) \right) \Big\},
	  \end{split}
	\end{align}
	which rules out this pathological behaviour. Notice that
          by definition we have
          $l_{n,i,t,s}(0) \in \partial H_-^0(e_2)$, while
          $l_{n,i,t,s}(1) \in A_{\lambda_n} R_{\nu_{n,i}(t)} (\partial
          \Omega_n - \gamma_{n,i}(t))$.
		
		  	\begin{figure}
		\centering
		\includegraphics{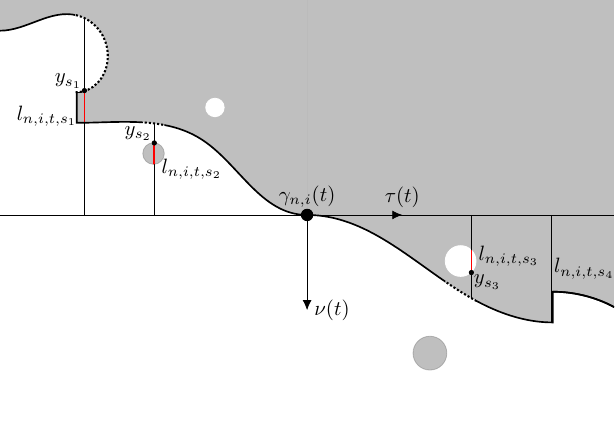}
		\caption{\label{fig:rough_3} Sketch of
                  $\gamma_{n,i}(t+T_\lambda s)$ on $
                  S^4_{K,n,i,t}$. The parts of $\gamma_{n,i}$ not
                  along $ S^4_{K,n,i,t}$ are shown as dotted, the set
                  $\Omega$ is shown in gray. We have
                  $0,s_4 \in S^4_{K,n,i,t}$, but
                  $s_1,s_2,s_3 \not \in S^4_{K,n,i,t}$. The
                  corresponding sets $D_{s_1}$, $D_{s_2}$ and
                  $D_{s_3}$ are shown as red lines, while
                  $y_{s_1}$, $y_{s_2}$ and $y_{s_3}$ are shown as
                  small dots. }
	\end{figure}

	Let
	\begin{align}
		S_{K,n,i,t} & := S^1_{K,n,i,t} \cap S^2_{K,n,i,t} \cap S^3_{K,n,i,t}\cap S^4_{K,n,i,t}.
	\end{align}
	Additionally, we define the sets $ S_{K,\infty,i,t}$ and
        $S^j_{K,\infty,i,t}$ for $j =1,2,3,4$ as
	\begin{align}\label{def:S}
          S_{K,\infty,i,t}& := \bigcup_{n \in \N} \bigcap_{n' \in \N: n' >n}  S_{K,n',i,t},\\
          S^j_{K,\infty,i,t}& := \bigcup_{n \in \N} \bigcap_{n' \in
                              \N: n' >n}
                              S^j_{K,n',i,t}, \label{def:S_j} 
	\end{align}
	being the sets of points that for sufficiently large $n$ lie
        in all sets $S_{K,n',i,t}$, resp., $S^j_{K,n',i,t}$ for
        $n' \geq n$.  and observe the decomposition
	\begin{align}\label{eq:decomp_s}
          S_{K,\infty,i,t} = S^1_{K,\infty,i,t} \cap
          S^2_{K,\infty,i,t} \cap S^3_{K,\infty,i,t} \cap
          S^4_{K,\infty,i,t}. 
	\end{align}
	
	The convergence \eqref{eq:conv_tangents} states that for all
        $i\in I$, almost all $t\in [0,1]$, and all $s\in [-K,K]$,
        we have $g_{n,i,t}(s) \to s$ as $n \to \infty$.  Therefore,
        for such $t\in [0,1]$, all $s \in S^1_{K,\infty,i,t}$
        satisfy $s \in S^1_{K,n,i,t}$ for $n\in \N$ sufficiently
        large, giving
        $|s| = \lim_{n \to \infty} |g_{n,i,t}(s)| \geq K^{-1}$.
        Similarly, if $|s| > K^{-1}$, then we have
        $|g_{n,i,t}(s)| > K^{-1}$ for $n \in \N$ sufficiently large,
        giving $s \in S^1_{K,\infty,i,t}$.  Consequently, we have
	\begin{align}\label{eq:S1_full}
          |S^1_{K,\infty,i,t} \Delta \left( [-K,K] \setminus
          [-K^{-1},K^{-1}]\right)| = 0 
	\end{align}
	for almost all $t \in [0,1]$.  By the convergence
        \eqref{eq:normal_parts_bound}, for almost all $t \in [0,1]$ we
        have
	\begin{align}\label{eq:S2_full}
	  \begin{split}
            & \quad[-K,K] \setminus S^2_{K,\infty,i,t} \\\
            & = \bigcap_{n>0} \bigcup_{n' \in \N: n' >n} \left\{ s \in
              [-K,K]: \lambda_{n'}^{-\frac{1}{4}}|h_{n',i,t}|(s) >
              1 \right\}  \\
            &= \emptyset.
	  \end{split}
	\end{align}

	\textit{Step 3: For all $K \in \N$ and $i \in I$, we prove
          $|[-K,K]\setminus S^3_{K,\infty,i,s}| =0$, i.e., we may
          invert $g_{n,i,t}$ almost everywhere for sufficiently large
          $n$.}

	Let $n \in \N$ and $t \in [0,1]$.  We decompose
        $[-K, K] \setminus {S^3_{K,n,i,t}}$ into its at most countably
        many connected components in the following way: We claim that
        there exists a set $J \subset \N$ and $0 < a_j < b_j \leq K$
        for $j \in J$ such that the sets $(a_j,b_j]$ are pairwise
        disjoint and
	\begin{align}
          \label{eq:S3decomp}
	 	 [0,K] \setminus S^3_{K,n,i,t} = \bigcup_{j\in J} (a_j, b_j].
	\end{align}
	Here, the half-open intervals are a result of $g_{n,i,t}(s)$
        being left-continuous for $s \geq 0$ and the definition of
        $S^3_{K,n,i,t}$.

	Let $s \in [0,K] \setminus S^3_{K,n,i,t} $.
	Let 
	\begin{align}
          a_s & := \min\left\{  s' \in [0,s] : g_{n,i,t}( s') =
                \max_{[0,s]} g_{n,i,t}\right\}, 
	\end{align}
	which exists by continuity of $g_{n,i,t}$.  Furthermore, let
	\begin{align}
          b_s & := \max\left\{  s' \in [s,K] : g_{n,i,t}( r)
                \leq  \max_{[0,s]} g_{n,i,t}\,   \forall r \in
                [s,s'] \right\},
	\end{align}
	which also exists by continuity of $g_{n,i,t}$.  By definition
        of $S^3_{K,n,i,t}$, we have $a_s < s \leq b_s$ and
        $a_s \in S^3_{K,n,i,t}$.  For all $\tilde s \in (a_s,s]$ we
        have $ g_{n,i,t}(\tilde s) \leq g_{n,i,t}(a_s)$ by definition
        of $a_s$, while for $\tilde s \in (s,b_s]$ the same holds by
        definition of $b_s$. Thus, for all
          $\tilde s \in (a_s,b_s]$ we have
          $\tilde s \not \in S^3_{K,n,i,t}$ and
	\begin{align}\label{eq:bs}
		 \max_{[0,s]} g_{n,i,t} = g_{n,i,t}( a_s) =
          \max_{[0,\tilde s]} g_{n,i,t}.
	\end{align}
	Therefore, we have
	\begin{align}
		a_s & = \min\left\{  s' \in [0,\tilde s] : g_{n,i,t}(
                      s') =  \max_{[0,\tilde s]} g_{n,i,t}\right\} =
                      a_{\tilde s}. 
	\end{align}
	In particular, applying the identity \eqref{eq:bs} for $\tilde
        s=b_s$ gives $g_{n,i,t}( a_s) = \max_{[0,b_s]} g_{n,i,t}$.
        Conversely, by definition of $b_s$ for every $\eps>0$ there
        exists $r\in (b_s,b_s+\eps)$ such that we have
        $g_{n,i,t}(r) > \max_{[0,b_s]} g_{n,i,t}$.  Consequently, we
        have
	\begin{align}
          b_s & = \max\left\{  s' \in [\tilde s,K] : g_{n,i,t}( r)
                \leq  \max_{[0,\tilde s]} g_{n,i,t} \, \forall r \in
                [\tilde s ,s'] \right\}.
	\end{align}
	Thus, the sets $(a_s, b_s]$ for
        $s \in [0,K] \setminus S^3_{K,n,i,t}$ provide the
          decomposition of $[0,K] \setminus S^3_{K,n,i,t}$ into
        connected components.  As each half-open interval must contain
        at least one rational number, there may be at most
        countably many pairwise disjoint, connected
        components. This proves the claim, yielding the
          decomposition in \eqref{eq:S3decomp}.
	
          By construction, for all $j \in J$ we have
          $g_{n,i,t}(b_j) = g_{n,i,t}(a_j)$, unless $b_j = K$, in
          which case we only have
          $g_{n,i,t}(b_j) \leq g_{n,i,t}(a_j)$.  Therefore, we have
	\begin{align}
		| [a_j,b_j] | \leq  b_j  - a_j + g_{n,i,t}(a_j) - g_{n,i,t}(b_j).
	\end{align}
	In total, we consequently get
	\begin{align}
		\left|  [0,K] \setminus S^3_{K,n,i,t}\right| \leq
          \sum_{j\in J} \left| \left( b_j - g_{n,i,t}(b_j) - \left(
          a_j - g_{n,i,t}(a_j)\right) \right) \right|. 
	\end{align}
	By an approximation argument using $\tilde J \subset J$
        finite, we get
	\begin{align}
          \left|  [0,K] \setminus S^3_{K,n,i,t} \right|  \leq
          TV_{[0,K]}(s -g_{n,i,t}(s)),  
	\end{align}
	where the latter is the one-dimensional total variation of the
        function $s \mapsto s -g_{n,i,t}(s)$ for $s \in [0,K]$.  Due
        to
        \eqref{eq:g_short}, this map is non-decreasing, and therefore
	\begin{align}
          \left|   [0,K] \setminus S^3_{K,n,i,t}\right|  \leq
         K - g_{n,i,t}(K). 
	\end{align}
	
	An analogous argument works for
        $[-K,0] \setminus S^3_{K,n,i,t}$, giving in total
	\begin{align}
          \left| [-K,K] \setminus S^3_{K,n,i,t} \right| \leq 2 K
          - g_{n,i,t}(K) + g_{n,i,t}(-K). 
	\end{align}
	Together with the summability of this expression that
          follows from \eqref{eq:another_fucking_Borel_Cantelli}, the
        Borel-Cantelli lemma implies
	\begin{align}\label{eq:S3_full}
          |[-K,K]\setminus S^3_{K,\infty,i,s}| =0.
	\end{align}
	
	\textit{Step 4: For all $K \in \N$ and $i \in I$, we prove
          that $S^4_{K,\infty,i,t}$ has full measure.}
	
        For $t\in [0,1]$, let
        $s \in S^2_{K,n,i,t} \setminus S^4_{K,n,i,t}$ for some
        $n \in \N$.  In particular, we have $h_{n,i,t}(s) \neq 0$.  As
        the argument in this step is most naturally done in the
        original coordinates, we define
        \begin{align}\label{eq:repr_line}
        \begin{split}
          \tilde l_{n,i,t,s}(r) & := \gamma_{n,i}(t) +
          R_{\nu_{n,i}(t)}^T
          A_{\lambda_n}^{-1} l_{n,i,t,s}(r) \\
          & = \gamma_{n,i}(t) + \lambda_n^{-1} g_{n,i,t}(s)
          \tau_{n,i}(t) + r \lambda_n^{-2} h_{n,i,t}(s) \nu_{n,i}(t).
	\end{split}
      \end{align} 
      for $r\in [0,1]$. 
      Recalling the definitions \eqref{eq:s_4} and
      \eqref{def:half-space}, the set
        \begin{align}
        \begin{split}
          D_s & := \Big\{ r \in (0,1) : l_{n,i,t,s}(r) \not \in
          \left(\overline{H_-^0(e_2)}\setminus A_{\lambda_n}
            R_{\nu_{n,i} (t) } \left(\Omega_{n} -
              \gamma_{n,i}(t )\right) \right) \\
          & \qquad \qquad \qquad \qquad \qquad \qquad \cup
          \left(A_{\lambda_n} R_{\nu_{n,i} (t) }
            \left(\overline{\Omega_{n}} - \gamma_{n,i}(t
              )\right)\setminus H_-^0(e_2) \right) \Big\},
	  \end{split}
      	\end{align}
	see Figure \ref{fig:rough_3}, is non-empty.  Therefore, we
        have
          \begin{align}
            r_{\mathrm{max}, s} & :=  \sup D_{s} \in (0, 1].
         \end{align}
         
         By
         $ \left(\tilde l_{n,i,t,s}(r) - \gamma_{n,i}(t) \right) \cdot
         \nu_{n,i}(t) = r \lambda^{-2}_n h_{n,i,t}(s)$, for all
         $r \in (0,1]$ we have
        \begin{align}\label{eq:signs_1}
        	\operatorname{sgn}  \left(\tilde l_{n,i,t,s}(r) -
          \gamma_{n,i}(t) \right) \cdot \nu_{n,i}(t) =
          \operatorname{sgn} h_{n,i,t}(s) \neq 0.
	\end{align}
	Consequently, for all $r \in (0,1]$ we have either
        $\tilde l_{n,i,t,s}(r) \in H_-(\gamma_{n,i}(t))$ or
        $\tilde l_{n,i,t,s}(r) \in \R^2 \setminus
        \overline{H_-(\gamma_{n,i}(t))}$.  Thus if
        $ r_{\mathrm{max},s} < 1$, we have
        $y_{s} := \tilde l_{n,i,t,s}( r_{\mathrm{max},s} ) \in
        \partial \Omega_{n} $.  If $ r_{\mathrm{max},s} = 1$, by the
        representation \eqref{eq:repr_line} and the definitions
        \eqref{eq:def_g} and \eqref{eq:def_h} we have
        $ \tilde l_{n,i,t,s}( r_{\mathrm{max},s}) \in \partial
        \Omega_n$.  At the same time, a direct computation gives
       \begin{align}
         \partial_r \left( \nu(y_{s}) \cdot \tilde l_{n,i,t,s}(r)\right)
         = \lambda_n^{-2} \nu(y_{s})\cdot \nu_{n,i}(t) h_{n,i,t}(s). 
       \end{align}

       Let us consider the case
       $\tilde l_{n,i,t,s}(r) \in H_-(\gamma_{n,i}(t))$ for all
       $r\in (0,1]$.  Then by identity \eqref{eq:signs_1}, we have
       $h_{n,i,t}(s) <0$, so that
       \begin{align}\label{eq:signs_something}
         \operatorname{sgn} 	\partial_r \left( \nu(y_{s}) \cdot
         \tilde l_{n,i,t,s}(r)\right) = - 	\operatorname{sgn}
         \nu(y_{s})\cdot \nu_{n,i}(t). 
       \end{align}
       For $r \in D_{s}$, we furthermore get
       $\tilde l_{n,i,t,s}(r) \in \Omega_n$.  As $\Omega_n$ is regular
       and $\nu(y_{s})$ is the outer unit normal to $\Omega_n$, there
       exists $\eps>0$ such that
       $\tilde l_{n,i,t,s}(r) \in \overline{H_-(y_{s})}$ for
       $r \in (r_{\mathrm{max},s}-\eps, r_{\mathrm{max},s})$.
       Consequently, linearity of $\tilde l_{n,i,t,s}(r)$ as a
       function of $r$ gives
       \begin{align}
         \partial_r \left( \nu(y_{s}) \cdot \tilde l_{n,i,t,s}(r)\right) \geq 0.
       \end{align}
       so that \eqref{eq:signs_something} implies
       $\nu(y_{s})\cdot \nu_{n,i}(t) \leq 0$.
       
       In the case
       $\tilde l_{n,i,t,s}(r) \in \R^2 \setminus
       \overline{H_-(\gamma_{n,i}(t))}$ for all $r\in (0,1]$, we
       instead have $h_{n,i,t}(s) >0$ and
        \begin{align}\label{eq:signs_something_2}
          \operatorname{sgn} 	\partial_r \left( \nu(y_{s}) \cdot \tilde
          l_{n,i,t,s}(r)\right) =  	\operatorname{sgn}
          \nu(y_{s})\cdot \nu_{n,i}(t). 
       \end{align}
       Additionally, we have
       $\tilde l_{n,i,t,s}(r) \not \in \overline{\Omega_n}$ for all
       $r \in D_s$.  Therefore, we have
       $\tilde l_{n,i,t,s}(r) \not \in H_-(y_{s})$ for all
       $r \in (r_{\mathrm{max},s}-\eps, r_{\mathrm{max},s})$ and
       $\eps>0$ small enough.  As a result, we also have
       \begin{align}
         \partial_r \left( \nu(y_{s}) \cdot \tilde l_{n,i,t,s}(r)\right) \leq 0,
       \end{align}
       resulting in $\nu(y_{s})\cdot \nu_{n,i}(t) \leq 0$.

       Furthermore, as a result of
       \eqref{eq:repr_line}, \eqref{eq:g_short} and
       $s \in S^2_{K,n,i,t}$ we have
        \begin{align}
        	 |y - \gamma_{n,i}(t)| \leq 
          \lambda_n^{-1} |s| + \lambda_n^{-2} |h_{n,i,t}(s)| \leq  K \lambda_n^{-1} +
          \lambda_n^{-\frac{7}{4}}. 
        \end{align}
        For $n \in \N$ sufficiently big, we thus have
        $y \in Z_{2K,n,i,t}$, see the definition in
        \eqref{def_Z_disint}.  Additionally, representation
        \eqref{eq:repr_line} implies
        \begin{align}
        		p(y) = g_{n,i,t}(s),
        \end{align}
        where
        $p(y) := \lambda_n \left(y - \gamma_{n,i}(t) \right) \cdot
        \tau_{n,i}(t)$ for $y \in \R^2$.
    
        The above can therefore be compiled into the statement
        \begin{align}
          g_{n,i,t}\left(S^2_{K,n,i,t}\setminus S^4_{K,n,i,t} \right)
          \subset p(Z_{2K,n,i,t}), 
        \end{align}
        so that $p$ being a $\lambda_n$-Lipschitz map ensures
        \begin{align}\label{eq:rewritten}
          \left| g_{n,i,t} \left(S^2_{K,n,i,t}\setminus S^4_{K,n,i,t}
          \right) \right|   \leq |p(Z_{2K,n,i,t})|  \leq \lambda_n
          \mathcal{H}^1(Z_{2K,n,i,t}). 
        \end{align}
        As $g_{K,n,i,t}$ is injective on $S^3_{K,n,i,t}$, the
        coarea-formula thus gives
        \begin{align}
          \label{eq:S234H1}
          \begin{split}
            \int_{\left(S^2_{K,n,i,t} \cap S^3_{K,n,i,t} \right)
              \setminus S^4_{K,n,i,t} } \left |g_{n,i,t}' \right|(s)
            \intd s & \leq \left| g_{n,i,t} \left(\left(S^2_{K,n,i,t}
                  \cap S^3_{K,n,i,t} \right) \setminus S^4_{K,n,i,t}
              \right) \right| \\
            & \leq \lambda_n \mathcal{H}^1(Z_{2K,n,i,t}).
	  \end{split}
        \end{align}
        
        Let
        $A_{K,n,i,t}:= \{ s \in [-K,K] : |g_{n,i,t}'(s)| \leq
        \frac{1}{2K} \}$. Then from \eqref{eq:S234H1} we have the
      estimate
        \begin{align}\label{eq:sum_H}
          \begin{split}
            \left| \left(S^2_{K,n,i,t} \cap S^3_{K,n,i,t} \right)
              \setminus S^4_{K,n,i,t} \right| & = \left| \left(
                \left(S^2_{K,n,i,t} \cap S^3_{K,n,i,t} \right)
                \setminus S^4_{K,n,i,t} \right) \setminus A_{K,n,i,t}\right| \\
            & \qquad + \left| \left( \left(S^2_{K,n,i,t} \cap
                  S^3_{K,n,i,t} \right)
                \setminus S^4_{K,n,i,t} \right) \cap A_{K,n,i,t} \right|  \\
            & \leq 2K \int_{\left( \left(S^2_{K,n,i,t} \cap
                  S^3_{K,n,i,t} \right) \setminus S^4_{K,n,i,t}
              \right) \setminus A_{K,n,i,t} } \left |g_{n,i,t}'
            \right|(s) \intd s \\
            & \qquad + \left| \left( \left(S^2_{K,n,i,t} \cap
                  S^3_{K,n,i,t} \right)
                \setminus S^4_{K,n,i,t} \right) \cap A_{K,n,i,t} \right| \\
            & \leq 2 K \lambda_n \mathcal{H}^1(Z_{2K,n,i,t}) +
            |A_{K,n,i,t}|.
         \end{split}
        \end{align}        
        On the other hand, using \eqref{eq:g_short} we may
        compute
        \begin{align}
          \left( 1- {1 \over 2 K} \right) |A_{K,n,i,t}| \leq
          \int_{A_{K,n,i,t}} \left( 1- 
          g_{n,i,t}'(s) \right) \intd s \leq \max_{s \in [-K,K]}
          | s - g_{n,i,t}(s)| . 
        \end{align}
        Therefore, combining this with estimates
          \eqref{eq:sum_H}, \eqref{eq:series_finite},
        \eqref{eq:measure_converge}, and
        \eqref{eq:another_fucking_Borel_Cantelli} gives
        \begin{align}\label{eq:some_series_finite}
          \sum_{n\in \N} \left| \left(S^2_{K,n,i,t} \cap
          S^3_{K,n,i,t} \right) \setminus S^4_{K,n,i,t}\right| <
          \infty. 
        \end{align}
       Again, we employ the Borel-Cantelli lemma to get
       \begin{align}
         \label{eq:BCS234}
          \left| \bigcap_{n\in \N }\bigcup_{n'\in \N: n' >n }
          \left(S^2_{K,n',i,t} \cap  S^3_{K,n',i,t} \right) \setminus
          S^4_{K,n',i,t}  \right| =0.  
	\end{align}
        
	If
        $s \in \left( S^2_{K,\infty,i,t} \cap S^3_{K,\infty,i,t}
        \right) \setminus S^4_{K,\infty,i,t} $, by definition
        \eqref{def:S_j} we have
        $s \in S^2_{K,n',i,t} \cap S^3_{K,n',i,t}$ for sufficiently
        big $n \in \N$ and all $n' \in \N$, $n'\geq
        n$. Conversely, for all $n \in \N$ there exists
        $n' \in \N$ with $n' \geq n$ such that
        $s \not\in S^4_{K,n',i,t}$. In particular, for all $n \in \N$
        sufficiently big there exists $n' \in \N$ with
        $n' \geq n$ such that
        $ s\in \left(S^2_{K,n',i,t} \cap S^3_{K,n',i,t} \right)
        \setminus S^4_{K,n',i,t}$. Therefore, from \eqref{eq:BCS234}
        we obtain
	\begin{align}\label{eq:remove_s_4}
		\left|  \left( S^2_{K,\infty,i,t} \cap
          S^3_{K,\infty,i,t} \right)  \setminus S^4_{K,\infty,i,t}
          \right| =0. 
	\end{align}
	
	Finally, computing the set inclusions 
	\begin{align}
	  \begin{split}
            & \quad S_{K,\infty,i,t} \Delta \left( [-K,K] \setminus
              [-K^{-1},K^{-1}]\right) \\ 
            & \subset S^1_{K,\infty,i,t}  \setminus  \left( [-K,K]
              \setminus [-K^{-1},K^{-1}]\right) \cup \bigcup_{j=1}^4
            \left( [-K,K] \setminus [-K^{-1},K^{-1}]\right) \setminus
            S^j_{K,\infty,i,t}\\ 
            & \subset S^1_{K,\infty,i,t}  \Delta  \left( [-K,K]
              \setminus [-K^{-1},K^{-1}]\right) \\ 
            & \qquad \cup \left( [-K,K] \setminus
              S^2_{K,\infty,i,t}\right) \cup \left( [-K,K] \setminus
              S^3_{K,\infty,i,t}\right) \\ 
            & \qquad \cup \left( S^2_{K,\infty,i,t} \cap
              S^3_{K,\infty,i,t} \right) \setminus S^4_{K,\infty,i,t}, 
	  \end{split}
	\end{align}
	we get from the identities \eqref{eq:decomp_s},
        \eqref{eq:S1_full}, \eqref{eq:S2_full}, \eqref{eq:S3_full},
        and \eqref{eq:remove_s_4} that
	\begin{align}\label{eq:S_total}
          \left|S_{K,\infty,i,t} \Delta \left( [-K,K] \setminus
          [-K^{-1},K^{-1}]\right)\right| = 0. 
	\end{align}

\textit{Step 5: For all $i \in I$, prove
	\begin{align}\label{eq:lower_bound_one_curve}
	  \begin{split}
            & \quad \liminf_{n \to \infty} \int_0^1 \int_{ H^0_-(e_2)
              \Delta A_{\lambda_n} R_{\nu_{n,i}(t)} (\Omega_n -
              \gamma_{n,i}(t)) } \left| z_2 \right|
            \frac{e^{-\alpha_n\sqrt{z_1^2 + \frac{z_2^2}{\lambda_n^2}
                }}}{z_1^2 + \frac{z_2^2}{\lambda_n^2}} \intd^2 z \intd t\\
            & \geq 2^\frac{1}{2} \pi^\frac{3}{2} \int_0^1
            \kappa^2_{\infty,i}(t) \intd t.
	  \end{split}
	\end{align}}
      Let $t\in [0,1]$, $n\in \N$, and $K \in \N$.
      By the properties of
      $S^3_{K,n,i,t}$ and $S^4_{K,n,i,t}$ , as well as the fact that
      $\partial H_-^0(e_2)$ and $\partial \Omega_n$ are sets of
      two-dimensional
      Lebesgue measure zero, we have  
	\begin{align}\label{eq:exploit_graph}
	  \begin{split}
            & \quad \int_{ H^0_-(e_2) \Delta A_{\lambda_n}
              R_{\nu_{n,i}(t)} (\Omega_n - \gamma_{n,i}(t)) }
            \left| z_2 \right| \frac{e^{-\alpha_n\sqrt{z_1^2 +
                  \frac{z_2^2}{\lambda_n^2} }}}{z_1^2 +
              \frac{z_2^2}{\lambda_n^2}} \intd^2 z \\
            & \geq \int_{g_{n,i,t}\left( S_{K,n,i,t}\right)}
            \int_0^{|h_{n,i,t}(g_{n,i,t}^{-1}(z_1))|} z_2
            \frac{e^{-\alpha_n\sqrt{z_1^2 + \frac{z_2^2}{\lambda_n^2}
                }}}{z_1^2 + \frac{z_2^2}{\lambda_n^2}} \intd z_2 \intd
            z_1.
	  \end{split}
	 \end{align}
	 
	 Due to definition \eqref{eq:defS2} and invertibility of
         $g_{n,i,t}$ on $S^3_{K,n,i,t}$, for all
         $z_1 \in g_{n,i,t}\left( S_{K,n,i,t}\right)$ we have the
         bound
         $|h_{n,i,t}(g_{n,i,t}^{-1}(z_1))| \leq
         \lambda_n^{\frac{1}{2}}$ for $n$ large enough. Therefore, for all
         $0 \leq z_2 \leq |h_{n,i,t}(g_{n,i,t}^{-1}(z_1))|$ we have by
         monotonicity
	\begin{align}\label{eq:kernel_below}
          \frac{e^{-\alpha_n\sqrt{z_1^2 + \frac{z_2^2}{\lambda_n^2}
          }}}{z_1^2 + \frac{z_2^2}{\lambda_n^2}} \geq
          \frac{e^{-\alpha_n \sqrt{z_1^2 + \frac{1}{\lambda_n}
          }}}{z_1^2 + \frac{1}{\lambda_n}}. 
	\end{align}
	Furthermore, we calculate using concavity of the square root
        and convexity of the exponential
	\begin{align}
          \label{eq:ker1}
	  \begin{split}
            0 & \leq \frac{e^{-\alpha_n |z_1|}}{z_1^2} -
            \frac{e^{-\alpha_n \sqrt{z_1^2 + \frac{1}{\lambda_n}
                }}}{z_1^2 + \frac{1}{\lambda_n}} \\
            & = \frac{e^{-\alpha_n |z_1|} }{z_1^2} \left(1 -
              e^{-\alpha_n |z_1| \left( \sqrt{1 + \frac{1}{\lambda_n
                      z_1^2}} - 1 \right)}\right)
            + \frac{ e^{-\alpha_n \sqrt{z_1^2 + \frac{1}{\lambda_n}
                }}}{\lambda_n z_1^2\left(z_1^2 +
                \frac{1}{\lambda_n}\right)}\\ 
            & \leq \frac{e^{-\alpha_n |z_1|} }{z_1^2} \left(1 -
              e^{-\frac{\alpha_n}{2\lambda_n |z_1|}}\right) + \frac{
              e^{-\frac{\alpha_n}{\sqrt{\lambda_n}}}}{\lambda_n
              z_1^2\left(z_1^2 + \frac{1}{\lambda_n}\right)}  \\
            & \leq \frac{\alpha_n e^{-\alpha_n |z_1|} }{2 \lambda_n
              |z_1|^3 }
            + \frac{ e^{-\frac{\alpha_n}{\sqrt{\lambda_n}}}}{ z_1^2} \\
            & \leq \frac{\alpha_n }{2 \lambda_n |z_1|^3 } + \frac{
              e^{-\frac{\alpha_n}{\sqrt{\lambda_n}}}}{ z_1^2}.
            \end{split}
           \end{align}
	Similarly, we have
	\begin{align}
          \label{eq:ker2}
          \frac{e^{- \frac{|z_1|}{\sqrt{2\pi}}}}{z_1^2}  -
          \frac{e^{-\alpha_n |z_1|}}{z_1^2} = \frac{e^{-
          \frac{|z_1|}{\sqrt{2\pi}}}}{z_1^2}  \left( 1 -
          e^{-\left(\alpha_n -
          \frac{1}{\sqrt{2\pi}}\right)|z_1|}\right) \leq
          \frac{\alpha_n - \frac{1}{\sqrt{2\pi}} }{|z_1|}. 
	\end{align}
	Combining \eqref{eq:ker1} and \eqref{eq:ker2}, for all
        $z_1 \in g_{n,i,t}\left( S_{K,n,i,t}\right)$ we have by
        definition \eqref{eq:defS1} that
          \begin{align}
            \frac{e^{- \frac{|z_1|}{\sqrt{2\pi}}}}{z_1^2} -
            \frac{e^{-\alpha_n \sqrt{z_1^2 + \frac{1}{\lambda_n}
            }}}{z_1^2 + \frac{1}{\lambda_n}}  \leq  a_n :=
            \frac{\alpha_n K^3  }{2 \lambda_n} 
            + K^2 e^{-\frac{\alpha_n}{\sqrt{\lambda_n}}} +K \left(
            \alpha_n - \frac{1}{\sqrt{2\pi}}\right), 
	\end{align}
	which vanishes in the limit $n \to \infty$.  Therefore,
        explicitly computing the remaining, trivial integral over
        $z_2$, we have
	\begin{align}\label{eq:int_x_2}
	  \begin{split}
            & \quad \int_{g_{n,i,t}\left( S_{K,n,i,t}\right)}
            \int_0^{|h_{n,i,t}(g_{n,i,t}^{-1}(z_1))|} z_2
            \frac{e^{-\alpha_n \sqrt{z_1^2 + \frac{z_2^2}{\lambda_n^2}
                }}}{z_1^2 + \frac{z_2^2}{\lambda_n^2}} \intd z_2 \intd
            z_1\\
            & \geq \frac{1}{2} \int_{g_{n,i,t}\left(
                S_{K,n,i,t}\right)} h^2_{n,i,t}(g_{n,i,t}^{-1}(z_1))
            \left( \frac{e^{-\frac{ |z_1|}{ \sqrt{2\pi} }}}{z_1^2} -
              a_n \right)
            \intd z_1\\
            & = \frac{1}{2} \int_{ S_{K,n,i,t}} h^2_{n,i,t}(s) \left(
              \frac{e^{-\frac{|g_{n,i,t}(s)|}{ \sqrt{2\pi}
                  }}}{g^2_{n,i,t}(s)}-  a_n  \right)
            |g_{n,i,t}'(s)|\intd s.
            \end{split}
	\end{align}

	Now, due to \eqref{eq:g_short} and estimate
        \eqref{eq:est_h}, we have
	\begin{align}\label{eq:x_2}
		\limsup_{n\to \infty} a_n \int_{0}^1 \int_{ S_{K,n,i,t}} 
		            h^2_{n,i,t}(s) |g_{n,i,t}'(s)|\intd s \intd t = 0.
	\end{align}
        Combining the bounds
        $|h_{n,i,t}(s)| \leq \lambda_n^\frac{1}{4}$ and
        $|g_{n,i,t}(s)| >K^{-1}$ for $s \in S_{K,n,i,t}$ with the
        convergence \eqref{eq:gprime}, we obtain 
	\begin{align}\label{eq:random_exponent_explained}
	  \begin{split}
            &\quad \limsup_{n \to \infty} \frac{1}{2} \int_0^1 \int_{
              S_{K,n,i,t}} h^2_{n,i,t}(s) \frac{e^{-\frac{
                  |g_{n,i,t}(s)| }{ \sqrt{2 \pi}}
              }}{g^2_{n,i,t}(s)} \left(1
              - |g_{n,i,t}'(s)| \right) \intd s \intd t\\
            & \leq \limsup_{n \to \infty} \frac{K^2}{2} \int_0^1
            \int_{ S_{K,n,i,t}} \lambda_n^\frac{1}{2} \left(1
              - g_{n,i,t}'(s) \right) \intd s \intd t\\
            & = 0.
	  \end{split}
	\end{align}
	Using $|g_{n,i,t}(s)| \leq |s|$ obtained by integrating the
        bound \eqref{eq:g_short} from $g_{n,i,t}(0) = 0$, we
        get \begin{align}\label{eq:throw_out_g} \frac{1}{2} \int_{
            S_{K,n,i,t}} h^2_{n,i,t}(s)
          \frac{e^{-\frac{|g_{n,i,t}(s)|}{ \sqrt{2\pi}
              }}}{g^2_{n,i,t}(s)} \intd s \geq \frac{1}{2} \int_{
            S_{K,n,i,t}} h^2_{n,i,t}(s) \frac{e^{-
              \frac{|s|}{\sqrt{2\pi}}} }{s^2}\intd s.
	\end{align}
	
	Let $\bar n \in \N$. For $n \in \N$ with $n \geq \bar n$, we
        have
        \begin{align}
          \tilde S_{K,\bar n, i,t}:= \bigcap_{n'= \bar n} ^\infty S_{K, n',i,t}
          \subset S_{K,n,i,t} \subset [-K,K].
        \end{align}
        Therefore,
        using Fubini's theorem we obtain
	\begin{align}
	  \begin{split}
            & \quad \liminf_{n \to \infty} \frac{1}{2} \int_0^1 \int_{
              S_{K,n,i,t}}
            h^2_{n,i,t}(s)  \frac{e^{- \frac{|s|}{\sqrt{2\pi}} }}{s^2} \intd s \intd t\\
            & \geq \liminf_{n \to \infty} \frac{1}{2} \int_0^1 \int_{
              \tilde S_{K,\bar n, i,t}}
            h^2_{n,i,t}(s)  \frac{e^{- \frac{|s|}{\sqrt{2\pi}}}}{s^2}\intd s \intd t\\
            & = \liminf_{n \to \infty} \frac{1}{2} \int_{-K}^K
            \frac{e^{- \frac{|s|}{\sqrt{2\pi}}}}{s^2} \int_0^1
            \chi_{\tilde S_{K,\bar n, i,t}} (s) h^2_{n,i,t}(s)\intd t
            \intd s.
           \end{split}
	\end{align}
	As
        $\chi_{\tilde S_{K,\bar n, i,t}} = \chi^2_{\tilde S_{K,\bar n,
            i,t}}$, we also get
	\begin{align}\label{eq:5.100}
	  \begin{split}
            & \quad \liminf_{n \to \infty} \frac{1}{2} \int_0^1 \int_{
              S_{K,n,i,t}}
            h^2_{n,i,t}(s)  \frac{e^{- \frac{|s|}{\sqrt{2\pi}} }}{s^2} \intd s \intd t\\
            & \geq \liminf_{n \to \infty} \frac{1}{2} \int_{-K}^K
            \frac{e^{- \frac{|s|}{\sqrt{2\pi}}}}{s^2} \int_0^1 \left(
              \chi_{\tilde S_{K,\bar n, i,t}} (s)
              h_{n,i,t}(s)\right)^2 \intd t \intd s.
           \end{split}
	\end{align}
	For all $s\in \R$, the convergence \eqref{eq:h_weak} then gives
	\begin{align}
          \chi_{\tilde S_{K,\bar n, i,t}} (s)
          h_{n,i,t}(s) \warr -  \chi_{\tilde S_{K,\bar n, i,t}} (s) \frac{s^2}{2}
          \kappa_{\infty,i}(t) 
         \end{align}
         in $L_t^2(0,1)$.  Fatou's Lemma and weak lower-semicontinuity
         of the $L^2$-norm imply
        	\begin{align}
	  \begin{split}
            & \quad \liminf_{n \to \infty} \frac{1}{2} \int_{-K}^K
            \frac{e^{- \frac{|s|}{\sqrt{2\pi}}}}{s^2} \int_0^1 \left(
              \chi_{\tilde S_{K,\bar n, i,t}} (s)
              h_{n,i,t}(s)\right)^2\intd t \intd s \\
            & \geq \frac{1}{2} \int_{-K}^K \frac{e^{-
                \frac{|s|}{\sqrt{2\pi}}}}{s^2} \liminf_{n \to \infty}
            \int_0^1 \left( \chi_{\tilde S_{K,\bar n, i,t}} (s)
              h_{n,i,t}(s)\right)^2\intd t \intd s\\
            & \geq \frac{1}{8} \int_{-K}^K e^{-
              \frac{|s|}{\sqrt{2\pi}}} s^2 \int_0^1 \chi_{\tilde
              S_{K,\bar n, i,t}} (s) \kappa_{\infty,i}^2(t) \intd t
            \intd s.
            \end{split}
	\end{align}
	Combining this with estimate \eqref{eq:5.100}, we get
	\begin{align}\label{eq:uniform_set}
	  \begin{split}
	  	& \quad \liminf_{n \to \infty} \frac{1}{2} \int_0^1 \int_{
              S_{K,n,i,t}}
            h^2_{n,i,t}(s)  \frac{e^{- \frac{|s|}{\sqrt{2\pi}} }}{s^2} \intd s \intd t \\
             & \geq  \frac{1}{8} \int_{-K}^K e^{-
              \frac{|s|}{\sqrt{2\pi}}} s^2 \int_0^1 \chi_{\tilde
              S_{K,\bar n, i,t}} (s) \kappa_{\infty,i}^2(t) \intd t
            \intd s.
          \end{split}
	\end{align}
	
	On the other hand, as the sets $\tilde S_{K,\bar n, i,t} $ are
        increasing in $\bar n$, we may take the supremum in $\bar n$
        on the right hand side of estimate \eqref{eq:uniform_set} and
        get after an application of Fubini's theorem and the monotone
        convergence theorem that 
	\begin{align}
	  \begin{split}
            \liminf_{n \to \infty} \frac{1}{2} \int_0^1 \int_{
              S_{K,n,i,t}} h^2_{n,i,t}(s) \frac{e^{-
                \frac{|s|}{\sqrt{2\pi}} }}{s^2} \intd s \intd t \geq
            \frac{1}{8}\int_0^1 \kappa_{\infty,i}^2(t)
            \int_{S_{K,\infty,i,t}} e^{- \frac{|s|}{\sqrt{2\pi}}} s^2
            \intd s \intd t.
            \end{split}
	\end{align}
	Equation \eqref{eq:S_total} then gives
	\begin{align}\label{eq:insert_S}
	  \begin{split}
            \liminf_{n \to \infty} \frac{1}{2} \int_0^1 \int_{
              S_{K,n,i,t}} h^2_{n,i,t}(s) \frac{e^{-
                \frac{|s|}{\sqrt{2\pi}} }}{s^2} \intd s \intd t \geq
            \frac{1}{4}\int_0^1 \kappa_{\infty,i}^2(t) \int_{K^{-1}}^K
            e^{- \frac{|s|}{\sqrt{2\pi}}} s^2 \intd s \intd t.
            \end{split}
	\end{align}
	
	Chaining together estimates \eqref{eq:exploit_graph},
        \eqref{eq:int_x_2}, \eqref{eq:x_2},
        \eqref{eq:random_exponent_explained}, \eqref{eq:throw_out_g},
        and \eqref{eq:insert_S} and noticing that the estimates
          apply for all $K \in \N$, we get 
	\begin{align}
	  \begin{split}
            & \quad \liminf_{n \to \infty} \int_0^1 \int_{ H^0_-(e_2)
              \Delta A_{\lambda_n} R_{\nu_{n,i}(t)} (\Omega_n -
              \gamma_{n,i}(t)) } \left| z_2 \right|
            \frac{e^{-\alpha_n\sqrt{z_1^2 + \frac{z_2^2}{\lambda_n^2}
                }}}{z_1^2 +
              \frac{z_2^2}{\lambda_n^2}} \intd^2 z \intd t\\
            & = \frac{1}{4} \sup_{K \in \N} \int_0^1 \kappa_{\infty,i}^2(t)
            \int_{K^{-1}}^K e^{- \frac{|s|}{\sqrt{2\pi}}} s^2 \intd s
            \intd t\\
            & = \frac{1}{4} \int_0^1 \kappa_{\infty,i}^2(t) \int_{0}^\infty e^{-
              \frac{|s|}{\sqrt{2\pi}}} s^2 \intd s
            \intd t\\
            & = 2^{\frac{1}{2}} \pi^\frac{3}{2} \int_0^1
            \kappa^2_{\infty,i}(s) \intd t,
	  \end{split}
	\end{align}
	which is what we claimed for this step of the proof.

        \textit{Step 6: Combine all limit curves.}  Note that since
        the boundary curves converge strongly in
        $H^1(\Sph^1; \R^2),$ we have
        $L(\gamma_{n,i}) \to L(\gamma_{\infty,i})$ as $n \to \infty$.
        Inserting this with $\sigma_n \to \sigma$,
        $\alpha_n \to \frac{1}{\sqrt{2\pi}}$ as $n\to \infty$, and the
        estimate \eqref{eq:lower_bound_one_curve} into the
        representation \eqref{eq:representation_revised} we get
	\begin{align}
	  \begin{split}
            \liminf_{\lambda \to \infty} \lambda_n^2
            F_{\lambda,\alpha}(\Omega_\lambda)\geq \sum_{i\in I}
            L(\gamma_{\infty,i}) \left( \sigma + \frac{ \pi}{2}
              \int_0^1 \kappa^2_{\infty,i}(t) \intd t \right) = \hat
            F_{\infty,\sigma} (\gamma),
	  \end{split}
	\end{align}
	concluding the proof.
\end{proof}

\subsection{Concluding arguments}

Finally, the proofs of Proposition \ref{prop:expansion}, Theorem
\ref{thm:main} and Corollary \ref{cor:convergence_minimizers}
essentially consist of putting together all the information at hand.
Additionally, we provide the proof for Proposition
\ref{prop:centered}.

\begin{proof}[Proof of Proposition \ref{prop:expansion}]
  The upper bound is given by Lemma \ref{lem:expansion}, while the
  lower bound follows by applying Proposition \ref{prop:lower_bound}
  to the constant sequence $n \mapsto \Omega$ for $n \in \N$.
\end{proof}

\begin{proof}[Proof of Theorem \ref{thm:main}]
  By the fact that $F_{\infty,\sigma}$ is the $L^1$-relaxation of the
  functional in \eqref{elastica_smooth}, the upper bound follows
  immediately from Proposition \ref{prop:expansion}.
	
  To prove the lower bound let $\Omega_n \in \mathcal{A}_\pi$ for
  $n \in \N$ be sets such that
  $\chi_{\Omega_n} \to \chi_{\Omega_\infty}$ for
  $\Omega_\infty \in \mathcal{A}_\pi$ and such that
	\begin{align}
		\liminf_{n\to \infty} \lambda_n^2 F_{\lambda_n,\alpha_n}(\Omega_n) <\infty.
	\end{align}
	By a standard approximation argument, we may suppose the set
        $\Omega_n$ to be regular for all $n \in \N$.  Combining
        Theorem \ref{thm:relaxation} with Propositions
        \ref{prop:compactness} and \ref{prop:lower_bound}, we get a
        system of curves
        $\left\{\gamma_{\infty,i}\right\}_{i\in I} \in
        G(\Omega_\infty)$ such that
	\begin{align}
          F_{\infty, \sigma}(\Omega_\infty) \leq \hat
          F_{\infty,\sigma}(\gamma_\infty) \leq \liminf_{n\to \infty}
          \lambda_n^2 F_{\lambda_n,\alpha_n}(\Omega_n), 
	\end{align}
	concluding the proof.
\end{proof}

\begin{proof}[Proof of Corollary \ref{cor:convergence_minimizers}]
  Existence of minimizers follows from Proposition
  \ref{lem:minimizers}.  As the minimizers are connected, after a
  suitable translation and by the bound on the perimeter there exists
  $\Omega_\infty \in \mathcal{A}_\pi$ and a subsequence (not
  relabeled) such that $\chi_{\Omega_n} \to \chi_{\Omega_\infty}$ in
  $L^1(\R^2)$.  We then have by the lower bound that
	\begin{align}
		\inf_{\mathcal{A}_\pi} F_{\infty,\sigma} \leq
          F_{\infty,\sigma}(\Omega_\infty)  \leq \liminf_{n \to
          \infty} \lambda_n^2 F_{\lambda_n, \alpha_n}(\Omega_n) =
          \liminf_{n \to \infty} \inf_{\mathcal{A}_\pi} \lambda_n^2
          F_{\lambda_n, \alpha_n}, 
	\end{align}
	while the upper bound implies
	\begin{align}
          \liminf_{n \to \infty} \inf_{\mathcal{A}_\pi} \lambda_n^2
          F_{\lambda_n, \alpha_n} \leq \inf_{\mathcal{A}_\pi}
          F_{\infty,\sigma}. 
	\end{align}
	Therefore, we have equality everywhere and $\Omega_\infty$ is
        a minimizer of $F_{\infty,\sigma}$.  The characterization of
        minimizers in the limit was carried out by Goldman, Novaga,
        and R{\"o}ger \cite{goldman2020quantitative}.
\end{proof}

\begin{proof}[Proof of Proposition \ref{prop:centered}]
  	The alternative representation of $f$ immediately follows from
	\begin{align}
          \int_{\Omega_x} \int_{\R^2}  K( |y-z|) \intd^2 y \intd^2 z =
          2\pi^2 \int_0^\infty rK(r) \intd r. 
	\end{align}
  	Setting $R := \sqrt{1 + r^2}$, we calculate
	\begin{align}
	  \begin{split}
            f(x) & = \int_{B_R(0)} \int_{B_R(0)} K( |y-z|) \intd^2 y
            \intd^2 z + \int_{B_r(x)} \int_{B_r(x)} K(
            |y-z|) \intd^2 y \intd^2 z\\
            & \qquad - 2 \int_{B_R(0)} \int_{B_r(x)} K( |y-z|)
            \intd^2 y \intd^2 z\\
            & = \int_{B_R(0)} \int_{B_R(0)} K( |y-z|) \intd^2 y
            \intd^2 z + \int_{B_r(0)} \int_{B_r(0)} K(
            |y-z|) \intd^2 y \intd^2 z\\
            & \qquad - 2 \int_{B_R(0)} \int_{B_r(x)} K( |y-z|) \intd^2
            y \intd^2 z.
	  \end{split}
	\end{align}
	By the Riesz rearrangement inequality applied to the last
        term, see \cite[Lemma 3]{lieb1977existence} for a sharp
        version of the inequality, we get $f(x) \geq f(0)$.  If $K$ is
        additionally strictly monotone decreasing, the sharp version
        implies that $x= 0$ is the only minimizer.
  \end{proof}

\appendix

\section{Model derivation}
\label{sec:model-derivation}

For the sake of the reader's convenience, below we present a first
principles derivation of the energy in \eqref{energy} along the lines
of Ref.\ \cite{andelman87}, except that we take the sharp interface
approach and model the Langmuir layer as an incompressible
two-dimensional patch of amphiphilic molecules. For a fixed patch area
there is, therefore, no non-trivial local contribution to the energy
from the interior of the patch and all the local interactions due to
van der Waals forces can be captured by an interfacial energy term
representing line tension. We also focus on a regime that takes
advantage of the large dielectric constant of water at moderate
droplet sizes (see further discussion at the end of this section). For
the clarity of the derivation, in this section we adhere to the
standard physics notations.

Consider a monolayer of amphiphilic molecules at the air-water
interface located at the $z = 0$ plane in $\R^3$, with water occupying
the $z < 0$ half-plane. The molecules are restricted to a set
$\Omega \subset \R^2$ with fixed area $|\Omega|$ in the
$xy$-plane. The excess energy associated with this monolayer patch may
be written as
\begin{align}
  \label{eq:E12}
  \mathcal E(\Omega) = \mathcal E_\mathrm{surf}(\Omega) + \mathcal
  E_\text{long-range}(\Omega), 
\end{align}
where the first term is the surface energy of the patch:
\begin{align}
  \label{eq:Esurf}
  \mathcal E_\text{surf} = \gamma P(\Omega), 
\end{align}
with $P(\Omega)$ being the perimeter of the set $\Omega$, coinciding
with the one-dimensional Hausdorff measure
$\mathcal H^1(\partial \Omega)$ of the boundary of $\Omega$ for
sufficiently regular sets \cite{ambrosio}, and $\gamma > 0$ being
the line tension. The long-range part
$\mathcal E_\text{long-range}(\Omega)$ of the energy is due to the
electrostatic interaction of the charged polar heads of the
amphiphilic molecules immersed in water:
\begin{align}
  \label{eq:Elr}
  \mathcal E_\text{long-range}(\Omega) = \frac12 \int_\Omega q \rho
  \bar U d^2 r, 
\end{align}
where $-q$ is the charge taken away from the amphiphilic molecule's
polar head by water, $\rho$ is the areal density of the amphiphilic
molecules, and $\bar U$ is the electrostatic potential at $z = 0$. The
latter may be found with the help of the Debye-H\"uckel theory by
solving for the potential $U$ in the whole space (in the SI units)
\cite{andelman87}:
\begin{align}
  \Delta U - \kappa^2 U = 0, & \qquad z < 0. \label{DUm} \\
  \Delta U = 0, & \qquad z > 0,
  \end{align}
  subject to the conditions at the air-water interface:
\begin{align}
  \lim_{z \to 0^-} U(\cdot, z) &= \lim_{z \to 0^+} U(\cdot, z),\\
  \epsilon_d \lim_{z \to 0^-} U_z(\cdot, z) - \lim_{z \to 0^+}
  U_z(\cdot, z) &= {q \over \epsilon_0} \rho \chi_\Omega, \label{Uzpm}
\end{align}
with $U$ vanishing at infinity.  Here $\kappa$ is the Debye-H\"uckel
screening parameter equal to the inverse of the screening length in
water, $\epsilon_d$ is water's dielectric constant, $\epsilon_0$ is
the vacuum permittivity, and $\chi_\Omega$ is the characteristic
function of $\Omega$.

For a given bounded set $\Omega$, this elliptic problem has a unique
solution, which can be found by means of the Fourier transform with
respect to the in-plane variables. Denoting
\begin{align}
  \widehat U_\mathbf{k}(z) := \int_{\R^2} e^{i \mathbf k \cdot \mathbf r}
  U(\mathbf r, z) \, d^2 r \qquad \mathbf k \in \R^2,
\end{align}
and passing to the Fourier space in \eqref{DUm}--\eqref{Uzpm}, after
some simple algebra we obtain \cite{andelman87}
\begin{align}
  \widehat U_\mathbf{k}(z) = {q \rho e^{z \sqrt{\kappa^2 + |\mathbf
  k|^2}}  \widehat \chi_\Omega(\mathbf k) \over
  \epsilon_0 (\epsilon_d 
  \sqrt{\kappa^2 + |\mathbf k|^2} + |\mathbf k|)}, & \qquad z < 0,
  \\
  \widehat U_\mathbf{k}(z) = {q \rho e^{- z |\mathbf k|}  \widehat
  \chi_\Omega(\mathbf k)\over 
  \epsilon_0 (\epsilon_d \sqrt{\kappa^2 + |\mathbf k|^2} + |\mathbf
  k|)}, & \qquad z > 0, 
\end{align}
where $\widehat \chi_\Omega(\mathbf k)$ is the Fourier transform of
$\chi_\Omega$.  Notice that since $\epsilon_d \simeq 80$ is very large
for water, with a very good accuracy one could neglect the
$|\mathbf k|$ term compared to
$\epsilon_d \sqrt{\kappa^2 + |\mathbf k|^2}$ in the expression for
$\widehat U_\mathbf{k}(0)$. Thus, we have
\begin{align}
  \widehat U_\mathbf{k}(0) \simeq  {q \rho \widehat
  \chi_\Omega(\mathbf k)\over 
  \epsilon_0 \epsilon_d \sqrt{\kappa^2 + |\mathbf k|^2}},
\end{align}
and returning to the real space, we get
\begin{align}
  \label{eq:Ubar}
  \bar U(\mathbf r) \simeq {q \rho \over 2 \pi \epsilon_0 \epsilon_d} 
  \int_\Omega {e^{-\kappa |\mathbf r - \mathbf r'|} \over |\mathbf r
  -  \mathbf r'|} \, d^2 r'
\end{align}
at the air-water interface. Thus, the non-local part of the energy is,
to the leading order in $\epsilon_d \gg 1$:
\begin{align}
  \mathcal E_\text{long-range} (\Omega) = {q^2 \rho^2 \over 4 \pi
  \epsilon_0 
  \epsilon_d} \int_\Omega \int_\Omega  {e^{-\kappa |\mathbf r -
  \mathbf r'|} \over |\mathbf r -  \mathbf r'|} \, d^2 r \, d^2 r'. 
\end{align}

We now carry out a non-dimensionalization, introducing
\begin{align}
  \label{eq:EOm}
  E(\Omega) := P(\Omega) + {1 \over 4 \pi} \int_\Omega \int_\Omega
  {e^{-\alpha |\mathbf r - \mathbf r'|} \over |\mathbf r - \mathbf r'|}
  \, d^2 r \, d^2 r',
\end{align}
and noting that $\mathcal E(\ell \Omega) = \gamma \ell E (\Omega)$
with the choices of the scale and the dimensionless screening
parameter, respectively:
\begin{align}
  \ell = {\sqrt{\epsilon_0 \epsilon_d \gamma} \over q \rho}, \qquad
  \alpha = {\kappa \sqrt{\epsilon_0 \epsilon_d \gamma} \over q \rho}. 
\end{align}
Taking into account that
$\int_{\R^2} r^{-1} e^{-\alpha r} d^2 r =2 \pi / \alpha$, we can then
rewrite the energy $E(\Omega)$ as
\begin{align}
  E(\Omega) = E_\alpha(\Omega) + {|\Omega| \over 2 \alpha}, 
\end{align}
and so, up to an additive constant the energy $E(\Omega)$ coincides
with that in \eqref{energy}.

We note that the kernel appearing in \eqref{eq:Ubar} exhibits
exponential decay due to the fact that we neglected the $|\mathbf k|$
term in the Fourier transform of $\bar U$ for large
$\epsilon_d$. This, however, becomes invalid for arbitrarily large
separations, for which the kernel can be shown to exhibit an algebraic
decay of the form
$q^2 \rho^2 /(2 \pi \epsilon_0 \epsilon_d^2 \kappa^2 |\mathbf r -
\mathbf r'|^3)$, up to an additive constant. Therefore, in agreement
with the conventional wisdom the limit of large droplets should be
described by the model in which the long-range part of the energy is
of dipolar type \cite{mcconnell88,andelman85}. This model corresponds
to the case of strong ionic solutions and was first studied rigorously
in Ref.\ \cite{muratov2018nonlocal}. Nevertheless, for
$\epsilon_d \gg 1$ the model in \eqref{eq:EOm} is appropriate in a
certain range of droplet sizes, which corresponds to the case of weak
ionic solutions \cite{andelman87}.

\printbibliography[heading=bibintoc]
\end{document}